\documentclass[11pt]{amsart}

\usepackage{preamble}

\usepackage[margin=1.4in]{geometry}

\newcommand{\GL}{\operatorname{GL}}
\newcommand{\Or}{\operatorname{O}}

\begin{document}

\title[Tensors of $OSp$]{Khovanov algebras of type B and tensor powers of the natural $\mathrm{OSp}$-representation}
\author{{\rm Thorsten Heidersdorf and Jonas Nehme and Catharina Stroppel}}

\curraddr{T.H.: School of Mathematics, Statistics and Physics, Newcastle University}
\address{T. H.: Mathematisches Institut, Universit\"at Bonn}
\email{heidersdorf.thorsten@gmail.com} 
\address{J. N.: Max-Planck-Institut für Mathematik, Bonn }
\email{nehme@mpim-bonn.mpg.de}
\address{C. S.: Mathematisches Institut, Universit\"at Bonn}
\email{stroppel@math.uni-bonn.de}

\date{}

\begin{abstract} We develop the theory of projective endofunctors for modules of Khovanov algebras $K$ of type B. In particular we compute the composition factors and the graded layers of the image of a simple module under such a projective functor. We then study variants of such functors for a subquotient $e\tilde{K}e$. Via a comparison of two graded lifts of the Brauer algebra we relate the Khovanov algebra to the Brauer algebra and use this to show that projective functors describe translation functors on representations of the orthosymplectic supergroup $\mathrm{OSp}(r|2n)$. As an application we get a description of the Loewy layers of indecomposable summands in tensor powers of the natural representation of $\mathrm{OSp}(r|2n)$.
\end{abstract}



\maketitle

\setcounter{tocdepth}{1}
\tableofcontents

\section{Introduction}

\subsection{Decomposition of tensor powers} The decompositon of tensor powers $V^{\otimes d}$ of the standard representation $V = \mathbb{C}^n$ of $\GL(n)$ can be understood via Schur--Weyl duality. Passing to the orthogonal group $\Or(n)$ means replacing the group ring $\mathbb{C}[S_d]$ by the Brauer algebra $\Br_d(n)$. In particular, the description of the primitive idempotents in $\Br_d(n)$ can be used to dcompose $V^{\otimes d}$ into explicitely given irreducible $\Or(n)$-modules.

Remarkably, the statement of Schur--Weyl duality carries over almost unchanged to tensor powers $V^{\otimes d}$ of the standard representation $V = \mathbb{C}^{m|n}$ of the general linear supergroup $\mathrm{GL}(m|n)$. However, when we replace $\Or(n)$ by its super analog, the orthosymplectic supergroup $\mathrm{OSp}(r|2n)$ ($r =2m$ or $2m+1$), the decomposition $V^{\otimes d}$ is no longer completely reducible in general. The problem of describing the decomposition of $V^{\otimes d}$ into indecomposable summands has a rich history. A first step was completed by Benkart, Ram and Shader \cite{BSR98} who in particular constructed highest weight vectors in such tensor powers.

An important ingredient in solving this problem are the Deligne interpolating categories $\Rep_{\delta}$, $\delta \in \mathbb{C}$ \cite{D19}. The indecomposable objects in $\Rep_{\delta}$ are parametrized by the set of all partitions, and we denote the corresponding indecomposable object by $\R_{\delta}(\lambda)$. As Deligne showed, these categories admit symmetric monoidal functors $\mathbb{F}=\mathbb{F}_{(r|2n)}\colon\Rep_{\delta=r-2n}\to \Osp[r][2n]\text{-mod}$ which send the tensor generator $\R_{\delta}((1))$ attached to the partition $(1)$ to $V$, and therefore $\mathbb{F}_{r|2n} (\R_{\delta}((1))^{\otimes d}) = V^{\otimes d}$.

This viewpoint was used by Lehrer--Zhang \cite{LZ17} to establish a fundamental result about $\End(V^{\otimes d})$: They proved that the functor $\mathbb{F}_{r|2n}$ is full which amounts to say that the Brauer algebras surject onto $\End(V^{\otimes d})$. A similar categorical approach had been previously used by Comes--Wilson \cite{CW} to prove a corresponding statement for mixed tensor powers of $\GL(m|n)$. Here the walled Brauer algebra $\mathrm{B}_{r|s}(m-n)$ surjects onto $\End(V^{\otimes r }\otimes (V^{\vee})^{\otimes s} )$ for the standard representation $V= \mathbb{C}^{m|n}$.

The results of Lehrer--Zhang were used by Comes--Heidersdorf \cite{CH}. They classified the indecomposable representations in $\R_{\delta}((1))^{\otimes d}$ in $\Rep_{\delta}$ for any $\delta$ and deduced from the fullness of $\mathbb{F}_{r|2n}$ a description of the kernel of $\mathbb{F}_{r|2n}$ and subsequently a classification of the indecomposable summands in $V^{\otimes d}$ for any $d \in \mathbb{N}$. They moreover gave an explicit tensor product decomposition formula for $\R_{\delta}(\lambda) \otimes \R_{\delta}(\mu)$ and characterized the projective summands in $V^{\otimes d}$.

The results in \cite{CH} give however no method to describe the composition factors or Loewy layers of the indecomposable representation $\mathbb{F}_{r|2n}(\R_{\delta}(\lambda))$ nor do they allow to determine which representations of $\mathrm{OSp}(r|2n)$ are of the form $\mathbb{F}_{r|2n}(\R_{\delta}(\lambda))$ for some partition $\lambda$. We follow in this article the methods of Brundan--Stroppel \cite{BS4}\cite{BS5} (that dealt with mixed tensor powers of $\GL(m|n)$) to completely solve this problem. The essential input is that we can model $V^{\otimes d}$ and its decomposition into indecomposable summands in the world of modules for Khovanov's arc algebra of type B. The $\mathrm{OSp}$-case has however a number of major added difficulties compared to the $\GL(m|n)$-case.

It is important to work with representations of the supergroup $\Osp[r][2n]$ here instead of the connected supergroup $\Sosp[r][2n]$ (or, equivalently, with representations of the Lie superalgebra $\osp[r][2n]$). While there is a symmetric monoidal functor from the Deligne category to $\Sosp[r][2n]\text{-mod}$, it is not full. There is also no diagrammatic description via Khovanov algebras in the $\Sosp[r][2n]$ or $\osp[r][2n]$-case.

\subsection{Khovanov's arc algebra}

Khovanov algebras arise naturally in the representation theory of supergroups as follows. Let $G$ be a quasi-reductive supergroup. The category $\mathcal{F}$ of finite-dimensional algebraic representations decomposes into blocks $\Gamma$. Let $P_{\Gamma}$ be a minimal projective generator of a block, i.e \[ P_{\Gamma} = \bigoplus_{\lambda \in \Gamma} P(\lambda) \] where $P(\lambda)$ is the projective cover of the irreducible representation $L(\lambda) \in \Gamma$. By Morita theory, the category of finite dimensional left modules over the locally finite endomorphism algebra \[ \Mod_{lf}(P_{\Gamma}) = \bigoplus_{\lambda,\mu \in \Gamma} \Hom(P(\lambda),P(\mu) )\] is equivalent to $\Gamma$. 
For $GL(m|n)$ Brundan--Stroppel \cite{BS4} gave a diagrammatic description of $\Mod_{lf}(P_{\Gamma})$: They constructed a diagrammatically defined algebra $K_{\Gamma}$ (a Khovanov algebra of type $A$) which is isomorphic to $\Mod_{lf}(P_{\Gamma})$. Summing over all blocks yields an equivalence of abelian categories \[ K\text{-mod} \cong \GL(m|n)\text{-mod}\] between finite dimensional left $K$-modules and finite dimensional representations of $\GL(m|n)$.

The type B analog of the Khovanov algebra (again denoted $K$) was studied in \cite{ES2}\cite{ES3}. It was shown by Ehrig--Stroppel \cite{ES2} how to relate it to blocks of the orthosymplectic supergroup $\mathrm{OSp}(r|2n)$. Apart from added combinatorial difficulties, there is a substantial difference: The Khovanov arc algebra of type $B$ is not isomorphic to $\Mod_{lf}(P_{\Gamma})$. In \cite{OSPII}*{Theorem 10.5}, Ehrig and Stroppel proved that a subquotient (here called $e\tilde{K}e$) of the Khovanov algebra of type $B$ is in fact isomorphic to the locally finite endomorphism algebra of a projective generator for $\Osp[r][2n]$ and thus gives rise to an equivalence 
\begin{equation}\label{equiv}
	\Psi\colon(e\tilde{K}e)\text{-mod}\to \Osp[r][2n]\text{-mod}
\end{equation}
of categories between the finite dimensional representations of $\Osp[r][2n]$ and finite dimensional $e\tilde{K}e$-modules. Both algebras $K$ and $e\tilde{K}e$ can be endowed with a nonnegative grading and this actually induces a grading on $\Osp[r][2n]$-mod. 

Note that it is not apparent at all yet that all this is useful to analyze $V^{\otimes d}$ since the above equivalences are not monoidal.

\subsection{Translation functors and projective functors}

The crucial idea is to look at $V^{\otimes d}$ as the image of the trivial representation $\mathbf{1}$ under a series of translation functors $\theta_i$. First results about translation functors in the $\osp[r][2n]$-case were obtained by \cite{GS10}\cite{GS13}. Summing over all blocks gives an isomorphism of endofunctors \[ \bigoplus_i \theta_i \cong - \otimes V.\]

The key point is that we can model the effect of translation functors $\theta_i$ in the world of $e\tilde{K}e$-modules even though $\Psi$ is not a monoidal equivalence. The endofunctors in $\Mod_{lf}(e\tilde{K}e)$ that will eventually correspond to translation functors  on $\mathcal{F}$ are called projective functors. For $\GL(m|n)$ the theory of projective functors was developed in \cite{BS2}\cite{BS4}\cite{BS5}, for $\Osp[r][2n]$ some partial results were obtained in \cite{ES18} \cite{HNS}.

In Section \ref{sec:Proj} we describe the effect of projective functors on irreducible $K$-modules. Here many of the proofs from the type A case  \cite{BS2} carry over almost verbatim (with some notable exceptions like Lemma \ref{nodegbilinformadjunction}). In Section \ref{sec:nuclear} we are then discussing the analogous theory for the subquotient $e\tilde{K}e$ for which we can mostly not rely on previous work in type A. Our discussion culminates in the following theorem (Theorem \ref{projfunctorsonsimplenuclear}) about the image of an irreducible $e\tilde{K}e$-module $\overline{L}(\gamma)$ under a projective functor (the terminology is explained in Sections \ref{sec:Proj}--\ref{sec:nuclear}). Part \ref{projfunctorsonsimpleinuclear} describes the composition factors in the grading filtration and part \ref{projfunctorsonsimpleexplicitnuclear} gives structural information about the module.

\begin{thmintro}
	\label{thm:intro-proj}
	Suppose we are given a proper $\Lambda\Gamma$-matching $t$ and $\gamma\in\Gamma$. Then
	\begin{enumerate}
		\item\label{projfunctorsonsimpleinuclear} in the graded Grothendieck group of $\Mod_{lf}(e\tilde{K}_\Lambda e)$
		\begin{equation*}
			[\tilde{G}^t_{\Lambda\Gamma}\overline{L}(\gamma)] = \sum_{\mu}(q+q^{-1})^{n_\mu}[\overline{L}(\mu)],
		\end{equation*}
		where $n_\mu$ denotes the number of lower circles in $\underline{\mu}t$ and we sum over all $\mu\in\Lambda$ such that 
		\begin{enumerate}[label=\normalfont{(\alph*)}]
			\item\label{projfunctorsonsimpleianuclear} $\underline{\gamma}$ is the lower reduction of $\underline{\mu}t$,
			\item\label{projfunctorsonsimpleibnuclear} there exists no lower line in $\underline{\mu}t$,
		\end{enumerate}
		\item\label{projfunctorsonsimplenonzeronuclear} the module $\tilde{G}^t_{\Lambda\Gamma}\overline{L}(\gamma)$ is nonzero if and only if all cups of $t\gamma$ are anticlockwise oriented and
		\item\label{projfunctorsonsimpleexplicitnuclear} under the assumptions of \cref{projfunctorsonsimplenonzeronuclear} define $\lambda\in\Lambda$ such that $\overline{\lambda}$ is the upper reduction of $t\overline{\gamma}$ or alternatively $\lambda t\gamma$ is oriented and every cup and cap is oriented anticlockwise. In this case $\tilde{G}^t_{\Lambda\Gamma}\overline{L}(\gamma)$ is a self-dual indecomposable module with irreducible head $\overline{L}(\lambda)\langle -\ca(t)\rangle$.
	\end{enumerate}
\end{thmintro}
%

%

We then relate two types of endofunctors. On the one hand we consider the decomposition of the endofunctor $\-\otimes V=\bigoplus_{i\in\bbZ}\theta_i$ of $\mathcal{F}$, and on the other hand   we define the functors $\tilde{\Theta}_i\colon e\tilde{K}e\text{-mod}\to e\tilde{K}e\text{-mod}$, which are of the form $\tilde{G}^t_{\Lambda\Gamma}$ for certain choices of $\Lambda,\Gamma$. 

\begin{thmintro} \label{thm:intro-commute}
	We have an equivalence of categories $\Psi\colon(e\tilde{K}e)\text{-mod}\to \Osp[r][2n]\text{-mod}$ such that $\theta_i\circ\Psi\cong\Psi\circ\tilde{\Theta}_i$.
\end{thmintro}
 
The proof is rather involved and involves a comparison of two different graded versions of the Brauer algebra. The basic idea is to consider the analog of $V^{\otimes d}$ on the Khovanov-side $e\tilde{K}e$ which we call \[ T_d\coloneqq \bigoplus_{\mathbf{i}\in(\mathbb{Z}+\frac{\delta+1}{2})^d}\Theta_{\mathbf{i}}\overline{L}(\emptyset_\delta).\] The key theorem \ref{blackboxformainthm} shows that there exists an isomorphism of algebras $\xi_d\colon\Br_d(\delta)\overset{\cong}{\to}\End_K(T_d)$ which intertwines $i$-induction (defined in Definition \ref{iinddef}) and $\Theta_i$. Once this theorem is proven, we can use the known surjection $\Br_d(\delta) \to \End_{\mathcal{F}}(V^{\otimes d})$ to relate endomorphism spaces for $e\tilde{K}e$ with those in $\mathcal{F}$. The difficulty in proving Theorem \ref{blackboxformainthm} is that the idempotents picking out the eigenspaces for the $i$-induction are not part of the definition of the Brauer algebra and very hard to handle. We replace the Brauer algebra with two different graded lifts -- one, $G_d(\delta)$, due to Ge Li \cite{GELI}, has these idempotents build in the definition; the other algebra $C_d(\delta)$ \cite{OSPII}*{Section 11} \cite{diss}*{Section 4} can be easily identified with $\End_K(T_d)$. The isomorphism $\xi_d\colon\Br_d(\delta)\overset{\cong}{\to}\End_K(T_d)$ has been shown in \cite{diss} and it swaps $i$-induction and $\Theta_i$ by construction.
 
Analogous results following the methods of \cite{BS2} \cite{BS5} and the present paper have been obtained recently in the $\mathfrak{p}(n)$-case \cite{nehme2023khovanov} as well.



\subsection{Direct summands and representations of the form $\mathbb{F}\R_{\delta}(\lambda)$}

Using Theorems \ref{thm:intro-proj} and \ref{thm:intro-commute} we are now able to analyze indecomposable summands of $V^{\otimes d}$ via repeatedly tensoring the irreducible $K$-module corresponding to the trivial representation with geometric bimodules. More precisely, every indecomposable summand $\mathbb{F}\R_{\delta}(\lambda)$ is of the form $\Psi(\tilde{G}^t_{\Lambda\Gamma}\overline{L}((\emptyset_\delta)^\owedge_+))$ for some blocks $\Lambda$ and $\Gamma$ in $e\tilde{K}e$ and some $\Lambda \Gamma$-matching $t$. Conversely every such choice of $\Lambda$, $\Gamma$ and $t$ gives in this way an indecomposable summand in some $V^{\otimes d}$.

The corresponding questions for $GL(m|n)$ were studied in \cite{H} based on \cite{BS5}. Our description of $\tilde{G}^t_{\Lambda\Gamma}\overline{L}(\gamma)$ in Theorem \ref{thm:intro-proj} translates to the following corollary. 

\begin{corintro} The modules $\mathbb{F}\R_\delta(\lambda)$ are self-dual with simple head and socle. Their Loewy length is given by $2d(\lambda)+1$, where $d(\lambda)$ denotes the number of caps in the cap diagram of the weight diagram associated to $\lambda$ and their grading filtration agrees with its radical and its socle filtration.
\end{corintro}

We remark that this applies in particular to all projective covers $P(\lambda)$. Of course Theorem \ref{thm:intro-proj} is much stronger than what is listed in the corollary as it provides a diagrammatic description of all socle layers. It also gives a complete description of the action of translation functors on irreducible modules and their projective covers, substantially improving the results of \cite{GS10}\cite{GS13}.

Given this, we investigate the question which irreducible $\Osp[r][2n]$-modules appear as direct summands in $V^{\otimes d}$. We will look at this question from two different angles. First we give in \cref{mudeltairred} different characterizations, when an indecomposable summand $\mathbb{F}\R_\delta(\lambda)$ is irreducible and after that we try to classify the irreducibles $L(\lambda, \eps)$ appearing as a direct summand in \cref{irredasindecsummand}.

For this we will recall the notion of a Kostant module in \cref{kostandkazhdan} and  will also revisit Kazhdan--Lusztig polynomials in our setting.

\begin{corintro} The following statements are equivalent for an indecomposable direct summand $\mathbb{F}\R_\delta(\lambda)$ in $V^{\otimes d}$ associated to a partition $\lambda$ due to \cref{mudeltairred}.
\begin{itemize}
	\setlength\itemsep{0em}
	\item $\mathbb{F}\R_\delta(\lambda)$ is irreducible.
	\item $\lambda$ is a Kostant-Deligne weight.
	\item The Kazhdan--Lusztig polynomials $p_{\mu, \lambda}(q)$ are monomials for all $\mu\leq\lambda$.
	\item The weight diagram associated to $\lambda$ is $\vee\wedge$- and $\wedge\wedge$-avoiding.
	\item The cap diagram associated to $\lambda$ is cap-free.
	\item The weight diagram associated to $\lambda$ is maximal in the Bruhat order.
\end{itemize}
\end{corintro}

In order to classify the irreducible $\Osp[r][2n]$-modules, which appear as a direct summand in some $V^{\otimes d}$, we introduce an automorphism of order $2$ on the category of finite dimensional $\Osp[r][2n]$-modules. It is defined via some manipulation on the Khovanov algebra side and e.g.~interchanges the trivial with the natural representation for $r=2n+1$. It maps $L(\lambda, \eps)$ to $L(\lambda^\Box, -\eps)$ for some combinatorially defined weight $\lambda^\Box$.

\begin{corintro} The following statements are equivalent.
\begin{itemize}
	\setlength\itemsep{0em}
	\item $L(\lambda, \eps)$ is a direct summand of some $V^{\otimes d}$ for some $\eps$.
	\item $\lambda$ or $\lambda^{\Box}$ is a Kostant weight in the sense of \cite{GH}.
\end{itemize}
 And if $r$ is odd or $\mathrm{at}(\lambda)>1$ these are equivalent to
\begin{itemize}
 	\setlength\itemsep{0em}
	\item $L(\lambda)$ or $L(\lambda^{\Box})$ satisfies the Kac--Wakimoto conditions (considered as $\osp[r][2n]$-modules).
\end{itemize}
\end{corintro}

Since \cite{CH} established tensor product decomposition laws for $\mathbb{F}\R_\delta(\lambda) \otimes \mathbb{F}\R_\delta(\mu)$, a direct consequence of these corollaries are such decomposition laws for tensor products between any two Kostant or projective modules.

%
%


\section{The orthosymplectic supergroup} \label{SecL:OSp} The ground field is always assumed to be $\mathbb{C}$. 

\begin{defi}\label{defosp}
	Let $V$ be a vector superspace and $\beta\colon V\times V\to\bbC$ a nondegenerate supersymmetric bilinear form (i.e.~a nondegenerate bilinear form that is symmetric on $V_0$, skewsymmetric on $V_1$ and $0$ on mixed products). Then $\lie{osp}(V)$ is the Lie subsuperalgebra of $\lie{gl}(V)$ given by \[\lie{osp}(V)_i:=\{x\in\lie{gl}(V)_i\mid\beta(x(a), b)=-(-1)^{\abs{x}\abs{a}}\beta(a, x(b))\text{ for all }a,b\in V\}.\]
	If $V=\bbC^{r|2n}$ we write $\osp[r][2n]$ for $\lie{osp}(V)$.
\end{defi}
%

We fix the usual standard basis $\{ \eps_i\}_{1\leq i} \cap \{\delta_j\}_{1\leq j\leq n}$ of the dual $\lie{h}^*$ of the Cartan algebra of diagonal matrices. We will denote by \begin{equation*}
	X(\osp[r][2n])\coloneqq\bigoplus_{i=1}^m\bbZ\eps_i\oplus\bigoplus_{j=1}^n\bbZ\delta_j
\end{equation*} the integral weight lattice. When referring to a weight we will always mean an \emph{integral weight}, i.e.~an element of $X(\osp[r][2n])$.
The parity shift $\Pi$ gives rise to a decomposition of $\osp[r][2n]\-\Mod=\mathcal{F}'\oplus\Pi\mathcal{F}'$, where $\mathcal{F}'$ contains all objects such that the parity of the weight space agrees with the parity of the corresponding weight. By \cite{S11}*{Theorem 9.9} the finite dimensional irreducible $\osp[r][2n]$-modules are all highest weight modules and the finite dimensional irreducible modules are up to isomorphism and parity shift uniquely determined by their highest weight. In the following we will restrict ourselves to $\mathcal{F}'$.

We follow \cite{GS10}*{Section 5} and fix a certain choice of simple roots. This gives then rise to a set of positive roots $\Phi^+$ and the corresponding $\rho$ is given by $\frac{1}{2}(\sum_{\alpha\in\Phi^+_0}\alpha-\sum_{\beta\in\Phi^+_1}\beta)$. 

The following lemma is due to \cite{GS10}*{Cor. 3}.

\begin{lem}\label{integraldominanceosp}
	Let $\lambda\in X(\osp[r][2n])$ and write $\lambda+\rho=\sum_{i=1}^m a_i\eps_i +\sum_{j=1}^n b_j\delta_j$. Then $\lambda$ is integral dominant if and only if $\lambda\in\bigoplus_{i=1}^m\bbZ\eps_i\oplus\bigoplus_{j=1}^n\bbZ\delta_j$ and the following conditions hold:
	\begin{itemize}
		\item If $\lie{g}=\osp[2m+1][2n]$\begin{enumerate}
			\item either $a_1>a_2>\dots >a_m\geq\frac{1}{2}$ and $b_1>b_2>\dots>b_n\geq\frac{1}{2}$,
			\item or $a_1>a_2>\dots>a_{m-l-1}>a_{m-l}=\dots=a_m=-\frac{1}{2}$ and $b_1>b_2>\dots>b_{n-l-1}\geq b_{n-l}=\dots=b_n=\frac{1}{2}$.
		\end{enumerate}
		\item If $\lie{g}=\osp[2m][2n]$\begin{enumerate}
			\item either $a_1>a_2>\dots >\abs{a_m}$ and $b_1>b_2>\dots>b_n>0$,
			\item or $a_1>a_2>\dots>a_{m-l-1}\geq a_{m-l}=\dots=a_m=0$ and $b_1>b_2>\dots>b_{n-l-1}> b_{n-l}=\dots=b_n=0$.
		\end{enumerate}
	\end{itemize}
\end{lem}
\begin{defi}
	We denote the set of integral dominant weights for $\osp[r][2n]$ by $X^+(\osp[r][2n])$. We write $L^{\lie{g}}(\lambda)$ for the finite dimensional irreducible module in $\mathcal{F}'$ with highest weight $\lambda\in X^+(\osp[r][2n])$.
\end{defi}

%

\begin{defi}\label{defiatypicality}
	On $\lie{h}^*$ we have the standard symmetric bilinear form $(\_,\_)$ which is given by $(\eps_i,\eps_j)=\delta_{i,j}$, $(\eps_i,\delta_j)=0$ and $(\delta_i,\delta_j)=-\delta_{i,j}$.
	
	A root $\alpha\in\Phi$ is called \emph{isotropic} if $(\alpha, \alpha)=0$.
	
	The \emph{degree of atypicality} $\mathrm{at}(\lambda)$ of a weight $\lambda\in\lie{h}^*$ is then the maximum number of mutually orthogonal odd isotropic roots $\alpha\in\Phi_1^+$ such that $(\lambda+\rho, \alpha) = 0$. An element $\lambda\in\lie{h}^*$ is called \emph{typical} if $\mathrm{at}(\lambda)=0$ and \emph{atypical} otherwise.
\end{defi}

By \cite{CW12}*{Theorem 2.30} any two weights lying in the same block have the same atypicality.

\begin{defi}\label{orderonhighestweights}
	We define a partial order on $X^+(\osp[r][2n])$ by saying that $\lambda\geq\mu$ for $\lambda$, $\mu\in X^+(\osp[r][2n])$ if $\lambda-\mu\in\sum_{\alpha\in\Phi^+}\mathbb{N}_0\alpha$.
\end{defi}

\subsection{Hook partitions} 

There exists another commonly used labelling set of the integral dominant weights for $\osp[r][2n]$ by \emph{$(n,m)$-hook partitions} (for $r=2m+1$ or $r=2m$).

\begin{defi}\label{defipartition} A partition $\hook\lambda$ is called \emph{$(n,m)$-hook} if $\hook\lambda_{n+1}\leq m$.
	By $\emptyset$ we denote the partition given by $\emptyset = (0,0,\dots)$ and by $\Lambda$ we denote the set of all partitions.
\end{defi}

The following definition from \cite{ES3}*{Definition 2.19} relates integral dominant weights and $(n,m)$-hook partitions.

\begin{defi}\label{highestweighttohookpartitions}
	We associate to an $(n,m)$-hook partition $\hook\lambda$ the weight $\mathrm{wt}(\hook\lambda)\in X^+(\osp[r][2n])$ via $\mathrm{wt}(\hook\lambda)=(a_1,\dots ,a_m|b_1,\dots,b_n)-\rho$, where $a_i$ and $b_j$ are defined as follows:
\begin{itemize}
	\item If $r$ is odd:
	\begin{equation*}
		b_j=\max\Bigl(\hook\lambda_j-j-\frac{\delta}{2}+1,\frac{1}{2}\Bigr)\quad\text{and}\quad a_i=\max\Bigl(\hook\lambda^t_i-i+\frac{\delta}{2},-\frac{1}{2}\Bigr).
	\end{equation*}
\item If $r$ is even:
\begin{equation*}
	b_j=\max\Bigl(\hook\lambda_j-j-\frac{\delta}{2}+1,0\Bigr)\quad\text{and}\quad a_i=\max\Bigl(\hook\lambda^t_i-i+\frac{\delta}{2},0\Bigr).
\end{equation*}
\end{itemize}
\end{defi}

This almost defines an identification of $(n,m)$-hook partitions with $X^+(\osp[r][2n])$. Only the integral dominant weights for $\osp[2m][2n]$ with $a_m<0$ do not correspond to an $(n,m)$-hook partition. For the actual bijection for $\Osp[r][2n]$ we refer to \cref{irreduciblemodulesforOspodd} and \cref{irreduciblemodulesforOspeven}. 

\subsection{Algebraic supergroups}

Instead of representations of $\mathfrak{osp}(r|2n)$ we consider in this paper representations of the algebraic supergroup $\Osp[r][2n]$ and the connected component $\Sosp[r][2n]$ of the identity. We refer to \cite[Section 7.1]{CH} \cite{ES3} for the definition. For $\Sosp[r][2n]$ we have a monoidal isomorphism of categories
\begin{equation}
	\osp[r][2n]\text{-mod}\cong\Sosp[r][2n]\text{-mod}.
\end{equation} 

Representations of $\Osp[r][2n]$ can then be understood via Harish-Chandra induction from $\Sosp[r][2n]$ \cite{ES3}. In order to explicitly describe the irreducible objects we need to distinguish whether $r=2m+1$ is odd or $r=2m$ is even. We will use $L^{\lie{g}}(\lambda)$ to refer to the irreducible $\Sosp[r][2n]$-module of highest weight $\lambda$ for an integral dominant weight $\lambda\in X^+(\osp[r][2n])$.

%


\begin{prop}\label{irreduciblemodulesforOspodd}
	For $G=\Osp[2m+1][2n]$ the set
	\begin{equation}
		X^+(G)=X^+(\lie{g})\times\bbZ/2\bbZ=\{(\lambda, \eps)\mid\lambda\in X^+(\lie{g}), \eps\in\{\pm\}\}
	\end{equation} is a labelling set for the isomorphism classes of finite dimensional irreducible $G$-modules in $\cF$. The irreducible module $L(\lambda, \eps)$ is just the irreducible module $L^{\lie{g}}(\lambda)$ extended to a $G$-module by letting $-\id$ act by $\pm 1$. Moreover the following map is a bijection.
\begin{align*}
	\Psi\colon\{(m,n)\text{-hook partitions}\}\times\bbZ/2\bbZ&\to X^+(G)\\
	(\hook\lambda, \eps)&\mapsto(\mathrm{wt}(\hook\lambda), \eps)
\end{align*}
\end{prop}


\begin{rem} The irreducible representations $L(\lambda, \eps)$ and $L(\mu, \eps')$ lie in the same block if and only if $\eps=\eps'$ and $L^\lie{g}(\lambda)$ and $L^\lie{g}(\mu)$ belong to the same block in $\Sosp[2m+1][2n]$-mod, see \cite{ES3}*{Remark 2.8 + Corollary 2.9}.
\end{rem}

%


\begin{defi}
	For $G=\Osp[2m][2n]$ and $\sigma$ corresponding to the nontrivial element of $\bbZ/2\bbZ$ we introduce the set
\begin{equation}
	X^+(G)\coloneqq\{(\lambda, \eps)\mid\lambda\in X^+(\lie{g})/\sigma\text{ and }\eps\in\Stab_\sigma(\lambda)\}
\end{equation} where $\Stab_\sigma(\lambda)$ is the stabilizer of $\lambda$ for the group generated by $\sigma$.

	Every $\lambda\in X^+(\lie{g})$ is contained in a unique orbit consisting of either one or two elements.
	In the former case we denote the orbit by $\lambda$. The stabilizer has two elements and we will write $(\lambda, +)$ for $(\lambda, e)$ and $(\lambda, -)$ for $(\lambda, \sigma)$.
	In the latter case the stabilizer is trivial and we will denote the orbit by $\lambda^G$ and abbreviate $(\lambda^G, e)$ by $\lambda^G$.
\end{defi}

%

%


\begin{prop}\label{irreduciblemodulesforOspeven}
	Let $G=\Osp[2m][2n]$, $\lie{g}=\osp[2m][2n]$ and $G'=\Sosp[2m][2n]$. Assume that \begin{equation}
		 \lambda=\sum_{i=1}^m a_i\eps_i +\sum_{j=1}^n b_j\delta_j-\rho\in X^+(\lie{g})	
		\end{equation}
	 is an integral dominant weight. Then we have the following:
	\begin{enumerate}
%
	\item The $L(\lambda, \eps)$ for $(\lambda, \eps)\in X^+(G)$ form a complete list of pairwise nonisomorphic irreducible $\Osp[2m][2n]$-modules in $\cF$.
	\item\label{irreduciblemodulesforOspeventhree} We have a bijection
	\begin{align*}
		\Psi\colon\{\text{signed }(n,m)\text{-hook partitions}\}\to X^+(G),\quad\hook\lambda&\mapsto\mathrm{wt}(\lambda),\\
		(\hook\lambda,\pm)&\mapsto(\mathrm{wt}(\lambda),\pm,).
	\end{align*}

	\end{enumerate}
\end{prop}

	The proof can be found in \cite{ES3}*{Proposition 2.12+2.13, Lemma 2.21}.

\subsection{The Deligne category $\Rep_\delta$}

\begin{defi}
	Let $\delta\in\mathbb{C}$. A \emph{Brauer diagram of type $(r,s)$} is a partitioning of the set $\{1, \dots, r+s\}$ into subsets of cardinality $2$. This can be represented diagrammatically by identifying $p\in\{1, \dots, r+s\}$ with the point $(p,0)$ if $1\leq p\leq r$ and $(p-r, 1)$ if $r<p\leq r+s$ in the plane and connecting the points in each subset by an arc inside the rectangle spanned by these points.
\end{defi}


\begin{defi}\label{defibrauercat}
	The \emph{Brauer category} $\Br(\delta)$ (also called the skeletal Deligne category $\Rep^0(O_{\delta})$ in \cite{CH}) for $\delta\in\bbC$ is the 
	category with objects $d\in\mathbb{Z}_{\geq0}$ and $\Hom_{\Br(\delta)}(r,s)$ is the $\mathbb{C}$-vector space with basis given by all $(r,s)$-Brauer diagrams. The multiplication is given by stacking diagrams vertically and evaluating a circle to $\delta$. The Brauer category admits a monoidal structure, given by $m\otimes n=m+n$ on objects and on morphisms it is giving by stacking diagrams horizontally.
	The \emph{Deligne category} $\Rep_\delta$ is the additive Karoubian envelope of $\Br(\delta)$.
	The \emph{Brauer algebra} $\Br_d(\delta)$ is the endomorphism algebra $\End_{\Br(\delta)}(d)$ of $d$.
\end{defi}

The primitive idempotents of $\Br_d(\delta)$ can be constructed from the group algebra of the symmetric group, see \cite{CH}. For any partition $\lambda$ with $|\lambda| < d$ we denote the so-obtained primitive idempotents by $e_{\lambda}^{(i)}$ (for certain $i \in \mathbb{N}$). 

\begin{thm} (\cite{CH}*{Thm 3.4})
	The set $\{e^{(i)}_\lambda\mid\lambda\in\Lambda_d(\delta)\}$ is a complete set of pairwise nonconjugate primitive idempotents in $\Br_d(\delta)$, where $\Lambda_d(\delta)$ denotes the set
	\begin{equation}	
		\Lambda_d(\delta)\coloneqq\begin{cases}
			\{\lambda\in\Lambda\mid\abs{\lambda}=d-2i, 0\leq i\leq\frac{d}{2}\}&\text{if $\delta\neq 0$ or $d$ is odd or $d=0$,}\\
			\{\lambda\in\Lambda\mid\abs{\lambda}=d-2i, 0\leq i<\frac{d}{2}\}&\text{if $\delta=0$ and $d>0$ is even.}
		\end{cases}
	\end{equation}
\end{thm}

\begin{defi}
	In $\Rep_\delta$ every idempotent has an image and we set $\R_\delta(\lambda)\coloneqq\im e^0_\lambda$.
\end{defi}

\begin{thm}\label{classindecbrcat} \cite{CH}*{Thm. 3.5}
	The assignment $\lambda\mapsto\R_\delta(\lambda)$ defines a bijection between the set $\Lambda$ of all partitions and isomorphism classes of nonzero indecomposable objects in $\Rep_\delta$.
\end{thm}

\begin{rem} For the object $\R_{\delta}((1))$ we write also $\R_\delta(\ydiagram{1})$. It agrees with the object $1$ in the Brauer category $\Br(\delta)$.
\end{rem}

\subsection{From $\Rep_{\delta}$ to $\mathrm{OSp}$}

As the natural representation $V$ of $\Osp[r][2n]$ has superdimension $\delta=r-2n$ there exists a symmetric monoidal functor $\mathbb{F}=\mathbb{F}_{(r|2n)}\colon\Rep_\delta\to\cF$ by the universal property of the Deligne category $\Rep_\delta$, which is given by sending $1$ to $V$ (see \cite{D19}*{Proposition 9.4}). Then we have $\mathbb{F}(d)=V^{\otimes d}$ and we get an action of $\Br_d(\delta)$ on $V^{\otimes d}$. This functor is actually full (see \cite{LZ17}*{Thm. 5.6}) and thus we have in particular a surjective algebra homomorphism
\begin{equation}\label{FromBrauerToOspIsFull}
	\Phi_{d, \delta}\colon \Br_d(\delta)\to \End_{\cF}(V^{\otimes d}).
\end{equation}

%

\begin{thm}\label{classificationofindecsummandsabstract} \cite{CH}*{Thm. 7.3} The assignment $\lambda\mapsto\mathbb{F}\R_\delta(\lambda)$ defines a bijection between the set $\Lambda(d,r,n)\coloneqq\{\lambda\in\Lambda_d(\delta)\mid\mathbb{F}\R_\delta(\lambda)\neq 0\}$ and a set of representatives of isomorphism classes of nonzero indecomposable summands in $V^{\otimes d}$.
\end{thm}


\subsection{Endofunctors}

In the category $\Rep_\delta$, we have the endofunctor $\ind=\_\boxtimes\R_\delta(\ydiagram{1})$ which is given by tensoring with $\R_\delta(\ydiagram{1})$. Note that diagrammatically it adds to each basis morphism one strand to the right. On the other hand we can consider the endofunctor $\_\otimes V$ in the category $\Osp[r][2n]$-mod, and as the functor $\mathbb{F}$ is monoidal, we also have 
\begin{equation*}
	(\_\otimes V)\circ\mathbb{F}\cong\mathbb{F}\circ\ind.
\end{equation*}
In the following we would like to refine this isomorphism by decomposing $\_\otimes V$ and $\ind$ into a direct sum of functors. For this we are going to introduce the so called \emph{Jucys--Murphy elements}, originally defined by Nazarov in \cite{Naz96}.
\begin{defi}
	The \emph{Jucys--Murphy elements} $\xi_k\in\Br_d(\delta)$ for $1\leq k\leq d$ are the elements 
	\begin{equation*}
		\xi_k\coloneqq\frac{\delta-1}{2}+\sum_{1\leq i<k}(s_{i,k}-\tau_{i,k}).
	\end{equation*}
Furthermore we define $\Omega_d\coloneqq2(\xi_1+\dots+\xi_d)$.
\end{defi}
\begin{lem}
	The Jucys--Murphy elements generate a commutative subalgebra $\mathrm{GZ}_d(\delta)$ of $\Br_d(\delta)$ and the element $\Omega_d$ is central in $\Br_d(\delta)$.
\end{lem}
\begin{proof}
	This is \cite{Naz96}*{Cor. 2.2}.
\end{proof}
This leads us to the following refinement of the induction functor for $\Rep_\delta$. For the well-definedness, we invite the reader to consult \cite{OSPII}*{Lemma 2.15}.
\begin{defi}\label{iinddef}
	For $i\in\bbZ+\frac{\delta}{2}$ we define the \emph{$i$-induction functor}
	\begin{equation*}
	\begin{aligned}
		i\mind\colon\Rep_\delta&\to\Rep_\delta,\\
		M&\mapsto\text{proj}_i(M\boxtimes\R_\delta(\ydiagram{1})).
	\end{aligned}
	\end{equation*}
Here, for an indecomposable object $\R_\delta(\lambda)$, $\text{proj}_i$ is the projection onto the generalized $i$-eigenspace of $\xi_{\abs{\lambda}+1}$ viewed as an element in $\End_{\Rep_\delta}(\R_\delta(\lambda)\boxtimes\R_\delta(\ydiagram{1}))$, which is then extended to arbitrary objects $M$.
\end{defi}
We clearly have $\bigoplus_{i\in\bbZ+\frac{\delta}{2}}i\mind =\ind$.

On the other hand we have the Casimir element $C$ in the universal enveloping algebra $U(\osp[r][2n]$) (see \cite{M}*{Lemma 8.5.1}). This is central and thus multiplication by $C$ denotes an endomorphism of every $\Osp[r][2n]$-module and we can look at the eigenvalue of this endomorphism. 
\begin{defi}\label{defiitranslation}
	The endofunctor $\_\otimes V$ of $\cF$ decomposes as $\_\otimes V=\bigoplus_{i\in\bbZ+\frac{\delta}{2}}\theta_i$, where $\theta_i$ denotes the projection onto the summand, which changes the generalized eigenvalue of $C$ by $2i$. We call $\theta_i$ the \emph{$i$-translation functor}.
\end{defi}

The following theorem from \cite{OSPII}*{Thm. 8.10} relates the notions of $i$-induction and $i$-translation.
\begin{thm}\label{iinductionswapsitranslation}
	The functor $\mathbb{F}_{(r|2n)}$ intertwines $i$-induction with $i$-translation for any $i\in\bbZ+\frac{\delta}{2}$, i.e.\begin{equation*}
		\mathbb{F}_{(r|2n)}\circ i\mind\cong\theta_i\circ\mathbb{F}_{(r|2n)}.
	\end{equation*}
\end{thm}
\begin{ex}
	The endofunctor $0$-ind maps $\R_\delta(\emptyset)$ to $\R_\delta(\ydiagram{1})$. On $\cF$ this corresponds to the fact that $0$-translation applied to the trivial representation gives the natural representation.
\end{ex}


\section{Weight diagrams}\label{wdiag}

Throughout this section we fix natural numbers $r$ and $n$ and set $\delta=r-2n$, $m=\lfloor\frac{\delta}{2}\rfloor$ and denote by $L=\Bigl(\bbZ+\frac{\delta}{2}\Bigr)\cap\mathbb{R}_{\geq 0}$ the nonnegative (half) integer line. Furthermore we call elements of $L$ \emph{vertices}.

For every $\R_\delta(\lambda)$ we will later get an irreducible $K$-module (where $K$ is the Khovanov algebra defined in Section \ref{sec:Proj}). These irreducible modules will be labeled by \emph{Deligne weight diagrams} and we will present the correspondence between partitions and Deligne weight diagrams.
On the other hand the highest weights of irreducible $\Osp[r][2n]$-modules are parametrized by $(n,m)$-hook partitions (together with a sign $\eps$) and we will also recall their associated weight diagrams (the \emph{hook weight diagrams}).

 Ehrig and Stroppel provided in \cite{OSPII}*{Definition 7.7} a combinatorially defined map $\dagger$ which associates to the Deligne weight diagram the hook weight diagram of the head of the corresponding $\mathbb{F}\R_\delta(\lambda)$ (see \cref{daggermapgiveshwofrlambda}).

Ehrig and Stroppel provided a third set of weight diagrams, the so called \emph{super weight diagrams}. For proving a Morita equivalence between $\Osp[r][2n]$-mod and a subquotient $e\tilde{K}e$ of the Khovanov algebra, they identified $e\tilde{K}e$ with a projective generator for $\Osp[r][2n]$ (see \cite{OSPII}*{Theorem 10.5} and also the proof of \cref{mainthm}). The super weight diagrams are mainly the Deligne weight diagrams which are associated to partitions which give rise to projective $\mathbb{F}\R_\delta(\lambda)$. 

\subsection{Basic definitions}
\begin{defi}
	A \emph{weight diagram} $\mu$ is a map $\mu\colon L\to\{\times,\circ,\vee,\wedge,\diamond\}$ such that $\diamond$ can only occur as image of $0$ and conversely the image of $0$ can only be $\circ$ or $\diamond$. Furthermore for $?\in\{\circ, \times, \vee, \wedge, \diamond\}$, we denote by $\#?(\mu)\in\bbZ_{\geq0}\cup\{\infty\}$ the number of $?$'s appearing in $\mu$.
	
	The symbols $\circ$, $\times$, $\vee$, $\wedge$ and $\diamond$ are called \emph{nought}, \emph{cross}, \emph{down}, \emph{up} and \emph{diamond}, respectively.
\end{defi}
\begin{defi}
	We call a weight diagram $\mu$ \emph{admissible} if $\#{\circ}(\mu)+\#{\times}(\mu)+\#{\wedge}(\mu)<\infty$ and \emph{flipped} if $\#{\circ}(\mu)+\#{\times}(\mu)+\#{\vee}(\mu)<\infty$.
\end{defi}
\begin{defi}\label{defiblock}
	Two admissible weight diagrams $\lambda$ and $\mu$ belong to the same block $\Lambda$ if the position of $\circ$'s and $\times$'s agree and either $\#{\wedge}(\lambda)\equiv\#{\wedge}(\mu)$(mod $2$) or they both start with $\diamond$.
\end{defi}

We usually draw a weight diagram as a sequence together with the lowest number of $L$ (sometimes we also omit this number), i.e.
\begin{center}
\begin{tikzpicture}[scale=0.7]
	\node at (0, 0.5) {$0$};
	\wdiagnoline{d v v w o x v v {$\dots$}}
\end{tikzpicture}\\
or\\
\begin{tikzpicture}[scale=0.7]
	\node at (0, 0.5) {$\frac{1}{2}$};
	\wdiagnoline{w w x x o v o v {$\dots$}}
\end{tikzpicture}
\end{center}
By turning every symbol upside down (i.e.~exchanging $\vee$'s and $\wedge$'s) we obtain a bijection between admissible and flipped weight diagrams.

\begin{defi}\label{defbruhatmove}We call two symbols neighbored if they are only separated by $\circ$'s and $\times$'s. For the following a $\diamond$ can be interpreted either as $\vee$ or $\wedge$.
	For two admissible weight diagrams $\mu$, $\lambda$ we say that $\mu$ is obtained from $\lambda$ by a \emph{Bruhat move}, if one of the following holds:
	\begin{itemize}
		\item $\lambda$ has a pair of neighboring labels $\vee\wedge$ (say at positions $i$,$j$) and $\mu$ is obtained by replacing these by $\wedge\vee$. This is called a \emph{type A move} applied at positions $i$ and $j$.
		\item $\lambda$ starts (up to some $\circ$'s and $\times$'s) with neighboring labels $\wedge\wedge$ at positions $i$ and $j$ and $\mu$ is obtained by replacing these with $\vee\vee$. This is called a \emph{type D move} applied at positions $i$ and $j$.
	\end{itemize}
	We define a partial order on the set of admissible weight diagrams by saying $\lambda\leq\mu$ if $\mu$ can be obtained from $\lambda$ by a sequence of Bruhat moves. Note that $\lambda\leq\mu$ implies that $\lambda$ and $\mu$ lie in the same block.
\end{defi}

%
%

\begin{defi}
	Let $\lambda$ and $\mu$ be two admissible weight diagrams belonging to the same block. 
	Suppose that $\mu$ has $m$ symbols $\wedge$ and $\lambda$ has $m+2k$ symbols $\wedge$. We define then $l_0(\lambda, \mu)\coloneqq2k$ and for $i\in L$ we set $l_i(\lambda, \mu)=0$ if $\lambda(i)\in\{\times, \circ\}$ and otherwise
	\begin{equation*}
		l_i(\lambda, \mu)\coloneqq2k+\#\{i\geq j\in L\mid \lambda(j)=\vee\}-\#\{i\geq j\in L\mid\mu(j)=\vee\}.
	\end{equation*}
	Note that, as $\lambda$ and $\mu$ are admissible, we have $l_n(\lambda, \mu)=0$ for big enough $n$. Therefore $l(\lambda, \mu)=\sum_{i\geq 1}l_i(\lambda, \mu)$ is well-defined and finite.
\end{defi}

\begin{lem}Let $\lambda$ and $\mu$ be admissible weight diagrams. Then $\lambda\leq\mu$ if and only if $l_i(\lambda, \mu)\geq 0$ for all $i\geq 0$. 
\end{lem}

\begin{proof}
	This is \cite{TS}*{Prop. 1.1.8} after observing that we can restrict to the finite case as $\lambda(i)=\mu(i)=\vee$ for all $i\in L$ big enough.
\end{proof}

%

\begin{defi}\label{assoccupdiag}The \emph{cup diagram $\underline{\mu}$} associated to an admissible or a flipped weight diagram $\mu$ is obtained by applying the following steps.
	\begin{enumerate}[label=(C-\arabic*), itemsep=-0.25ex]
		\item First connect neighbored vertices labeled $\vee\wedge$ successively by a cup, i.e.~we connect the vertices by an arc forming a cup below. Repeat this step as long as possible, ignoring already joint vertices. Note that the result is independent of the order in which the connections are made
		\item Attach a vertical ray to each remaining $\vee$.
		\item\label{enum:dotone} Connect pairs of neighbored $\wedge$'s from left to right by cups (we interpret $\diamond$ for this as a $\wedge$). It might be necessary to attach infinitely many cups in this step.
		\item\label{enum:dottwo} If a single $\wedge$ or $\diamond$ remains, attach a vertical ray.
		\item\label{enum:dotthree} Put a marker $\bullet$ on each cup created in \cref{enum:dotone} and each ray created in \cref{enum:dottwo}.
		\item We erase the marker from the component that contains the $\diamond$ if the number of placed markers in \cref{enum:dotthree} is finite and odd.
		\item Finally delete all $\vee$ and $\wedge$ labels at vertices.
	\end{enumerate} 
	
The cups and rays are always drawn without intersections, and two cup diagrams are said to be the same if there is a bijection between the cups and rays, respecting the connected vertices and the markers $\bullet$. We call cups and rays with a marker \emph{dotted} and those without $\bullet$ \emph{undotted}.

A \emph{cap diagram} is just the horizontal mirror image of a cup diagram, for a cup diagram $a$ we denote by $a^*$ the cap diagram obtained by horizontal mirroring and vice versa.


\end{defi}

\begin{defi}
	Given a weight diagram $\mu$, we call the total number of cups (dotted as well as undotted) in its weight diagram $\underline{\mu}$ the \emph{defect} $\Def(\mu)$ of $\mu$. The \emph{rank} of $\mu$ is defined to be $\rk(\mu)\coloneqq\min(\#{\circ}(\mu),\#{\times}(\mu))$. The \emph{layer number} of $\mu$ is $\kappa(\mu)\coloneqq\Def(\mu)+\rk(\mu)$.
\end{defi}

\begin{defi}\label{wdiagtosubset}
	We associate to a subset $S\subseteq\bbZ+\frac{\delta}{2}$ the \emph{weight diagram} $\lambda_S$, which is given at position $i\in L$ by $\diamond$ if $i=0\in S$, and otherwise
	\begin{equation*}
		\begin{cases}
			\wedge&\text{if $i\in S$ but $-i\notin S$,}\\
			\vee&\text{if $-i\in S$ but $i\notin S$,}\\
			\times&\text{if $i\in S$ and $-i\in S$,}\\
			\circ&\text{if $i\notin S$ and $-i\notin S$.}
		\end{cases}
	\end{equation*}
\end{defi}

\begin{defi}
	An \emph{oriented cup diagram} $a\lambda$ is a cup diagram $a$ together with a weight diagram $\lambda$ such that the positions of the appearing $\circ$'s (resp.~$\times$'s) agree and every cup (resp.~ray) is oriented as in \Cref{orient}.
	An oriented cap diagram $\lambda b$ is just a cap diagram $b$ together with a Deligne weight diagram $\lambda$ such that $b^*\lambda$ is an oriented cup diagram.
\end{defi}

\begin{figure}[h]
	\begin{tikzpicture}[scale=0.7]
		\FPset\stddiff{2}
		\wdiagnoline{v w w v d w d v w w v v d w d v v d w d}\cups{0 1, 2 3, 4 5, 6 7, 8 9 d, 10 11 d, 12 13 d, 14 15 d}\rays{i 16, i 17, i 18 d, i 19 d}\node at (0.5, {\currh+0.6}) {0};\node at (2.5, {\currh+0.6}) {1};\node at (4.5, {\currh+0.6}) {0};\node at (6.5, {\currh+0.6}) {1};
		\node at (8.5, {\currh+0.6}) {0};\node at (10.5, {\currh+0.6}) {1};\node at (12.5, {\currh+0.6}) {0};\node at (14.5, {\currh+0.6}) {1};
		\node at (16, {\currh+0.6}) {0};\node at (17, {\currh+0.6}) {0};\node at (18, {\currh+0.6}) {0};\node at (19, {\currh+0.6}) {0};
		\begin{scope}[yshift={2*\currh cm + 1.2cm}]
			\caps{0 1, 2 3, 4 5, 6 7, 8 9 d, 10 11 d, 12 13 d, 14 15 d}\rays{16 i, 17 i, 18 i d, 19 i d}\wdiagnoline{v w w v d w d v w w v v d w d v v d w d}
		\end{scope}
	\FPset\stddiff{1}
	\end{tikzpicture}
\caption{Orientations and degrees}\label{orient}
\end{figure}
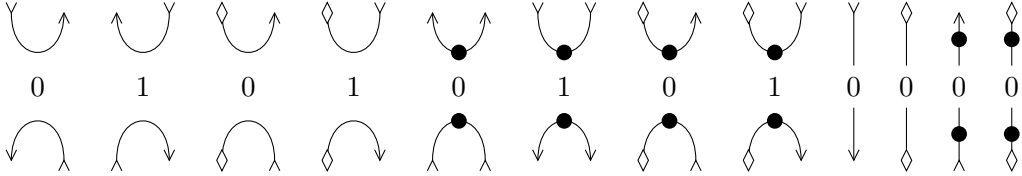

\begin{defi}
	A \emph{circle diagram} $ab$ is a cup diagram $a$ put beneath a cap diagram $b$, such that the positions of the appearing $\circ$'s (resp.~$\times$'s) agree. An \emph{oriented circle diagram} $a\lambda b$ is a circle diagram $ab$ together with a Deligne weight diagram $\lambda$ such that $a\lambda$ and $\lambda b$ are oriented cup (resp.~cap) diagrams.
\end{defi}

\begin{defi}Given an oriented cup diagram $a\lambda$ each cup and ray has an associated integer according to \Cref{orient}. The sum of all these integers is called the \emph{degree} $\deg(a\lambda)$ of the oriented cup diagram $a\lambda$. The degree of an oriented cap diagram $\lambda b$ is defined as $\deg(\lambda b)\coloneqq\deg (b^*\lambda)$.
	For an oriented circle diagram $a\lambda b$, we define $\deg(a\lambda b)=\deg(a\lambda)+\deg(\lambda b)$.
	
	The cups and caps in \Cref{orient} with a $1$ are called \emph{clockwise} and those with a $0$ \emph{anticlockwise}.
\end{defi}

%

\subsection{Deligne weight diagrams}

\begin{defi}\label{defdeligneweightdiagram}
	Given a partition $\lambda\in\Lambda$, we associate to it the set 
	\begin{equation*}
		X(\lambda)\coloneqq\{\lambda^t_i-i+1-\frac{\delta}{2}\mid i\geq 1\}\subset\bbZ+\frac{\delta}{2}.
	\end{equation*}
Using \cref{wdiagtosubset} we can associate a weight diagram to $X(\lambda)$. We denote it by $\lambda_\delta$ and call it \emph{Deligne weight diagram}. Furthermore we denote the set of all Deligne weight diagrams by $\Lambda_\delta$.
\end{defi}

\begin{lem}\label{bijectionpartitiondeligneweightdiag}
The assignment $\lambda\mapsto\lambda_\delta$ defines a bijection \begin{equation*}
	\{\text{partitions}\}\to\{\text{admissible weight diagrams $\mu$ such that $\#{\circ}(\mu)-\#{\times}(\mu)=\lfloor\frac{\delta}{2}\rfloor$}\}
\end{equation*}
\end{lem}

\begin{proof}
	This is \cite{OSPII}*{Lemma. 7.1}
\end{proof}

With the description of partitions in terms of Deligne weight diagrams, we are now able to classify the set $\Lambda(d,r,n)$ \cite{CH}*{Cor. 7.14}.

\begin{thm}\label{classificationofindecomposablesummandsvialayernumberofdelignewdiag}
	There is an equality of sets $\Lambda(d,r,n)=\{\lambda\in\Lambda_d(\delta)\mid\kappa(\lambda_\delta)\leq\min(m,n)\}$.
\end{thm}
We will call these diagrams \emph{tensor weight diagrams}. Furthermore we have that $\mathbb{F}\R_\delta(\lambda)$ is projective if and only if $\kappa(\lambda_\delta)=\min(n,m)$ (see \cite{CH}*{Lemma 7.16}).

So far we have no idea what the head of the indecomposable summands $\mathbb{F}\R_\delta(\lambda)$ looks like. The next section is going to address this and for this purpose introduces \emph{hook weight diagrams}.

\subsection{Hook weight diagrams}

Given a partition $\lambda$, we denote by $\lambda^\infty$ the weight diagram, which is obtained from $\lambda_\delta$ by turning all symbols upside down (or equivalently swapping $\vee$'s and $\wedge$'s). 
Note that these diagrams are then \emph{flipped}.

The bijection in \cref{bijectionpartitiondeligneweightdiag} clearly induces a bijection 
\begin{equation}\label{bijectionpartitionflippedweightdiag}
	\{\text{partitions}\}\to\{\text{flipped weight diagrams $\mu$ such that $\#{\circ}(\mu)-\#{\times}(\mu)=\lfloor\frac{\delta}{2}\rfloor$}\}
\end{equation} which is given by $\lambda\mapsto\lambda^\infty$.
Integral dominant weights for $\osp[r][2n]$ are characterized by $(n,m)$-hook partitions and thus we restrict the bijection \eqref{bijectionpartitionflippedweightdiag} to $(n,m)$-hook partitions. 

\begin{lem}\label{hookpartitionstohookweightdiag}
	The map $\lambda\mapsto\lambda^\infty$ gives rise to a bijections of sets
	\begin{equation*}
		\{\text{$(n,m)$-hook partitions}\}\to\left\{\substack{\text{flipped weight diagrams $\mu$ such that}\\\text{ $\#{\circ}(\mu)-\#{\times}(\mu)=\lfloor\frac{\delta}{2}\rfloor$ and}\\\text{ $\#{\vee}(\mu)\leq\min(m,n)-\rk(\mu)$}}\right\}=:\Gamma_\delta(n,m).
	\end{equation*}
\end{lem}

We can transport the equivalence relation on $X^+(\lie{g})\times\{\pm\}$ (see \cref{irreduciblemodulesforOspodd} and \cref{irreduciblemodulesforOspeven}) to $\Gamma_\delta(n,m)\times\{\pm\}$. We denote the set of equivalence classes by $s\Gamma_\delta(n,m)$ and call such equivalence classes \emph{signed hook weight diagrams}.  Hence we have a bijection between $X^+(G)$ and $s\Gamma_\delta(n,m)$. We will abuse notation and write $(\lambda, \eps)$ for the equivalence class of $(\lambda, \eps)$ in $s\Gamma_\delta(n,m)$.


\begin{defi}\label{blocksforhookweigthdiag}
	Two signed hook weight diagrams $(\lambda, \eps)$, $(\mu, \eps')$ belong to the same block if the positions of $\circ$ and $\times$ in $\lambda$ and $\mu$ agree and if $\eps=\eps'$ for some representatives of the respective equivalence classes. 
	
	For two signed hook weight diagrams $(\lambda, \eps)$ and $(\mu, \eps')$ belonging to the same block, we have $(\lambda, \eps)\leq(\mu, \eps')$ if $\mu$ can be obtained from $\lambda$ via changing some $\wedge$'s into $\vee$'s or by changing $\vee\wedge$'s into $\wedge\vee$'s.
\end{defi}

\begin{rem}\label{hookweightdiagatypicalityandorder}

The notion of blocks according to \cref{blocksforhookweigthdiag} agrees with the one given before \cref{defiatypicality} by \cite{GS10}*{Section 6} (see also \cite{GS13}*{Section 4.5}) after translating their combinatorics to the one of Ehrig and Stroppel using \cite{ES3}*{Section 6}.

For the degree of atypicality from \cref{defiatypicality} for a weight $\lambda\in X^+(\lie{g})$ (with $a_m\geq 0$ in the notation of \cref{integraldominanceosp} if $r=2m$), we have $\mathrm{at}(\lambda)=\min(m,n)-\rk(\lambda^\infty)$. The condition $a_m\geq 0$ in the even case is necessary because for those weights with $a_m<0$ we did not define an associated $(n,m)$-hook partition. 
This follows from \cite{GS13}*{Section 4.5} by translating their combinatorics to our setting.

Additionally we can see that if such a weight is typical, $\lambda^\infty$ is actually $\vee$-avoiding (i.e.~no $\vee$ occurs) as $\min(m,n)=\rk(\lambda^\infty)$ and thus $\#{\vee}(\mu)=0$ by \cref{hookpartitionstohookweightdiag}.

Furthermore \cref{blocksforhookweigthdiag} agrees with \cref{orderonhighestweights} under the identifications in \cref{highestweighttohookpartitions} and \cref{hookpartitionstohookweightdiag} for two weights of the same block.
\end{rem}
%
%
%

Given a tensor weight diagram $\lambda_\delta$ (i.e.~a Deligne weight diagram with $\kappa(\lambda_\delta)\leq\min(n,m)$), we would like to determine the head of the associated indecomposable $\Osp[r][2n]$-module $\mathbb{F}\R_\delta(\lambda)$. It will turn out that the head is actually irreducible, and its highest weight can be obtained via the map $\dagger$ defined below.

\begin{defi}
	The map $\dagger\colon\{\text{tensor weight diagrams}\}\to s\Gamma_\delta(n,m)$ is defined as follows.
For  a Deligne weight diagram $\lambda_\delta$ with $\kappa(\lambda_\delta)=\min(m,n)$ (i.e.~$\mathbb{F}\R_\delta(\lambda)$ is projective), we define $\lambda_\delta^\dagger\coloneqq(\Phi(\lambda), \eps)$, where $\Phi(\lambda)$ is the weight diagram obtained from $\lambda_\delta$ by turning all symbols $\vee$ corresponding to rays in $\underline{\lambda_\delta}$ into $\wedge$'s. In case that $\delta$ is odd, the sign $\eps$ is given by $+$ (resp.~$-$) if the parity of the partition $\lambda$ (under the bijection from \cref{bijectionpartitiondeligneweightdiag}) is even (resp.~odd).
In case that $\delta$ is even, the sign $\eps$ is $+$ (resp.~$-$) if the leftmost ray of $\underline{\lambda_\delta}$ is undotted (resp.~dotted) and not at position zero and $\eps=\pm$ is the leftmost ray is at position zero.
For a tensor weight diagram $\lambda_\delta$ with $\kappa(\lambda_\delta)<\min(n,m)$, we define $\lambda_\delta^\dagger\coloneqq(\Phi(\lambda), \eps)$, where $\Phi(\lambda)$ is given by turning \underline{all} symbols corresponding to rays in $\underline{\lambda_\delta}$ upside down. The sign is defined in the same way as for projective tensor weight diagrams if $\delta$ is odd. In case that $\delta$ is even, we always set $\eps=+$.
\end{defi}

The main result  of this section is the classification theorem from \cite{OSPII}*{Theorem 7.8}.

\begin{thm}\label{daggermapgiveshwofrlambda}
	Let $\lambda\in\Lambda(d,r,n)$, then:
	\begin{enumerate}
		\item The indecomposable summand $\mathbb{F}\R_\delta(\lambda)$ of the $\Osp[r][2n]$-module $V^{\otimes d}$ has irreducible head isomorphic to $L(\lambda_\delta^\dagger)$.
		\item In particular, if $\mathbb{F}\R_\delta(\lambda)$ is projective, then $\mathbb{F}\R_\delta(\lambda)\cong P(\lambda_\delta^\dagger)$.
		\item Any indecomposable projective in $\Osp[r][2n]$-mod is obtained in this way for some $\lambda$ and $d$.
	\end{enumerate}
\end{thm}

%

We describe all socle/radical layers of $\mathbb{F}\R_\delta(\lambda)$ later in Section \ref{applications}.

\subsection{Super weight diagrams}


%

The ``problem'' with hook weight diagrams is that the associated cup diagrams always have infinitely many dotted cups. To define the Khovanov algebra we ``only need'' the Deligne weight diagrams $\lambda$ such that $\mathbb{F}\R_\delta(\lambda)$ is projective. Up to some technicalities these are the \emph{super weight diagrams}.


\begin{defi}\label{defsuperweightdiag}
	Given a signed hook weight diagram $(\lambda, \eps)\in s\Gamma_\delta(n,m)$, we define the associated \emph{super weight diagram} $\lambda_\eps^\owedge$ as the unique admissible weight diagram $\mu$ with $\kappa(\mu)=\rk(\mu)+\Def(\mu)=\min(m,n)$ such that 
	\begin{itemize}
		\item $\underline{\mu}$ is obtained from $\underline{\lambda}$ by replacing (infinitely many) dotted cups by two vertical rays each
		\item and possibly a dot on the resulting leftmost ray depending on $\eps$ according to the following rule:
		\begin{itemize}
			\item If $\delta$ is even, we put a dot on the leftmost ray if $\eps=+$ and no dot if $\eps=-$. 
			\item If $\delta$ is odd, we do the following: For each symbol $\circ$ or $\times$ we count the number of endpoints of rays and cups in $\lambda$ to the left of this symbol (this is the same as the number of $\vee$'s and $\wedge$'s to the left), and take the sum plus the total number of undotted cups in $\lambda$ (this equals the number of $\vee$'s). Let this be $s$. If $s$ is even, we put a dot on the first ray if $\eps=+$ and no dot if $\eps=-$. If $s$ is odd, we put a dot if $\eps=-$ and no dot if $\eps=+$.
		\end{itemize}
	\end{itemize}
\end{defi}

If one follows the explicit construction steps, one sees the following (or consult \cite{OSPII}*{Proposition 8.4}) 

\begin{prop}\label{projsuperweightdiag}
	Let $\lambda_\delta$ be a Deligne weight diagram associated to a projective $\mathbb{F}\R_\delta(\lambda)$. We denote the super weight diagram $(\lambda_\delta^\dagger)^\owedge_\eps$ by $\mu$. Then we have that $\underline{\mu}$ and $\underline{\lambda_\delta}$ agree up to a dot on the leftmost ray, and additionally a dot on the cup attached to $\diamond$ in case there is such a cup.
\end{prop}
\begin{rem}\label{swdiagdifferentsignrules}
	We would like to emphasize here that the rule whether or not to put a dot, can be altered. We could have also chosen the reverse association, but we decided to stick with the convention of \cite{OSPII}*{Definition 8.1}.
	In case of the reverse association, the analogue of \cref{projsuperweightdiag} would say that $(\lambda_\delta^\dagger)^\owedge_\eps=\lambda_\delta$ for a Deligne weight diagram associated to a projective $\mathbb{F}\R_\delta(\lambda)$.
\end{rem}

\section{Khovanov algebras and projective functors} \label{sec:Proj}

Throughout this chapter we fix $\delta\in\bbZ$. By a weight diagram we mean a Deligne weight diagram corresponding to this $\delta$ and by a cup or cap diagram, we mean a cup or cap diagram associated to some Deligne weight diagram for $\delta$.

It was proven by Ehrig--Stroppel \cite{OSPII}*{Theorem 10.5} that there is an equivalence $\Psi$ of categories between $\cF$ and $e\tilde{K}e$-mod where $e\tilde{K}e$ is a certain subquotient of the Khovanov algebra $K$ of type B. We will later refine this equivalence in \cref{mainthm}. This equivalence is however not monoidal. That means that we have no direct analogue of $V^{\otimes d}$ on the Khovanov algebra side.
The key idea to overcome this problem is to look at the endofunctor $\_\otimes V=\bigoplus_{i\in\bbZ+\frac{\delta}{2}}\theta_i$ and find an endofunctor on the Khovanov side, which identifies with $\_\otimes V$ under $\Psi$.

This approach was also successfully taken by Brundan and Stroppel for $\gl[m][n]$ and the Khovanov algebra of type $A$ in \cite{BS2} and \cite{BS4}. They defined certain \emph{geometric bimodules} $K^t_{\Lambda\Gamma}$ and proved that tensoring with these actually corresponds to $\_\otimes V$ for $\gl[m][n]$.

We follow their ideas and adapt the definitions to the type $B$ setting. Many properties of $K$ and the geometric bimodules carry over immediately from type A with the same proof; and in such a case we simply refer to \cite{BS2} for the statements and proofs. A full account of the theory of geometric bimodules of type $B$ can be found also in \cite{Nehme}. We will look at two different versions of geometric bimodules.  Ehrig and Stroppel proved in \cite{OSPII}*{Theorem 6.22} that $K$ is related to Brauer algebras and we will see that tensoring with this geometric bimodules then corresponds to the $i$-induction from \cref{iinddef} on the Brauer category, for the precise statement consider \cref{blackboxformainthm}. 
However, $\cF$ is only equivalent to a subquotient of $K$, called $e\tilde{K}e$ here (see \cite{OSPII}*{Theorem 10.5} or \cref{mainthm}). So we also define geometric bimodules for $e\tilde{K}e$. But in this case, even though the statements are very similar to \cite{BS2}*{Sections 3--4}, the proofs differ markedly. In \cref{mainthm} we will see that tensoring with these geometric bimodules translates to $i$-translation from \cref{defiitranslation}.

\begin{defi}
	The Khovanov algebra $K$ is the graded associative algebra with underlying basis given by all oriented circle diagrams $a\lambda b$, where $a\lambda b$ is homogeneous of degree $\deg(a\lambda b)$.
	The multiplication $(a\lambda b)(c\mu d)$ is defined to be $0$ whenever $b^*\neq c$ and if $b^*=c$ we draw the circle diagram $(a\lambda b)$ under the circle diagram $(b^*\mu d)$, where we connect the rays of $b$ and $b^*$ and apply a certain surgery procedure (see e.g. \cite[Section 5]{ES2}\cite{ES3}). All these surgery procedures take a cup-cap-pair and replace it by straight lines. After every cup-cap-pair is removed, one collapses the middle section and defines to tbe $(a\lambda b)(c \mu d)$.
	
%
%
\end{defi}

%

Drawing $b^*\mu d$ on top of $a\lambda b$ gives a so called \emph{oriented stacked circle diagram of height $2$}. This can be generalized to arbitrary height by stacking more compatible diagrams (for details we refer to \cite{ES2}*{Section 5.1}). We give the vertices the coordinate $(x, l-1)$ if it appears in the $l$-th diagram at position $x$ for $l\in\bbZ_{>0}$ and $x\in L$. 
Note that in an oriented stacked circle diagram the positions of $\circ$ and $\times$ in each of the weight diagrams agree.

A \emph{tag} of a stacked circle diagram associates to each circle $C$ a rightmost vertex $t(C)$, i.e.~a vertex such that the horizontal coordinate is maximal among all vertices $C$.
Given an orientable stacked circle diagram $D$, a tag $t$ and a coordinate $(x,l)$ such that the connected component of $(x,l)$ in $D$ is a circle, we define 
\begin{equation}
	\mathrm{sign}_D(i,l) = (-1)^{\#\{j\mid\gamma_j\text{ is a dotted cup/cap}\}},
\end{equation}
where $\gamma_1, \dots, \gamma_t$ is a sequence of arcs in $D$ such that their concatenation is a path from $(i,l)$ to $t(C)$. 
This sign is actually independent of the chosen tag $t$ and the sequence of arcs. A proof of this can be found in \cite{ES2}*{Lemma 5.7}. 

\begin{defi}\label{defianticlock}
	A circle $C$ in an oriented stacked circle diagram is oriented \emph{clockwise} if the symbol at $t(C)$ is $\vee$ and \emph{anticlockwise} if it is $\wedge$. A line is always oriented \emph{anticlockwise} by convention.
\end{defi}

For any Deligne weight diagram $\lambda$ the circle diagram $e_{\lambda}\coloneqq\underline{\lambda}\lambda\overline{\lambda}$ is an idempotent in $K$ and $e_\lambda e_\mu=0$ whenever $\lambda\neq \mu$. This gives the algebra $K=\bigoplus_{\lambda, \mu\in\Lambda_\delta}e_\lambda Ke_\mu$ the structure of a locally unital algebra. By $\Mod_{lf}(K)$ we refer to locally finite dimensional graded modules over $K$, i.e.~graded modules $M$ such that $\dim e_\lambda M<\infty$ for all $\lambda\in\Lambda_\delta$.

The irreducible locally finite dimensional $K$-modules are in bijection with $\Lambda_\delta$. Given $\lambda\in\Lambda_\delta$ we construct a one dimensional irreducible $K$-module $L(\lambda)$ as follows. As a vector space it is just $\bbC$ and $e_\mu$ acts by $1$ if $\lambda=\mu$ and $0$ otherwise.
The indecomposable projective objects in $\Mod_{lf}(K)$ are given by $P(\lambda)\coloneqq Ke_\lambda$ for $\lambda\in\Lambda_\delta$.

We have an anti-involution $*$ on $K$ which is given by sending $a\lambda b$ to $b^*\lambda a^*$. And this gives rise to a duality (also denoted $*$) on $\Mod_{lf}(K)$. For a locally finite dimensional graded $K$-module $M$, we define the graded piece $(M^{\dual})_j\coloneqq\Hom_{\bbC}(M_{-j}, \bbC)$ and $x\in K$ acts on $f\in M^{\dual}$ by $(xf)(m)\coloneqq f(x^*m)$. We also easily see that $L(\lambda)^{\dual}=L(\lambda)$.

The indecomposable injective objects are then $I_\delta(\lambda)\coloneqq (P(\lambda))^{\dual}$ for $\lambda\in\Lambda_\delta$.

Furthermore, we define \emph{standard modules} $V(\mu)$ for $\mu\in\Lambda_\delta$. These are the cell modules associated to the cellular structure (in the sense of \cite{GL}) of $K$ in \cite{ES2}*{Theorem 7.1}. As a vector space it has a basis given by formal symbols $(\underline{\gamma}\mu|$  for all $\gamma\in\Lambda_\delta$ such that $\underline{\gamma}\mu$ is oriented. The multiplication is defined as \begin{equation}\label{deficellmodule}
	(a\lambda b)(\underline{\gamma}\mu|=\begin{cases}
	s_{a\lambda b}(\mu)(a\mu|&\text{if $b\neq\overline{\gamma}$ and $a\mu$ is oriented,}\\
	0&\text{otherwise,}
\end{cases}
\end{equation}where $s_{a\lambda b}(\mu)$ is either the coefficient from \cite{ES2}*{Thm. 7.1} or as in \cite[Theorem 3.1]{BS2}.
The standard module $V(\mu)$ is also the quotient of $P(\mu)$ and the $K$-submodule generated by all oriented circle diagrams $a\lambda\overline{\mu}$ with $\lambda\neq\mu$ (then we necessarily have $\lambda>\mu$).
The irreducible module $L(\mu)$ is the quotient of $V(\mu)$ and the $K$-submodule generated by all $(\underline{\gamma}\mu|$ with $\gamma\neq\mu$ (and hence $\gamma>\mu$).


\begin{thm} \cite[Theorem 4.5]{HNS} \label{upperfinitehighweight}
	The category $\Mod_{lf}(K)$ is an upper finite highest weight category in the sense of \cite{BS21} with standard objects $V(\lambda)$, $\lambda\in\Lambda_\delta$.
\end{thm}
%

\subsection{Geometric bimodules}
In this section we generalize the diagrammatics of Khovanov's arc algebra by incorporating \emph{crossingless matchings} (of type B). This section proves furthermore analogous results to \cite{BS2}*{Sections 2--4} and many ideas from the proofs there can be directly applied to our setting.

A \emph{crossingless matching} is a diagram $t$, which is obtained by drawing an admissible cap diagram $c$ underneath an admissible cup diagram $d$ and connecting the rays in $c$ to the rays in $d$ from left to right. This means that we allow dotted cups, caps and lines but each dot necessarily needs to be able to be connected to the left boundary without crossing anything, just as in the case of admissible circle diagrams (see \cite{ES2}*{Def. 3.5}). Furthermore we delete pairs of dots on each segment, such that each line segment contains at most one dot. Any crossingless matching is a union of (dotted) cups, caps and line segments, for example:
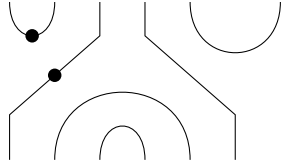
\begin{center}
	\begin{tikzpicture}[scale=0.6]
		\cups{1 2 d, 5 7}
		\caps{2 5, 3 4}
		\rays{1 3 d, 6 4}
	\end{tikzpicture}
\end{center}
We denote by $\cu(t)$ respectively $\ca(t)$ the number of cups respectively caps in $t$. Furthermore let $t^*$ be the horizontally reflected image of $t$.

We say that $t$ is a \emph{$\Lambda\Gamma$-matching} if the bottom and top number lines of $t$ agree with the number lines of $\Lambda$ respectively $\Gamma$.
More generally, given a sequence of blocks $\LAMBDA = \Lambda_k\dots\Lambda_0$, we define a \emph{$\LAMBDA$-matching} to be a diagram $\t = t_k\dots t_1$ obtained by glueing a sequence $t_1, \dots, t_k$ of crossingless matchings together from top to bottom such that
\begin{itemize}
	\item each $t_i$ is a $\Lambda_i\Lambda_{i-1}$-matching for each $i=1, \dots, k$,
	\item the free vertices at the bottom of $t_i$ are in the same position with the free vertices at the top of $t_{i+1}$ for $i=1, \dots, k-1$.
\end{itemize}
Given additionally a cup diagram $a$ and a cap diagram $b$ such that their number lines agree with the bottom number lines of $t_k$ respectively the top number line of $t_1$, we can glue them together and obtain a \emph{$\Lambda$-circle diagram} $a\t b = at_k\dots t_1b$.

Let $\Lambda$ and $\Gamma$ be blocks and let $t$ be a $\Lambda\Gamma$-matching. Given weights $\lambda\in\Lambda$ and $\mu\in\Gamma$ we can glue these together from bottom to top to obtain a new diagram $\lambda t\mu$. We call this an \emph{oriented $\Lambda\Gamma$-matching} if
\begin{itemize}
	\item each pair of vertices lying on the same dotted cup or the same undotted line segment is labeled such that both are either $\vee$ or both are $\wedge$,
	\item each pair of vertices lying on the same undotted cup or the same dotted line segment is labeled such that one is $\vee$ and one is $\wedge$,
	\item all other vertices are labeled $\circ$ or $\times$.
\end{itemize}
A diamond $\diamond$ can be interpreted as either $\vee$ or $\wedge$.

More generally an \emph{oriented $\LAMBDA$-matching} for a sequence of blocks $\LAMBDA=\Lambda_k\dots\Lambda_0$ is a composite diagram of the form 
\[
\tl = \lambda_kt_k\lambda_{k-1}\dots\lambda_1t_1\lambda_0
\]
where $\lam=\lambda_k\dots \lambda_0$ is a sequence of weights such that $\lambda_it_i\lambda_{i-1}$ is an oriented $\Lambda_i\Lambda_{i-1}$-matching for each $i=1, \dots k$.

Finally given an oriented $\LAMBDA$-matching and cap and cup diagrams $a$ and $b$ such that $a\lambda_k$ (resp.~$\lambda_0b$) is an oriented cup (resp cap) diagram we can glue these together to obtain an \emph{oriented $\LAMBDA$-circle diagram} $a\tl b$.

We call a $\LAMBDA$-matching $\t$ \emph{proper} if there exists at least one oriented $\LAMBDA$-matching for $\t$.
By a rightmost vertex $x$ on a circle $C$ we mean a vertex lying on $C$ such that on this numberline, there is no vertex to the right of $x$. In the bottom picture every rightmost vertex is marked by $x$. 
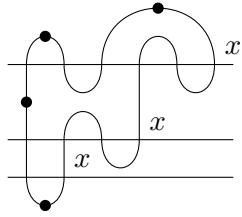
\begin{center}
	\begin{tikzpicture}[scale=0.5]
		\caps{0 1 d, 2 5 d, 3 4}\wdiag{- - - - - -}\node[anchor=south west] at (5, \currh) {$x$};\cups{1 2, 4 5}\caps{1 2}\rays{0 0 d, 3 3}\wdiag{- - - - - -}\node[anchor=south west] at (3, \currh) {$x$};\cups{2 3}\rays{0 0, 1 1}\wdiag{- - - - - -}\node[anchor=south west] at (1, \currh) {$x$};\cups{0 1 d}
	\end{tikzpicture}
\end{center}
We refer to a circle in an oriented $\LAMBDA$-diagram as \emph{clockwise} respectively \emph{anticlockwise} if a rightmost vertex on the circle is labeled $\vee$ respectively $\wedge$. It can be checked similarly to \cite[Corollary 5.9]{ES2} that this notion is well-defined.

\begin{lem}\label{rightmostvertex}
	Let $a\tl b$ be an oriented $\LAMBDA$-circle diagram and let $C$ be a closed component of this diagram. Then the rightmost vertices of $C$ all have the same orientation.
\end{lem}
\begin{proof}
	Take two rightmost vertices $x$ and $y$ in a circle $c$ and assume that $x\neq y$. Then there are exactly two paths connecting $x$ with $y$ in $C$. The crucial observation is that the ``right'' one of them is cut off by the other one from the left boundary of the diagram and thus cannot contain any dots.
	Without loss of generality assume that $y$ appears on a lower number line as the picture indicates.
	
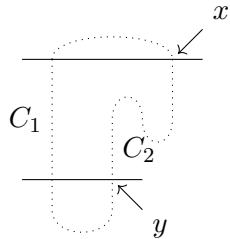
\begin{center}
		\begin{tikzpicture}[scale =0.8]
			\draw[dotted] (0,0)--(0, -2) node[midway, anchor=east] {$C_1$} (2, 0) -- (2, -1) .. controls+(0,-0.5) and+(0, -0.5).. (1.5, -1) ..controls+(0, 0.5)and+(0,0.5).. (1, -1) -- (1, -2) node[midway, anchor=west] {$C_2$};
			\draw (-0.5,0) -- (2.5,0) (-0.5, -2) -- (1.5, -2);
			\draw[dotted] (0,0) .. controls +(0,0.5) and +(0,0.5) .. (2,0);
			\draw[dotted] (0,-2) -- (0,-2.5) .. controls +(0,-0.5) and +(0,-0.5) .. (1,-2.5) -- (1, -2);
			\draw[->, shorten >=3pt] (2.5,0.5) node[anchor = south west] {$x$} -- (2,0);
			\draw[->, shorten >=3pt] (1.5,-2.5) node[anchor=north west] {$y$} -- (1,-2);
		\end{tikzpicture}
	\end{center}
	One of the paths leaves the vertex $y$ to the top ($C_2$) and one to the bottom ($C_1$). Note that $C_1$ has to cross the number line of $y$ again but by our assumption this happens to the left of $y$. Then this paths always stays to the left of $C_2$, hence $C_2$ is ``cut off'' by $C_1$. So $C_2$ cannot contain any dots as otherwise those could not be connected to the left boundary (contradicting the admissiblity assumption in the definition of crossingless matching). By a similar reasoning $C_2$ is also the path which enters $x$ from the bottom and hence $C_2$ has to contain an even number of cups which are all undotted. So the symbol ($\vee$ or $\wedge$) gets changed an even number of times, when moving from $y$ to $x$ along $C_2$ and thus the orientations agree.
\end{proof}

\begin{defi}
	The \emph{degree} of a circle or a line in an oriented $\LAMBDA$-circle diagram is the total number of clockwise cups or caps that it contains. The \emph{degree} of an oriented $\LAMBDA$-circle diagram is the sum of the degrees of each of its circles and lines.
	We call a circle only consisting of one cup and one cap a \emph{small circle}.
\end{defi}

The following lemma can be verified similar to \cite[Proposition 1.2.12, Proposition 1.2.13]{ES2}. 

\begin{lem}\label{onemoreoneless}
	The degree of an anticlockwise circle in an oriented $\LAMBDA$-circle diagram is one less than the total number of caps (equivalently, cups) that it contains. The degree of a clockwise circle is one more than the total number of caps (equivalently, cups) that it contains. The degree of a line is equal to the number of caps or the number of cups that it contains, whichever is greater.
\end{lem}

\begin{defi}
	Suppose we have a $\LAMBDA$-matching $\t=t_k\dots t_1$ for some sequence $\LAMBDA=\Lambda_k\dots\Lambda_0$ of blocks. We refer to circles in $\t$ not meeting the top or bottom number line as \emph{internal} circles. The \emph{reduction} of $\t$ is the $\Lambda_k\Lambda_0$-matching which is obtained by removing all internal circles, all but the top and bottom number line and maintaining the parity of dots on each component. 
\end{defi}


\begin{lem}\label{degreeformula}
	Assume that $a\tl b$ is an oriented $\LAMBDA$-circle diagram for some sequence $\lam = \lambda_k\dots\lambda_0$ of weights. Let $u$ be the reduction of $\t$. Then $a\lambda_ku\lambda_0b$ is an oriented $\Lambda_k\Lambda_0$-circle diagram and
	\begin{align*}
		\deg(a\tl b)&=\deg(a\lambda_ku\lambda_0b)+\ca(t_1)+\dots+\ca(t_k)-\ca(u)+p-q\\
		&=\deg(a\lambda_ku\lambda_0b)+\cu(t_1)+\dots+\cu(t_k)-\cu(u)+p-q,
	\end{align*}
	where $p$ (resp.~$q$) denotes the number of internal circles of $\t$ that are clockwise (resp.~anticlockwise) in the diagram $a\tl b$.
\end{lem}

\begin{proof} See \cite[Lemma 2.3]{BS2}.
\end{proof}

\begin{defi}\label{defiuplowreduction}
	Let $t$ be a $\Lambda\Gamma$-matching for some blocks $\Lambda$ and $\Gamma$. Let $a$ be a cup diagram such that its number line agrees with the bottom one of $t$. We refer to circles or lines not meeting the top number line in $at$ as \emph{lower} circles or lines. The \emph{lower reduction} of $at$ refers to the cup diagram which is obtained by removing all lower circles and lines as well as the bottom number line.
	
	Similarly if $b$ is a cap diagram whose number line agrees with the top one of $t$, we call each circle or line not meeting the bottom number line \emph{upper} circle or line. Similarly the \emph{upper reduction} of $bt$ means removing all upper lines or circles and the top number line.
\end{defi}


\begin{lem}\label{degreeformulaupperreduction} (\cite[Lemma 5.5]{HNS})
	If $a\lambda t\mu b$ is an oriented $\Lambda\Gamma$-circle diagram and $c$ is the lower reduction of $at$, then $c\mu b$ is an oriented circle diagram and
	\[
	\deg(a\lambda t\mu b) = \deg(c\mu b)+\ca(t) + p-q,
	\]
	where $p$ (resp.~$q$) is the number of lower circles that are clockwise (resp.~anticlockwise) in the diagram $a\lambda t\mu b$. 
	For the dual statement about upper reduction one needs to replace $\ca(t)$ by $\cu(t)$.
\end{lem}


\begin{defi}
	Let $\LAMBDA=\Lambda_k\dots\Lambda_0$ be a sequence of blocks, and let $\t=t_k\dots t_1$ be a $\LAMBDA$-matching. Define $K^{\t}_{\LAMBDA}$ to be the graded vector space with homogeneous basis
	\begin{equation*}
	\{(a\tl b)\mid\text{ for all closed oriented $\LAMBDA$-circle diagrams } a\tl b\}.
	\end{equation*}
	Define a degree preserving linear map
	\begin{equation}\label{definvolution}
		*\colon K^{\t}_{\LAMBDA}\to K^{\t^*}_{\LAMBDA^*}, \quad (a\tl b)\mapsto (b^*\t^*[\lam^*]a^*),
	\end{equation}
	where $\LAMBDA^*=\Lambda_0\dots\Lambda_k$, $\lam^*=\lambda_0\dots\lambda_k$, $t^*=t_1^*\dots t_k^*$ and $t_i^*$, $a^*$ and $b^*$ denote the mirror images of $t_i$, $a$, $b$ in the horizontal axis.
\end{defi}


Let $\GAMMA=\Gamma_l\dots\Gamma_0$ be another sequence of blocks with $\Lambda_0=\Gamma_l$. We denote by $\LAMBDA\bm{\wr\Gamma}$ the block sequence $\Lambda_k\dots\Lambda_1\Gamma_l\dots\Gamma_0$. Observe that one copy of $\Lambda_0$ is left out in comparison to the concatenation of the block sequences. Furthermore note that if $\u=u_l\dots u_1$ is a $\GAMMA$-matching the concatenation $\t\u=t_k\dots t_1u_l\dots u_1$ is a $\LAMBDA\bm{\wr\Gamma}$-matching. We then define a degree preserving linear multiplication
\begin{equation}\label{defmultiplication}
	m\colon K^{\t}_{\LAMBDA}\otimes K^{\u}_{\GAMMA}\to K^{\t\u}_{\LAMBDA\bm{\wr\Gamma}}
\end{equation}
as follows. The product $(a\tl b)(c\u[\MU]d)$ is defined to be $0$ whenever $b\neq c^*$. In the case $b=c^*$ we draw $(a\tl b)$ underneath $(c\u[\MU]d)$ and we then smooth out the symmetric middle section using surgery procedures exactly as in the Khovanov algebra $K$ of type $B$. Then we collapse the middle section by identifying the number lines adjacent to the middle section and declaring the product to be this sum of oriented $\LAMBDA\bm{\wr\Gamma}$-circle diagrams. That this is well-defined and homogeneous of degree 0 can be verified in the same manner as in \cite{ES2}*{Section 5}. 

In the special case $k=l=0$ this simplifies to the ordinary multiplication in the Khovanov algebra $K$ of type $B$. Additionally, given a third sequence of blocks $\bm{\Upsilon}$ with $\Upsilon_0=\Lambda_k$, this multiplication is associative in the sense that the following diagram commutes, which again can be verified analogously to \cite{ES2}*{Section 5}:
\begin{equation}\label{assocgeombimod}
	\begin{tikzcd}
		K^{\bm{s}}_{\bm{\Upsilon}}\otimes K^{\t}_{\LAMBDA}\otimes K^{\u}_{\GAMMA}\arrow[r, "1\otimes m"]\arrow[d, "m\otimes 1"]&K^{\bm{s}}_{\bm{\Upsilon}}\otimes K^{\t\u}_{\LAMBDA\bm{\wr\Gamma}}\arrow[d, "m"]\\
		K^{\bm{s}\t}_{\bm{\Upsilon\wr}\LAMBDA}\otimes  K^{\u}_{\GAMMA}\arrow[r, "m"]& K^{\bm{s}\t\u}_{\bm{\Upsilon\wr}\LAMBDA\bm{\wr\Gamma}}
	\end{tikzcd}
\end{equation}

Finally the linear map $*$ is antimultiplicative in the sense that the following diagram commutes ($P$ denotes the flip $x\otimes y\mapsto y\otimes x$):

\begin{equation}\label{antimultgeombimod}
	\begin{tikzcd}
		K^{\t}_{\LAMBDA}\otimes K^{\u}_{\GAMMA}\arrow[r, "P\circ(*\otimes *)"]\arrow[d, "m"]& K^{\u^*}_{\GAMMA^*}\otimes K^{\t^*}_{\LAMBDA^*}\arrow[d, "m"]\\
		K^{\t\u}_{\LAMBDA\bm{\wr\Gamma}}\arrow[r, "*"]&K^{\u^* \t^*}_{\GAMMA^*\bm{\wr\Lambda}^*}
	\end{tikzcd}
\end{equation}

\begin{rem}
	Letting $\bm{\Upsilon}=\Lambda_k$ and $\GAMMA=\Lambda_0$ we see by \eqref{assocgeombimod} that the multiplication $m$ turns $K^{\t}_{\LAMBDA}$ into a $(K_{\Lambda_k},K_{\Lambda_0})$-bimodule.
	
	Recalling the primitive idempotents $e_\alpha\in K_{\Lambda_k}$ and $e_\beta\in K_{\Lambda_0}$, we have that
	\begin{align}
		e_\alpha(a\tl b) = \begin{cases}
			(a\tl b)&\text{if } \underline{\alpha}=a,\\
			0&\text{otherwise},
		\end{cases}\\
		(a\tl b)e_\beta = \begin{cases}
			(a\tl b)&\text{if } \overline{\beta}=b,\\
			0&\text{otherwise}.
		\end{cases}
	\end{align}
\end{rem}

Using these definitions many statements from \cite{BS2} and \cite{BS5} carry over verbatim, namely \cite[Theorem 3.1, Corollary 3.2 and 3.3, Theorem 3.5 and 3.6]{BS2}, see also \cite{Nehme} for a detailed verification in our type B case. In particular, the analog of \cite[Theorem 3.6]{BS2} allows us to reduce the study of bimodules $K^{\t}_{\LAMBDA}$ for arbitrary sequences $\LAMBDA=\Lambda_k\dots\Lambda_0$ and $\t=t_k\dots t_1$ to the bimodules $K^t_{\Lambda\Gamma}$ for a single $\Lambda\Gamma$-matching $t$, markedly simplifying our notation. In other words, in order to understand $K^{\t}_{\LAMBDA}$ as a bimodule, it actually suffices to understand the bimodule $K^u_{\Lambda_k\Lambda_0}$ instead. This justifies why we are restricting ourselves to the latter case in the following section.

\subsection{Projective functors}

In this section we develop the theory of projective functors and compute their effect on simple, standard and projective modules.  

\begin{defi}\label{defiprojfunctor}
	Let $t$ be a proper $\Lambda\Gamma$-matching. Define the functor
	\begin{equation*}
		G^t_{\Lambda\Gamma}\coloneqq K^t_{\Lambda\Gamma}\langle-\ca(t)\rangle\otimes\_\colon \Mod_{lf}(K_\Gamma)\to \Mod_{lf}(K_\Lambda).
	\end{equation*}
	We call any functor which is isomorphic to a finite direct sum of the above functors (possibly shifted) a \emph{projective functor}.
\end{defi}

\begin{rem}
	The degree shift in the definition ensures that $G^t_{\Lambda\Gamma}$ commutes with duality, see \cite[Theorem 4.10]{BS2} (which will hold in type B as well). 
	Furthermore \cite[Theorem 3.5]{BS2} \cite[Theorem 3.6]{BS2} imply that the composition of projective functors is again projective.
\end{rem}

In most statements the additional dots do not play a role and we can therefore apply the theory of \cite{BS2}. It can then be immediately verified that \cite[Lemma 4.1]{BS2}, \cite[Theorem 4.2]{BS2} (noting that the map $f\colon K^t_{\Lambda\Gamma}e_{\gamma}\to K_\Lambda e_\lambda\otimes R^{\otimes n}$ is $K_\Lambda$-linear as every tag gets altered by an even number of undotted arcs (see \cref{rightmostvertex})), \cite[Corollary 4.3]{BS2} and \cite[Theorem 4.5]{BS2} (see also \cite[Theorem 5.11]{HNS}) and their proofs carry over verbatim. In particular we have a description of the grading filtration of the $K_\Lambda$-module $G^t_{\Lambda\Gamma}V(\gamma)$.

We would like to analyse the effect of projective functors on irreducible modules. For this we are interested in proving that the projective functors $G^t_{\Lambda\Gamma}$ and $G^{t^*}_{\Gamma\Lambda}$ form up to degree shift an adjoint pair (as in \cite{BS2}*{Section 4} for type $A$), so that we can understand the composition factors of $G^t_{\Lambda\Gamma}L(\gamma)$ in terms of $G^{t^*}_{\Gamma\Lambda}P(\mu)$. For this we define a linear map
\begin{equation}
	\phi\colon K^{t^*}_{\Gamma\Lambda}\otimes K^t_{\Lambda\Gamma}\to K_\Gamma
\end{equation}
by declaring that $\phi\coloneqq 0$ if $t$ is not a proper $\Lambda\Gamma$-matching. If $t$ is proper and given basis vectors $(a\lambda t^*\nu d)\in K^{t^*}_{\Gamma\Lambda}$ and $(d'\kappa t\mu b)\in K^t_{\Lambda\Gamma}$, we denote by $c$ the upper reduction of $t^*d$. Then if $d'=d^*$ and all mirror image pairs of upper respectively lower circles in $t^*d$ respectively $d^*t$ are oriented in \emph{opposite} ways in the corresponding basis vectors, we set
\begin{equation}\label{defiphi}
	\phi((a\lambda t^*\nu d)\otimes(d'\kappa t\mu b)) \coloneqq \pm(a\lambda c)(c^*\mu b)
\end{equation}
and otherwise we set $	\phi((a\lambda t^*\nu d)\otimes(d'\kappa t\mu b))\coloneqq0$. 
The sign in \eqref{defiphi} depends only on $t$ and $d$ and is defined inductively by the argument given in the next proof, i.e.~by the induction argument in the next proof one can reconstruct the sign for each $t$ and $d$.

\begin{lem}\label{nodegbilinformadjunction}
	The map $\phi\colon K^{t^*}_{\Gamma\Lambda}\otimes K^t_{\Lambda\Gamma}\to K_\Gamma$ is a homogeneous $(K_\Gamma, K_\Gamma)$-bimodule homomorphism of degree $-2\ca(t)$. Moreover it is $K_\Lambda$-balanced and thus induces a map $\overline{\phi}\colon K^{t^*}_{\Gamma\Lambda}\otimes_{K_\Lambda} K^t_{\Lambda\Gamma}\to K_\Gamma$.
\end{lem}

\begin{rem} This lemma is analogous to \cite[Lemma 4.6]{BS2} but considerably harder to prove due to the appearance of additional sign factors.
\end{rem}

\begin{proof}
	If $t$ is not a proper $\Lambda\Gamma$-matching the claim is trivial, thus we assume in the following that $t$ is proper.
	
	First of all, we are going to show that $\phi$ is homogeneous of degree $-2\ca(t)$. For this take again basis vectors as in \eqref{defiphi} (in every other case $\phi$ is $0$ by definition). Suppose that $p$ (resp.~$q$) of the upper circles in $t^*d$ are oriented clockwise (resp.~anticlockwise) in $a\lambda t^*\nu d$. Then by our assumptions on the basis vectors $q$ (resp.~$p$) of the lower circles in $d^*t$ are oriented clockwise (resp.~anticlockwise) in $d^*\kappa t\mu b$. By \cref{degreeformulaupperreduction} we have 
	\begin{align*}
		\deg(a\lambda t^*\nu d) &=\deg(a\lambda c) + \cu(t^*) + p-q\;\text{ and}\\
		\deg(d^*\kappa t\mu b) &=\deg(c^*\mu b) + \ca(t) + q-p.
	\end{align*}
	By definition of $t^*$, we have $\cu(t^*)=\ca(t)$, thus 
	\begin{equation}
		\deg((a\lambda t^*\nu d)\otimes(d^*\kappa t\mu b))=\deg((a\lambda c)(c^*\mu b))+2\ca(t).
	\end{equation}
	
	Secondly, the map $\phi$ is a left $K_\Gamma$-homomorphism as in the proof of \cite[Theorem 4.2]{BS2}, which showed that mapping $(a\lambda t^*\nu d)$ to $(a\lambda c)$ is a left $K_\Gamma$-homomorphism, and one argues similarly for the right action.
	
	Lastly we are going to prove that $\phi$ is $K_\Lambda$-balanced. For this we introduce the map
	\begin{equation*}
		\omega\colon K^{t^*t}_{\Gamma\Lambda\Gamma}\to K_\Gamma
	\end{equation*} as follows. Take a basis vector $(a\lambda t^*\mu t\nu b)\in K^{t^*t}_{\Gamma\Lambda\Gamma}$. If any of its internal circles in the diagram $t^*t$ are oriented anticlockwise, we declare that its image is $0$. Otherwise we define $u$ to be the reduction of $t^*t$ and consider the diagram $a\lambda u\nu b$. This contains a symmetric middle section as $u$ was the reduction of the symmetric diagram $t^*t$, so it makes sense to apply the surgery procedure to smooth this section out and obtain a linear combination of basis vectors of $K_\Gamma$. We define the image of $(a\lambda t^*\mu t\nu b)$ to be this linear combination. We claim that 
	\begin{equation}\label{phimultclaimtoproof}
		\phi=\omega\circ m,
	\end{equation} where $m$ is the multiplication map $m\colon K^{t^*}_{\Gamma\Lambda}\otimes K^{t}_{\Lambda\Gamma}\to K^{t^*t}_{\Gamma\Lambda\Gamma}$ from \eqref{defmultiplication}. As we know that $m$ is $K_\Lambda$-balanced (by associativity), this shows that $\phi$ is $K_\Lambda$-balanced.

	In some sense, we are trying to prove that first reducing and then multiplying ($\phi$) is ``the same as'' first multiplying and then reducing ($\omega\circ m$). 
	The general idea of the proof is to replace $t^*d$ by some easier $t_1^*d_1$ (for which we know the claim) such that both have the same upper reduction, and then trying to show that $\omega((a\lambda t^*\nu d)(d^*\kappa t\mu b))=\omega((a\lambda t_1^*\nu_1d_1)(d_1^*\kappa_1t_1\mu b))$. But in general, the above equality holds only up to sign and that is why we incorporated a sign in the definition of $\phi$.

	To prove the claim, we proceed by induction on $\ca(t)$. If $\ca(t)=0$, then there are neither upper circles in $t^*d$ nor internal circles in $t^*t$. Thus in this case applying the upper reduction to $t^*d$ gives a bijection between the caps in $t^*d$ and the caps in $c$, which is just given by reducing straight lines in $t^*d$. Hence the signs involved in surgery procedures will exactly be the same.
	Computing $\phi(a\lambda t^*\nu d)\otimes (d^*\kappa t\mu b)=(a\lambda c)(c^*\mu b)$ means that every cap in $c$ gets eliminated by surgeries. On the other hand applying $m$ eliminates each cap in $d$ and $\omega$ eliminates then the remaining caps in $t^*$. Thus by the above comment, the results are the same.
	
	For the induction step assume $\ca(t)>0$ and that \eqref{phimultclaimtoproof} is proven for all smaller cases. We will consider five different cases depending on certain subpictures of $t^*d$, the last one being the general case.
	
	\emph{Case 1:} Suppose  that $t^*d$ contains a small circle, i.e.~a circle consisting of only one cap and cup. If this circle in $t^*\nu d$ and its mirror image in $d^*\kappa t$ are oriented in the same way, then $\phi$ gives $0$ by definition. On the other hand $(a\lambda t^*\nu d)(d^*\kappa t\mu b)$ produces $0$ if both are oriented clockwise and the product produces an anticlockwise circle if both were oriented anticlockwise, but in this case $\omega$ produces $0$. Thus we may assume that these two circles are oriented in opposite ways in $t^*\nu d$ and $d^*\kappa t$. Now we can remove these two circles (and the vertices involved) to obtain diagrams $a\lambda t_1^*\nu_1d_1$ and $d_1^*\kappa_1t_1\mu b$ with $\ca(t_1)<\ca(t)$. Using the definitions, one can easily verify that
	\begin{align*}
		\phi((a\lambda t^*\nu d)\otimes (d^*\kappa t\mu b)) &= \phi((a\lambda t_1^*\nu_1d_1)\otimes(d_1^*\kappa_1t_1\mu b)),\\
		\omega((a\lambda t^*\nu d)(d^*\kappa t\mu b)) &= \omega((a\lambda t_1^*\nu_1d_1)(d_1^*\kappa_1t_1\mu b)).
	\end{align*}
	The first equality holds because the small circle is removed in the process of upper reduction, and thus it does not matter whether we remove it in the process of upper reduction or whether we remove it first and do the upper reduction after that.
	The second equality holds as merging the two small circles in a surgery for the left hand side produces exactly one small clockwise circle with no further signs, which then gets removed by $\omega$. But these circles play no role for the other surgeries, hence it agrees with the right hand side.
	Using the induction hypothesis the right hand sides coincide, thus the left hand sides agree as well.
	
	\emph{Case 2:} Suppose that $t^*d$ contains an upper line containing only one cup. Denote by $a\lambda t_1^*\nu_1 d_1$ and $d_1^*\kappa_1t_1\mu b$ the diagrams where this upper line and its mirror image in $d^*t$ get removed. When computing the product $(a\lambda t^*\nu d)(d^*\kappa t\mu b)$ one can apply the same surgery procedures as for $(a\lambda t_1^*\nu_1 d)(d^*\kappa_1t_1\mu b)$. There is no further surgery needed as the upper line contains only one cup. Now notice that, when drawing $(a\lambda t^*\nu d)$ underneath $(d^*\kappa t\mu b)$ the upper line and its mirror image form a clockwise circle. This is not changed throughout the whole surgery procedure and $(a\lambda t^*\nu d)(d^*\kappa t\mu b)$ and $(a\lambda t_1^*\nu_1 d)(d_1^*\kappa_1t_1\mu b)$ differ only by this clockwise circle, which is removed when applying $\omega$. 
	And as both of them clearly have the same upper reduction (upper lines get removed in this process) we get
	
	\begin{align*}
				\phi((a\lambda t^*\nu d)\otimes (d^*\kappa t\mu b)) &= \phi((a\lambda t_1^*\nu_1d_1)\otimes(d_1^*\kappa_1t_1\mu b))=
				\omega((a\lambda t_1^*\nu_1d_1)(d_1^*\kappa_1t_1\mu b))\\&=\omega((a\lambda t^*\nu d)(d^*\kappa t\mu b)).
		\end{align*}

%
	\emph{Case 3:} Suppose then that $t^*d$ contains one of the following local pictures on the top number line of $t^*d$ with a mirror image in $d^*t$. A dashed dot means that there may be a dot present and different dashing patterns correspond to different choices whether a dot is present or not. In any case, the parity of the number of dots stays the same.
	\begin{center}
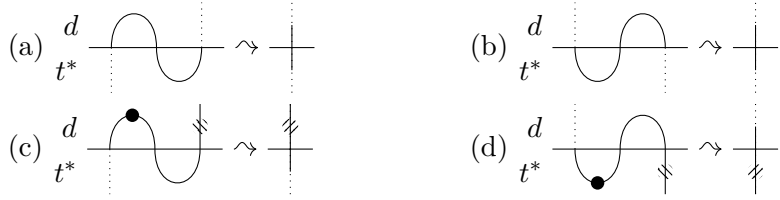
		\begin{tabular}{m{0.4\textwidth}m{0.4\textwidth}}
			\begin{tikzpicture}[scale=0.6]
				\node[anchor=east] at (-0.5, 0.5) {$d$}; \node[anchor=east] at (-0.5, -0.5) {$t^*$};
				\caps{0 1}\draw[dotted] (0,{\currh}) --(0,{\currh-1}) (2, {\currh})--(2, {\currh+1});\draw (-0.5, {\currh}) -- (2.5, {\currh});\draw[->, snake=snake] (2.75, \currh) -- (3.25, \currh);\node[anchor=east] at (-1.25, \currh) {(a)};
				\cups{1 2}
				\begin{scope}[xshift=4cm]
					\draw (0,-0.5) -- (0,0.5) (-0.5,0) -- (0.5,0);
					\draw[dotted] (0, -1) -- (0,1);
				\end{scope}
			\end{tikzpicture}&
			\begin{tikzpicture}[scale=0.6]\node[anchor=east] at (-0.5, 0.5) {$d$}; \node[anchor=east] at (-0.5, -0.5) {$t^*$};\caps{1 2}\draw[dotted] (0,{\currh}) --(0,{\currh+1}) (2, {\currh})--(2, {\currh-1});\draw (-0.5, {\currh}) -- (2.5, {\currh});\draw[->, snake=snake] (2.75, \currh) -- (3.25, \currh);\node[anchor=east] at (-1.25, \currh) {(b)};
				\cups{0 1}
				\begin{scope}[xshift=4cm]
					\draw (0,-0.5) -- (0,0.5) (-0.5,0) -- (0.5,0);
					\draw[dotted] (0, -1) -- (0,1);
				\end{scope}
			\end{tikzpicture}\\
			\begin{tikzpicture}[scale=0.6]\node[anchor=east] at (-0.5, 0.5) {$d$}; \node[anchor=east] at (-0.5, -0.5) {$t^*$};
				\caps{0 1 d}\rays{2 i}\fill[pattern=north west lines] (2, 0.5) circle (5pt);\draw[dotted] (0,{\currh}) --(0,{\currh-1});\draw (-0.5, {\currh}) -- (2.5, {\currh});\draw[->, snake=snake] (2.75, \currh) -- (3.25, \currh);\node[anchor=east] at (-1.25, \currh) {(c)};
				\cups{1 2}
				\begin{scope}[xshift=4cm]
					\draw (0,-0.5) -- (0,1) (-0.5,0) -- (0.5,0);
					\fill[pattern=north east lines] (0, 0.5) circle (5pt);
					\draw[dotted] (0, -1) -- (0,1);
				\end{scope}
			\end{tikzpicture}&
			\begin{tikzpicture}[scale=0.6]\node[anchor=east] at (-0.5, 0.5) {$d$}; \node[anchor=east] at (-0.5, -0.5) {$t^*$};\caps{1 2}\draw[dotted] (0, {\currh})--(0, {\currh+1});\wdiag{- - -}\draw[->, snake=snake] (2.75, \currh) -- (3.25, \currh);\node[anchor=east] at (-1.25, \currh) {(d)};\cups{0 1 d}\rays{i 2}\fill[pattern=north west lines] (2, -0.5) circle (5pt);
				\begin{scope}[xshift=4cm, yscale=-1]
					\draw (0,-0.5) -- (0,1) (-0.5,0) -- (0.5,0);
					\fill[pattern=north east lines] (0, 0.5) circle (5pt);
					\draw[dotted] (0, -1) -- (0,1);
				\end{scope}
			\end{tikzpicture}
		\end{tabular}
	\end{center}
	Denote by $a\lambda t_1^*\nu_1d_1$ and $d_1^*\kappa_1t_1\mu b$ the diagrams obtained by straightening these curved lines as in the picture above. We are then again in the situation that $\ca(t_1)<\ca(t)$. 
	In order to compute $(a\lambda t^*\nu d)(d^*\kappa t\mu b)$ we can apply exactly the same surgery procedures in the same order as for $(a\lambda t_1^*\nu_1d_1)(d_1^*\kappa_1t_1\mu b)$ and apply an additional one somewhere in the middle, which involves these curved lines we straightened. \Cref{addsurgtdthree} shows this additional surgery. Note that the dashed dots in the reduction process appear directly beneath each other when multiplying, thus they appear in pairs and get removed at the beginning of the multiplication process.
	\begin{figure}[h]\centering
		\setlength{\extrarowheight}{2.7cm}
		\begin{tabular}{c||c}
			\begin{tikzpicture}[scale=0.5]
				\caps{1 2}\draw[dotted] (0, \currh) -- (0, {\currh+1});\wdiag{- - -}
				\cups{0 1}\caps{0 1}\rays{2 2}\draw[->, snake=snake] (2.75, \currm) -- (3.25,\currm);\wdiag{- - -}
				\draw[dotted] (0, \currh) -- (0, {\currh-1}); \cups{1 2}
				\begin{scope}[xshift=4cm]
					\caps{1 2}\draw[dotted] (0, \currh) -- (0, {\currh+1});\wdiag{- - -}
					\rays{0 0, 1 1, 2 2}\wdiag{- - -}
					\draw[dotted] (0, \currh) -- (0, {\currh-1}); \cups{1 2}
				\end{scope}
			\end{tikzpicture}&
			\begin{tikzpicture}[scale=0.5]
				\caps{0 1}\draw[dotted] (2, \currh) -- (2, {\currh+1});\wdiag{- - -}
				\cups{1 2}\caps{1 2}\rays{0 0}\draw[->, snake=snake] (2.75, \currm) -- (3.25,\currm);\wdiag{- - -}
				\draw[dotted] (2, \currh) -- (2, {\currh-1}); \cups{0 1}
				\begin{scope}[xshift=4cm]
					\caps{0 1}\draw[dotted] (2, \currh) -- (2, {\currh+1});\wdiag{- - -}
					\rays{0 0, 1 1, 2 2}\wdiag{- - -}
					\draw[dotted] (2, \currh) -- (2, {\currh-1}); \cups{0 1}
				\end{scope}
			\end{tikzpicture}\\\hline\hline
			\begin{tikzpicture}[scale=0.5]
				\caps{1 2}\draw[dotted] (0, \currh) -- (0, {\currh+1});\wdiag{- - -}
				\cups{0 1 d}\caps{0 1 d}\rays{2 2}\draw[->, snake=snake] (2.75, \currm) -- (3.25,\currm);\wdiag{- - -}
				\draw[dotted] (0, \currh) -- (0, {\currh-1}); \cups{1 2}
				\begin{scope}[xshift=4cm]
					\caps{1 2}\draw[dotted] (0, \currh) -- (0, {\currh+1});\wdiag{- - -}
					\rays{0 0, 1 1, 2 2}\wdiag{- - -}
					\draw[dotted] (0, \currh) -- (0, {\currh-1}); \cups{1 2}
				\end{scope}
			\end{tikzpicture}&
			\begin{tikzpicture}[scale=0.5]
				\caps{0 1 d}\draw[dotted] (2, \currh) -- (2, {\currh+1});\wdiag{- - -}
				\cups{1 2}\caps{1 2}\rays{0 0}\draw[->, snake=snake] (2.75, \currm) -- (3.25,\currm);\wdiag{- - -}
				\draw[dotted] (2, \currh) -- (2, {\currh-1}); \cups{0 1 d}
				\begin{scope}[xshift=4cm]
					\caps{0 1 d}\draw[dotted] (2, \currh) -- (2, {\currh+1});\wdiag{- - -}
					\rays{0 0, 1 1, 2 2}\wdiag{- - -}
					\draw[dotted] (2, \currh) -- (2, {\currh-1}); \cups{0 1 d}
				\end{scope}
			\end{tikzpicture}
		\end{tabular}
	\caption{Additional surgeries for $t^*d$: \emph{Case 3}}\label{addsurgtdthree}
	\end{figure}
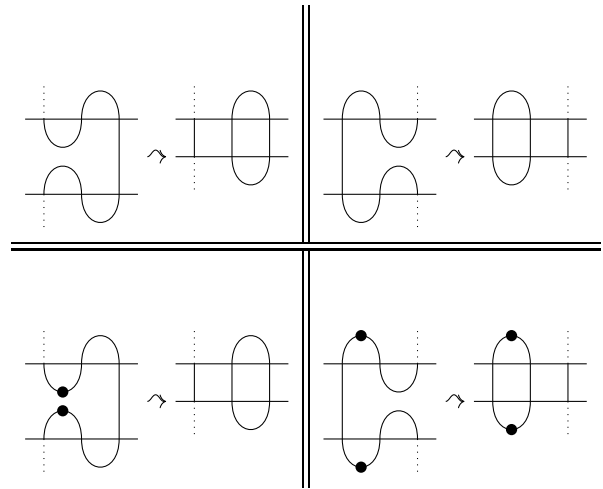
	Up to the point of the additional surgery procedure, the results of the surgeries applied so far is the same (except in the local spot that we changed).
	The additional surgery procedure is a split and produces one extra internal circle. We can concentrate on the case where this circle is oriented clockwise, as otherwise $\omega$ produces $0$. But in this case the other component is oriented in the same way as before, thus leaving ourselves only with a few possible signs. Looking at the definition of the surgery procedure \emph{Split}, one gets that the involved signs are $(-1)^{\po(i)+1}$ in all four cases. After this all the remaining surgeries produce exactly the same result (except for the additional circle produced by the split). This circle does not get altered by any other surgery and is later removed by $\omega$. Thus we have $\omega((a\lambda t^*\nu d)(d^*\kappa t\mu b))=(-1)^{\po(i)+1}\omega((a\lambda t_1^*\nu_1d_1)(d_1^*\kappa_1t_1\mu b))$.
	On the other hand $t^*d$ and $t_1^*d_1$ clearly have the same upper reduction. Defining the involved sign in the definition of $\phi$ to be exactly $(-1)^{\po(i)+1}$ times the sign associated to $t_1^*d_1$, we conclude $\phi((a\lambda t^*\nu d)\otimes (d^*\kappa t\mu b)) = \phi((a\lambda t_1^*\nu_1d_1)\otimes(d_1^*\kappa_1t_1\mu b))$, thus finishing this case.
	
	\emph{Case 4:} Suppose that $t^*d$ contains one of the following subpictures.
	\begin{center}
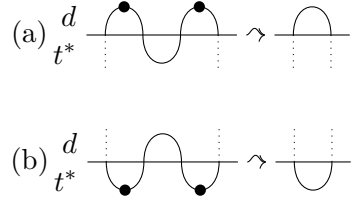
		\begin{tabular}{c}
			\begin{tikzpicture}[scale=0.5]
				\node[anchor=east] at (-0.5, 0.5) {$d$}; \node[anchor=east] at (-0.5, -0.5) {$t^*$};
				\caps{0 1 d, 2 3 d}\draw[dotted] (0, \currh) -- (0, {\currh-1}) (3, \currh)--(3, {\currh-1});\draw[->, snake=snake] (3.75,\currh) -- (4.25, \currh);\node[anchor=east] at (-1.25, \currh) {(a)};\wdiag{- - - -}
				\cups{1 2}
				\begin{scope}[xshift=5cm]
					\caps{0 1}\draw[dotted] (0, \currh) -- (0, {\currh-1}) (1, \currh) -- (1, {\currh-1});\wdiag{- -}	
				\end{scope}
			\end{tikzpicture}\\[15pt]
			\begin{tikzpicture}[yscale=-1, scale=0.5]
				\node[anchor=east] at (-0.5, 0.5) {$t^*$}; \node[anchor=east] at (-0.5, -0.5) {$d$};
				\caps{0 1 d, 2 3 d}\draw[dotted] (0, \currh) -- (0, {\currh-1}) (3, \currh)--(3, {\currh-1});\draw[->, snake=snake] (3.75,\currh) -- (4.25, \currh);\node[anchor=east] at (-1.25, \currh) {(b)};\wdiag{- - - -}
				\cups{1 2}
				\begin{scope}[xshift=5cm]
					\caps{0 1}\draw[dotted] (0, \currh) -- (0, {\currh-1}) (1, \currh) -- (1, {\currh-1});\wdiag{- -}	
				\end{scope}
			\end{tikzpicture}
		\end{tabular}
	\end{center}

We can apply the indicated reduction and we denote the reduced diagrams by $a\lambda t_1^*\nu_1d_1$ and $d_1^*\kappa_1t_1\mu b$ respectively.
	Let us first look at the second case. The com\-putation of $(a\lambda t^*\nu d)(d^*\kappa t\mu b)$ involves applying the same surgery procedures as for $(a\lambda t_1^*\nu_1d_1)(d_1^*\kappa_1t_1\mu b)$ and one additional surgery at the end. These first surgeries are actually the same because the orientations of every component agree and the tags are the same, as in this case we change two undotted cups and a cap into an undotted cup. 
	Then in the end we apply the following surgery procedure: the circle which gets split is either oriented anticlockwise or clockwise, but it has the same orientation as the one in the reduced picture (on the right).
	\begin{figure}[h]
		\centering
		\begin{tabular}{c||c}
			\begin{tikzpicture}[scale=0.5]
				\caps{0 1 d, 2 3 d}\wdiag{- - - -}
				\cups{1 2}
				\caps{1 2}
				\rays{0 0, 3 3}\draw[->, snake=snake] (3.75, \currm) -- (4.25, \currm);\wdiag{- - - -}
				\cups{0 1 d, 2 3 d}
				\begin{scope}[xshift=5cm]
					\FPset\stddiff{2}
					\caps{0 1 d, 2 3 d}\wdiag{- - - -}
					\rays{0 0,1 1,2 2,3 3}\wdiag{- - - -}
					\cups{0 1 d, 2 3 d}
				\end{scope}
			\end{tikzpicture}&
			\begin{tikzpicture}[scale=0.5]
				\FPset\stddiff{2}
				\caps{0 1}\wdiag{- -}
				\rays{0 0,1 1}\wdiag{- -}
				\cups{0 1}
			\end{tikzpicture}
		\end{tabular}
	\caption{Additional surgery for $t^*d$ in contrast to $t_1^*d_1$: \emph{Case 4a}}\label{addsurgtdfoura}
	\end{figure}
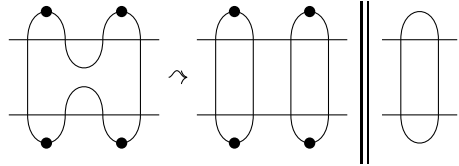
	This last additional surgery is a split and it either splits an anticlockwise or a clockwise circle. If the circle is oriented anticlockwise, it produces the sum of two basis vectors and in each of these, one circle is oriented anticlockwise. Thus $\omega$ produces $0$ and in the reduced picture (see right hand side of \cref{addsurgtdfoura}) we also have an anticlockwise circle and hence $\omega$ produces $0$ there as well. All together we have (as both pictures have the same upper reduction)
	\begin{align*}	\phi((a\lambda t^*\nu d)\otimes (d^*\kappa t\mu b)) &= \phi((a\lambda t_1^*\nu_1d_1)\otimes(d_1^*\kappa_1t_1\mu b))=
		\omega((a\lambda t_1^*\nu_1d_1)(d_1^*\kappa_1t_1\mu b))\\&=0=\omega((a\lambda t^*\nu d)(d^*\kappa t\mu b)).
	\end{align*}
	
	 If on the other hand the circles is oriented clockwise, the split produces exactly two clockwise circles and the involved sign is $(-1)^{\po(i)+1}$. In the end comparing  $(a\lambda t^*\nu d)(d^*\kappa t\mu b)$ with $(a\lambda t_1^*\nu_1d_1)(d_1^*\kappa_1t_1\mu b)$, we see that both agree up to the sign $(-1)^{\po(i)+1}$ and clockwise internal circles. But these internal clockwise circles get removed by $\omega$, hence we get up to the sign the same result
	\begin{equation*}
		\omega((a\lambda t^*\nu d)(d^*\kappa t\mu b))=(-1)^{\po(i)+1}\omega((a\lambda t_1^*\nu_1d_1)(d_1^*\kappa_1t_1\mu b)).
	\end{equation*}
	On the other hand both pictures have the same upper reduction and thus defining the involved sign for $t^*d$ to be $(-1)^{\po(i)+1}$ times the one associated to $t_1^*d_1$ we get
	\begin{equation*}
		\phi((a\lambda t^*\nu d)\otimes(d^*\kappa t\mu b))=(-1)^{\po(i)+1}\phi((a\lambda t_1^*\nu_1d_1)(d_1^*\kappa_1t_1\mu b)).
	\end{equation*}
	As $\ca(t_1)<\ca(t)$ we are done by induction.
	
	Now for the last case we again have the same surgeries for $(a\lambda t^*\nu d)(d^*\kappa t\mu b)$ and $(a\lambda t_1^*\nu_1d_1)(d_1^*\kappa_1t_1\mu b)$ and an additional one for $(a\lambda t_1^*\nu_1d_1)(d_1^*\kappa_1t_1\mu b)$ and two additional ones for $(a\lambda t^*\nu d)(d^*\kappa t\mu b)$ somewhere in the middle (but for both at the same point) (see \cref{addsurgtdfourb}).
	\begin{figure}[h]
		\begin{tabular}{c||c}
		\begin{tikzpicture}[xscale=0.69, yscale=0.73]
			\caps{1 2}\draw[dotted] (0, \currh) -- (0, {\currh+1}) (3, \currh)--(3, {\currh+1});\wdiag{- - - -}
			\cups{0 1 d, 2 3 d}\caps{0 1 d, 2 3 d}\draw[dotted] (0, \currh) -- (0, {\currh-1}) (3, \currh)--(3, {\currh-1});\draw[->, snake=snake] (3.75, \currm) -- (4.25, \currm);\wdiag{- - - -}
			\cups{1 2}
			\begin{scope}[xshift=5cm]
				\caps{1 2}\draw[dotted] (0, \currh) -- (0, {\currh+1}) (3, \currh)--(3, {\currh+1});\wdiag{- - - -}
				\cups{0 1 d}\caps{0 1 d}\rays{2 2,3 3}\draw[dotted] (0, \currh) -- (0, {\currh-1}) (3, \currh)--(3, {\currh-1});\draw[->, snake=snake] (3.75, \currm) -- (4.25, \currm);\wdiag{- - - -}
				\cups{1 2}
			\end{scope}
			\begin{scope}[xshift=10cm]
				\FPset\stddiff{2}
				\caps{1 2}\draw[dotted] (0, \currh) -- (0, {\currh+1}) (3, \currh)--(3, {\currh+1});\wdiag{- - - -}
				\rays{0 0, 1 1,2 2,3 3}\draw[dotted] (0, \currh) -- (0, {\currh-1}) (3, \currh)--(3, {\currh-1});\wdiag{- - - -}
				\cups{1 2}
			\end{scope}
		\end{tikzpicture}&
	\begin{tikzpicture}[scale=0.69]
		\draw[dotted] (0,\currh) -- (0, {\currh+1}) (1,\currh) -- (1, {\currh+1});\wdiag{- -}
		\cups{0 1}
		\caps{0 1}\draw[->, snake=snake] (1.75, \currm) -- (2.25, \currm);
		\wdiag{- -}
		\draw[dotted] (0,\currh) -- (0, {\currh-1}) (1,\currh) -- (1, {\currh-1});
		\begin{scope}[xshift=3cm]
			\FPset\stddiff{2}
			\draw[dotted] (0,\currh) -- (0, {\currh+1}) (1,\currh) -- (1, {\currh+1});\wdiag{- -}
			\rays{0 0,1 1}
			\wdiag{- -}
			\draw[dotted] (0,\currh) -- (0, {\currh-1}) (1,\currh) -- (1, {\currh-1});
		\end{scope}
	\end{tikzpicture}
	\end{tabular}
	\caption{The additional surgeries for $t^*d$ in contrast to $t_1^*d_1$: \emph{Case 4b}}\label{addsurgtdfourb}
	\end{figure}
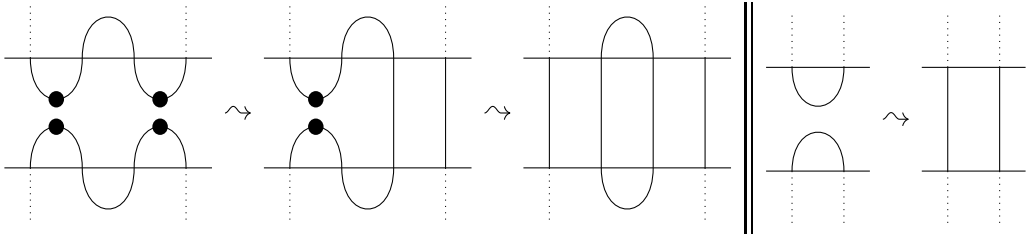
	Up to this point the applied surgery procedures again give the same result, as the component which is reduced is oriented in the same way as before and no sign involved in the multiplication process gets changed. Both first additional surgeries are of the same type and produce the same diagram just differing in this local spot (see \cref{addsurgtdfourb}) and maybe some additional signs. By looking at the definition of the multiplication one sees that the upcoming signs in a merge or a split turn out to be the same and for a reconnect at least one line would not be nonpropagating by admissibility, thus the results would be $0$ anyway.
	The second additional surgery is then splitting off the circle in the middle in \cref{addsurgtdfourb}. If the component is oriented clockwise, the split produces two clockwise oriented components and the involved sign is $(-1)^{\po(i)}$. So up to the sign and this clockwise oriented circle in between, the linear combinations agree. But $\omega$ removes the clockwise circle. If the component is oriented anticlockwise, the split produces a sum of two diagrams. In one the extra circle is oriented anticlockwise but then $\omega$ would produce $0$. So we can concentrate on the summand, where the extra circle is oriented clockwise. Then the other component is necessarily oriented anticlockwise (and thus as before) and the involved sign is $(-1)^{\po(i)}$. 
	
	All surgeries applied after these ones yield the same results, thus one finishes with the same linear combination of diagrams all just multiplied by $(-1)^{\po(i)}$.
	Again just as before the upper reduction of the reduced and the original picture is the same and furthermore $\ca(t_1)<\ca(t)$, and defining the sign for $\phi$ accordingly, we are done by induction and finished with this case.
	
	\emph{Case 5:} In the general setting, we may assume (using the base case of the induction, \emph{Case 1} and \emph{Case 2}) that we have cups in $t^*$ but neither a small circle in $t^*d$ nor an upper line containing only one cup.
	Then we have to have one of the subpictures
	\begin{center}
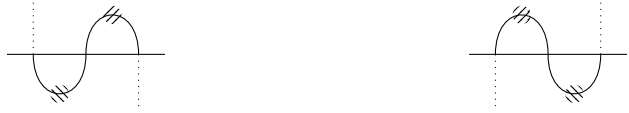
		\begin{tabular}{m{0.4\textwidth}m{0.4\textwidth}}
			\centering
		\begin{tikzpicture}[scale=0.7]
		\caps{1 2}\fill[pattern=north east lines] (1.5, {0.75}) circle (5pt);\draw[dotted] (0, \currh) -- (0, {\currh+1}) (2, \currh) -- (2, {\currh-1});\wdiag{- - -}\cups{0 1}\fill[pattern=north west lines] (0.5, {-0.75}) circle (5pt); 
	\end{tikzpicture}&
\centering
\begin{tikzpicture}[xscale=-1, scale=0.7]
	\caps{1 2}\fill[pattern=north east lines] (1.5, {0.75}) circle (5pt);\draw[dotted] (0, \currh) -- (0, {\currh+1}) (2, \currh) -- (2, {\currh-1});\wdiag{- - -}\cups{0 1}\fill[pattern=north west lines] (0.5, {-0.75}) circle (5pt); 
\end{tikzpicture}
\end{tabular}
\end{center}

where a dashed dot means that a dot can be present or not. 
We may assume that we choose a picture such that the horizontal distance between the endpoints is minimal.
This means that no attached cup or cap can end ``inside'' the cap or cup of the subpicture, i.e.~one endpoint is at one of the dashed lines and the other one is either to the left or to the right of the picture.
First observe that there cannot be two dots because then the picture would not be admissible, as the left dashed line would necessarily cut off one of the dots from the left boundary.
If no dot is present we are either in \emph{Case 3a} or \emph{Case 3b}.
If one dot is present, \Cref{casedistincgencase} makes a case distinction between which of the arcs is dotted and what happens on the dotted line attached to the undotted arc. This concludes the proof as in each case we can apply one of \emph{Case 3} and \emph{Case 4} and for those we have seen the claim before.

\begin{figure}[h]
	\centering
\scalebox{0.5}{
\setlength{\extrarowheight}{2.5cm}
\begin{tabular}{c||ccc}
	\begin{tikzpicture}
	\caps{1 2 d}\draw[dotted] (0, \currh) -- (0, {\currh+1}) (2, \currh) -- (2, {\currh-1});\wdiag{- - -}\cups{0 1}
	\end{tikzpicture}&
	\begin{tikzpicture}
		\caps{0 1, 2 3 d}\draw[dotted] (3, \currh) -- (3, {\currh-1})(0, \currh) -- (0, {\currh-1});\wdiag{- - - -}\draw[->, snake=snake] (3.75, \currm) -- (4.25, \currm);\node at (4.75, \currm) {\emph{3a}};\cups{1 2}
	\end{tikzpicture}&
\begin{tikzpicture}
\caps{0 1 d, 2 3 d}\draw[dotted] (3, \currh) -- (3, {\currh-1})(0, \currh) -- (0, {\currh-1});\wdiag{- - - -}\draw[->, snake=snake] (3.75, \currm) -- (4.25, \currm);\node at (4.75, \currm) {\emph{4a}};\cups{1 2}
\end{tikzpicture}&\\
	\begin{tikzpicture}
	\caps{1 2}\draw[dotted] (0, \currh) -- (0, {\currh+1}) (2, \currh) -- (2, {\currh-1});\wdiag{- - -}\cups{0 1 d}
\end{tikzpicture}&
\begin{tikzpicture}
	\caps{1 2}\draw[dotted] (3, \currh) -- (3, {\currh+1})(0, \currh) -- (0, {\currh+1});\wdiag{- - - -}\draw[->, snake=snake] (3.75, \currm) -- (4.25, \currm);\node at (4.75, \currm) {\emph{3b}};\cups{0 1 d, 2 3}
\end{tikzpicture}&
\begin{tikzpicture}
	\caps{1 2}\draw[dotted] (3, \currh) -- (3, {\currh+1})(0, \currh) -- (0, {\currh+1});\wdiag{- - - -}\draw[->, snake=snake] (3.75, \currm) -- (4.25, \currm);\node at (4.75, \currm) {\emph{4b}};\cups{0 1 d, 2 3 d}
\end{tikzpicture}&
\begin{tikzpicture}
	\caps{1 2}\draw[dotted](0, \currh) -- (0, {\currh+1});\wdiag{- - -}\draw[->, snake=snake] (2.75, \currm) -- (3.25, \currm);\node at (3.75, \currm) {\emph{3d}};\fill[pattern=north west lines] (2, {\currh-0.5}) circle (5pt); \cups{0 1 d}\rays{i 2}
\end{tikzpicture}\\ 
\begin{tikzpicture}[yscale=-1]
	\caps{1 2 d}\draw[dotted] (0, \currh) -- (0, {\currh+1}) (2, \currh) -- (2, {\currh-1});\wdiag{- - -}\cups{0 1}
\end{tikzpicture}&
\begin{tikzpicture}[yscale=-1]
	\caps{0 1, 2 3 d}\draw[dotted] (3, \currh) -- (3, {\currh-1})(0, \currh) -- (0, {\currh-1});\wdiag{- - - -}\draw[->, snake=snake] (3.75, \currm) -- (4.25, \currm);\node at (4.75, \currm) {\emph{3b}};\cups{1 2}
\end{tikzpicture}&
\begin{tikzpicture}[yscale=-1]
	\caps{0 1 d, 2 3 d}\draw[dotted] (3, \currh) -- (3, {\currh-1})(0, \currh) -- (0, {\currh-1});\wdiag{- - - -}\draw[->, snake=snake] (3.75, \currm) -- (4.25, \currm);\node at (4.75, \currm) {\emph{4b}};\cups{1 2}
\end{tikzpicture}&\\
\begin{tikzpicture}[yscale=-1]
	\caps{1 2}\draw[dotted] (0, \currh) -- (0, {\currh+1}) (2, \currh) -- (2, {\currh-1});\wdiag{- - -}\cups{0 1 d}
\end{tikzpicture}&
\begin{tikzpicture}[yscale=-1]
	\caps{1 2}\draw[dotted] (3, \currh) -- (3, {\currh+1})(0, \currh) -- (0, {\currh+1});\wdiag{- - - -}\draw[->, snake=snake] (3.75, \currm) -- (4.25, \currm);\node at (4.75, \currm) {\emph{3a}};\cups{0 1 d, 2 3}
\end{tikzpicture}&
\begin{tikzpicture}[yscale=-1]
	\caps{1 2}\draw[dotted] (3, \currh) -- (3, {\currh+1})(0, \currh) -- (0, {\currh+1});\wdiag{- - - -}\draw[->, snake=snake] (3.75, \currm) -- (4.25, \currm);\node at (4.75, \currm) {\emph{4a}};\cups{0 1 d, 2 3 d}
\end{tikzpicture}&
\begin{tikzpicture}[yscale=-1]
\caps{1 2}\draw[dotted](0, \currh) -- (0, {\currh+1});\wdiag{- - -}\draw[->, snake=snake] (2.75, \currm) -- (3.25, \currm);\node at (3.75, \currm) {\emph{3c}};\fill[pattern=north west lines] (2, {\currh-0.5}) circle (5pt); \cups{0 1 d}\rays{i 2}
\end{tikzpicture}
\end{tabular}}
\caption{Case distinction for the general case}\label{casedistincgencase}
\end{figure}
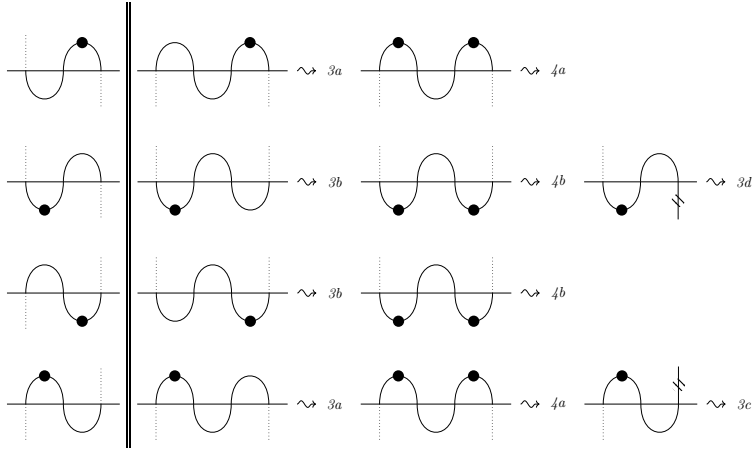
\end{proof}

With this key lemma at hand, the statements and proofs of \cite[Theorem 4.7]{BS2}, \cite[Corollary 4.8]{BS2}, \cite[Corollary 4.9]{BS2}, \cite[Theorem 4.10]{BS2}, \cite[Theorem 4.11]{BS2} and \cite[Corollary 4.12]{BS2} follow (see also \cite{Nehme}).

\section{Nuclear diagrams and projective functors}\label{sec:nuclear}
In this section we are going to introduce nuclear circle diagrams, define an analogue of the projective functors incorporating nuclear diagrams and study these. These nuclear circle diagrams do not appear in type $A$, hence we give mostly complete proofs from now on.

\begin{defi}
	A nuclear circle diagram $a\lambda b\in K_\Lambda$ is an oriented circle diagram with at least one nonpropagating line. We denote by $\mathbb{I}_\Lambda\subseteq K_\Lambda$ the span of all nuclear circle diagrams. 
\end{defi}

\begin{lem}\label{twosidednuclearideal}\label{nuclearideal}
	The vector space $\mathbb{I}_\Lambda$ is a two-sided ideal in $K_\Lambda$.
\end{lem}

\begin{proof}
	This is \cite{ES2}*{Proposition 5.3}.
\end{proof}
Using \cref{twosidednuclearideal} above, we get an induced multiplication on $\tilde{K}_\Lambda\coloneqq K_\Lambda/\mathbb{I}_\Lambda$ turning this into a graded algebra. As for $K$, the $e_\lambda, \lambda\in\Lambda$ (or rather their equivalence classes) provide a set of local units. Thus in this algebra the simple modules are again characterized by $\lambda\in\Lambda$. They are one-dimensional and  $e_\lambda$ acts by $1$ and every other circle diagram by $0$. 
Furthermore the projective indecomposable modules are given by $(K_\Lambda/\mathbb{I}_\Lambda)e_\lambda$ and these are in fact self-dual and hence prinjective (see \cite{ES3}*{Section II.4}). We will denote the simple and the projective indecomposable modules by $\overline{L}(\lambda)$ respectively $\overline{P}(\lambda)$. 
The statement from \cite{ES1}*{Theorem 6.10} that $K$ is generated in degrees $0$ and $1$ directly gives us the following result.

\begin{lem}\label{tildekgendegreeone}
	The algebra $\tilde{K}$ is generated in degrees $0$ and $1$.
\end{lem}


On the next pages, we are going to extend the notion of nuclear morphisms to $\LAMBDA$-circle diagrams and for this we fix  notation as follows. Let $\LAMBDA=\Lambda_k\dots\Lambda_0$ and $\GAMMA=\Gamma_l\dots\Gamma_0$ be sequences of blocks such that $\Lambda_0=\Gamma_l$. Let $\t=t_k\dots t_1$ (resp.~$\u=u_l\dots u_1$) be an oriented $\LAMBDA$-matching (resp.~$\GAMMA$-matching). As before denote the block sequence $\Lambda_k\dots\Lambda_1\Gamma_l\dots\Gamma_0$ by $\LAMBDA\WR\GAMMA$ and let $\t\u=t_k\dots t_1u_l\dots u_1$.

\begin{defi}
	An oriented $\LAMBDA$-circle diagram is called \emph{nuclear} if it contains at least one nonpropagating strand.
	Denote the span of these circle diagrams by $\mathbb{I}^{\t}_{\LAMBDA}$. Furthermore we will abbreviate $\tilde{K}^{\t}_{\LAMBDA}\coloneqq K^{\t}_{\LAMBDA}/\mathbb{I}^{\t}_{\LAMBDA}$.	
\end{defi}
\begin{lem}\label{mutliplicationnuclear}
	The multiplication $m$ from \eqref{defmultiplication} induces a degree preserving map 
	\begin{equation*}
		\tilde{m}\colon \tilde{K}^{\t}_{\LAMBDA}\otimes\tilde{K}^{\u}_{\GAMMA}\to\tilde{K}^{\t\u}_{\LAMBDA\WR\GAMMA}
	\end{equation*} which is associative and antimultiplicative in the same sense as in  \eqref{assocgeombimod} and \eqref{antimultgeombimod}.
\end{lem}
\begin{proof}
	By the definition of the multiplication it is easy to see that under $m$, $\mathbb{I}^{\t}_{\LAMBDA}\otimes K^{\u}_{\GAMMA}$ and $K^{\t}_{\LAMBDA}\otimes \mathbb{I}^{\u}_{\GAMMA}$ are sent to $\mathbb{I}^{\t\u}_{\LAMBDA\WR\GAMMA}$, and thus the multiplication factors as claimed. It is degree preserving because $m$ is and the subspaces $\mathbb{I}^{\t}_{\LAMBDA}$ are homogeneous by definition. Antimultiplicativity and associativity follow directly from the analogous statements for $m$.
\end{proof}
\begin{rem}
	In the special case that $\u$ is empty (and using the mirrored argument), we see that $\mathbb{I}^{\t}_{\LAMBDA}$ is a $(K_{\Lambda_k}, K_{\Lambda_0})$-bisubmodule of $K^{\t}_{\LAMBDA}$. In the subcase that $\t$ and $\u$ are empty we recover \cref{nuclearideal}. \end{rem}
\begin{lem}\label{mutliplicationisonuclear}
	The map $\tilde{m}$ is $\tilde{K}_{\Lambda_0}$-balanced and thus induces a map 
	\begin{equation*}
		\tilde{m}\colon \tilde{K}^{\t}_{\LAMBDA}\otimes_{\tilde{K}_{\Lambda_0}}\tilde{K}^{\u}_{\GAMMA}\to\tilde{K}^{\t\u}_{\LAMBDA\WR\GAMMA}
	\end{equation*}
	which is in fact an isomorphism.
\end{lem}
\begin{proof}
	That it is $\tilde{K}_{\Lambda_0}$-balanced follows from the associativity of $\tilde{m}$ and hence it factors as desired.
%
%
	
	In order to see that $\tilde{m}$ is an isomorphism, note that it is surjective because $m$ is. For injectivity we first prove that the restriction of the multiplication map $\mathbb{I}^{\t}_{\LAMBDA}\otimes_\bbC K^{\u}_{\GAMMA}+K^{\t}_{\LAMBDA}\otimes_\bbC \mathbb{I}^{\u}_{\GAMMA}\to\mathbb{I}^{\t\u}_{\LAMBDA\WR\Gamma}$ is surjective.
	For this let $a(\t\u)[\bm{\nu}]b\in\mathbb{I}^{\t\u}_{\LAMBDA\WR\GAMMA}$. Define $\lam\coloneqq \nu_{k+l}\dots\nu_{l}$ and $\MU'\coloneqq \nu_{l}\dots\nu_0$. Without loss of generality we may assume that one nonpropagating line ends at the bottom. Define $c$ to be the upper reduction of $\u[\MU']b$. Hence by definition of the upper reduction, $a\t[\lam]c$ contains a nonpropagating line, hence we have $a\t[\lam]c\in\mathbb{I}^{\t}_{\LAMBDA}$. By definition of the upper reduction $c^*\u[\MU']b$ is oriented. We define $\MU$ to be the same as $\MU'$ except that all components in $c^*\u b$ which lie partly in $c^*$ are oriented anticlockwise.
	We claim then that $(a\t[\lam]c)(c^*\u[\MU]b)=\pm a(\t\u)[\nu]b$. Observe that every surgery that needs to be applied is a merge and it always merges a component in $a\t[\lam]c$ with an anticlockwise circle in $c^*\u[\MU]b$. But this means that the vertices belonging to the anticlockwise circle in $c^*\u[\MU]b$ are exactly reoriented to agree with the parts in $\bm{\nu}$.
	Thus the surgery procedure produces up to possibly a sign the circle diagram $a(\t\u)[\bm{\nu}]b$, which finishes the proof of the claim.
	
	Now consider the following commutative diagram (the horizontal maps are all induced by the multiplication)
	\begin{center}
		\begin{tikzcd}[baseline={(X.base)}]
			\mathbb{I}^{\t}_{\LAMBDA}\otimes_\bbC K^{\u}_{\GAMMA}+K^{\t}_{\LAMBDA}\otimes_\bbC \mathbb{I}^{\u}_{\GAMMA}\arrow[rr, two heads]\arrow[d, hook]&&\mathbb{I}^{\t\u}_{\GAMMA\WR\LAMBDA}\arrow[d, hook]\\
			K^{\t}_{\LAMBDA}\otimes_{\bbC}K^{\u}_{\GAMMA}\arrow[r, two heads]\arrow[d, two heads]&K^{\t}_{\LAMBDA}\otimes_{K_{\Lambda_0}}K^{\u}_{\GAMMA}\arrow[d, dotted]\arrow[r, "\bar{m}"]&K^{\t\u}_{\LAMBDA\WR\GAMMA}\arrow[d, two heads]\\
			\tilde{K}^{\t}_{\LAMBDA}\otimes_{\bbC}\tilde{K}^{\u}_{\GAMMA}\arrow[r, two heads]&\tilde{K}^{\t}_{\LAMBDA}\otimes_{\tilde{K}_{\Lambda_0}}\tilde{K}^{\u}_{\GAMMA}\arrow[r, "\tilde{m}"]&|[alias=X]| \tilde{K}^{\t\u}_{\LAMBDA\WR\GAMMA}
		\end{tikzcd}.
	\end{center}
	The right and the left column are both short exact by definition and the map $\bar{m}$ is an isomorphism by \cite[Theorem 3.5]{BS2}.
	
	
	Now suppose $x\in \tilde{K}^{\t}_{\LAMBDA}\otimes_{\tilde{K}_{\Lambda_0}}\tilde{K}^{\u}_{\GAMMA}$ is mapped to $0$ by $\tilde{m}$.
	Lift this to an element $x'\in K^{\t}_{\LAMBDA}\otimes_\bbC K^{\u}_{\GAMMA}$. As $\tilde{m}(x)=0$ we must have $m(x')\in\mathbb{I}^{\t\u}_{\LAMBDA\WR\GAMMA}$. By the above claim we find some $x''\in\mathbb{I}^{\t}_{\LAMBDA}\otimes_\bbC K^{\u}_{\GAMMA}+K^{\t}_{\LAMBDA}\otimes_\bbC \mathbb{I}^{\u}_{\GAMMA}$ such that $m(x'')=m(x')$.
	Hence they agree in $K^{\t}_{\LAMBDA}\otimes_{K_{\Lambda_0}}K^{\u}_{\GAMMA}$ as $\bar{m}$ is an isomorphism by \cite[Theorem 3.5(iii)]{BS2}, and thus they also agree in $\tilde{K}^{\t}_{\LAMBDA}\otimes_{\tilde{K}_{\Lambda_0}}\tilde{K}^{\u}_{\GAMMA}$.
But as $x''\in\mathbb{I}^{\t}_{\LAMBDA}\otimes_\bbC K^{\u}_{\GAMMA}+K^{\t}_{\LAMBDA}\otimes_\bbC \mathbb{I}^{\u}_{\GAMMA}$ it becomes $0$ in $\tilde{K}^{\t}_{\LAMBDA}\otimes_{\tilde{K}_{\Lambda_0}}\tilde{K}^{\u}_{\GAMMA}$ and as $x'$ was a lift of $x$ we have $0=x\in\tilde{K}^{\t}_{\LAMBDA}\otimes_{\tilde{K}_{\Lambda_0}}\tilde{K}^{\u}_{\GAMMA}$. Thus $\tilde{m}$ is injective, finishing the proof.
\end{proof}

\begin{thm}\label{reductiontogeombimodnuclear}
	Let $\t=t_k\dots t_1$ be a proper $\LAMBDA$-matching. Denote the reduction of $\t$ by $u$ and let $n$ be the number of internal circles getting removed in the reduction process. Then we have
	\begin{align*}
		\tilde{K}^{\t}_{\LAMBDA}&\cong \tilde{K}^u_{\Lambda_k\Lambda_0}\otimes R^{\otimes n}\langle\ca(t_1)+\dots+\ca(t_k)-\ca(u)\rangle\\
		&\cong \tilde{K}^u_{\Lambda_k\Lambda_0}\otimes R^{\otimes n}\langle\cu(t_1)+\dots+\cu(t_k)-\cu(u)\rangle
	\end{align*}
	as graded $(\tilde{K}_{\Lambda_k}, \tilde{K}_{\Lambda_0})$-bimodules, viewing $\tilde{K}^u_{\Lambda_k\Lambda_0}\otimes R^{\otimes n}$ as a bimodule via acting on the first tensor factor.
\end{thm}

\begin{proof}
	Follow the proof of \cite[Theorem 3.6]{BS2} by observing that $a\t b$  contains a nonpropagating line if and only if $aub$ does and using that $\mathbb{I}^{\t}_{\LAMBDA}$ is a $(K_{\Lambda_k}, K_{\Lambda_0})$-bisubmodule of $K^{\t}_{\LAMBDA}$ by the proof of \cref{mutliplicationnuclear}.
\end{proof}

\begin{defi}
	A $\Lambda\Gamma$-matching is called a \emph{translation diagram} if the difference of the numbers of $\circ$'s and $\times$'s in $\Lambda$ (resp.~$\Gamma$) agrees. A $\LAMBDA$-matching $\t=t_k\dots t_1$ is called a \emph{translation diagram} if every $t_i$ is.
\end{defi}
\begin{rem}
	If the $\LAMBDA$-matching $\t$ is a translation diagram, so is its reduction. 
\end{rem}
From now on we will assume implicitly that every $\LAMBDA$-matching is in fact a translation diagram. Furthermore we will only consider the idempotent truncation by super weight diagrams. For this we make the following definition:

\begin{defi}\label{defieke}
	 Let $eK^{\t}_{\LAMBDA}e$ be the subalgebra of $K^{\t}_{\LAMBDA}$ spanned by all oriented stretched circle diagrams $a\lambda t\mu b^*$, where $a$ and $b$ are super weight diagrams. Additionally we let $e\mathbb{I}^{\t}_{\LAMBDA}e$ be the intersection of $eK^{\t}_{\LAMBDA}e$ and $\mathbb{I}^{\t}_{\LAMBDA}$ and $e\tilde{K}^{\t}_{\LAMBDA}e\coloneqq eK^{\t}_{\LAMBDA}e/e\mathbb{I}^{\t}_{\LAMBDA}e$.
We observe that the multiplication on $K_\Lambda$ induces a multiplication on $eK_\Lambda e$ as well and we get an induced $(eK_{\Lambda_k}e, eK_{\Lambda_0}e)$-bimodule structure on $eK^{\t}_{\LAMBDA}e$ (and analogously for $e\tilde{K}^{\t}_{\LAMBDA}e$).
We will abuse notation and denote the simple and indecomposable projective modules for the algebra $eK_\Lambda e$ (resp.~$e\tilde{K}_\Lambda e$) by $L(\lambda)$ and $P(\lambda)$ (resp.~$\overline{L}(\lambda)$ and $\overline{P}(\lambda)$).
	
We define $eKe\coloneqq \bigoplus_{\Lambda}eK_\Lambda e$, where the sum runs over all blocks and similarly $e\tilde{K}e\coloneqq\bigoplus_{\Lambda}e\tilde{K}_\Lambda e$. Using the projections $e\tilde{K}e\to e\tilde{K}_\Lambda e$ we might think of $e\tilde{K}^{\t}_{\LAMBDA}e$ as an $e\tilde{K}e$-bimodule.
\end{defi}

\begin{rem}\label{decompofbarproj}
	The algebra $e\tilde{K}e$ has a set of local units given by super weight diagrams. For a super weight diagram $\lambda$ we obtain $\overline{P}(\lambda)=e\tilde{K}e_\lambda$ and thus in the Grothendieck group we have 
	\begin{equation*}
		[\overline{P(\lambda)}]=\sum_{\underline{\mu}\nu\overline{\lambda}} q^{\deg(\underline{\mu}\nu\overline{\lambda})}[L(\mu)]=\sum_{\underline{\mu}\overline{\lambda}} (q+q^{-1})^{n_\mu}q^{\Def(\lambda)}[L(\mu)],
	\end{equation*}
where $n_\mu$ denotes the number of circles in the circle diagram $\underline{\mu}\overline{\lambda}$ and we sum over all (oriented) not nuclear circle diagrams $\underline{\mu}\overline{\lambda}$ (resp. $\underline{\mu}\nu\overline{\lambda}$) for a super weight diagram $\mu$. Compare this also with \cref{projfunctorsonsimplenuclear} and \cref{remprojfunctorsobtainproj}. Note that there are only finitely many $\mu$ such that $\underline{\mu}\overline{\nu}$ is oriented and not nuclear, and hence $\overline{P(\lambda)}$ is finite dimensional.
\end{rem}

Using this notation we obtain the following result:
\begin{lem}
	The map $\tilde{m}$ from \cref{mutliplicationisonuclear} restricts to an isomorphism
	\begin{equation*}
		\tilde{m}_e\colon e\tilde{K}^{\t}_{\LAMBDA}e\otimes_{e\tilde{K}_{\Lambda_0}e}e\tilde{K}^{\u}_{\GAMMA}e\to e\tilde{K}^{\t\u}_{\LAMBDA\WR\GAMMA}e.
	\end{equation*}
\end{lem}
\begin{proof}
	This follows easily by noting that if we have an oriented stretched circle diagram $a\lambda t\mu b^*$ such that $a$ is a super weight diagram, but if $b$ is not, we necessarily have $a\lambda t\mu b^*\in\mathbb{I}^{\t}_{\LAMBDA}$. Assume the layer numbers of $a$ and $b$ agree, i.e.~$\kappa(a)=\kappa(b)$, then $b$ would be a super weight diagram as well, because $t$ is a translation diagram. So we have $\kappa(a)\neq\kappa(b)$ and this means that $a\lambda t\mu b^*$ has to contain a nonpropagating line.
\end{proof}

In \cref{tildekgendegreeone} we have seen that $\tilde{K}$ is generated in degrees $0$ and $1$. In general idempotent truncations do not preserve this property. 

\begin{thm}\label{nucleargenerateddegreeone}
	The algebra $e\tilde{K}e$ is generated by its degree $0$ and $1$ part.
\end{thm}

\begin{rem}
	The basic idea of this proof is the same as \cite{ES1}*{Theorem 6.10}. However Ehrig and Stroppel use some reduction process in the finite case to only consider circle diagrams without lines. We will not do this for two reasons. First our weight diagrams are infinite opposed to finite, so one would first need to do some reduction to finite weight diagrams to adapt this idea to our setting. Secondly we want to make sure that in every step we actually use only circle diagrams which actually live in $e\tilde{K}e$.
\end{rem}
\begin{proof}
	We will prove the statement via induction. For this we are going to change weight diagrams locally. Every local change will either change the positions between a ray and a cup or the positions of two cups relative to each other. In particular, every local change in this proof preserves the property of being a super weight diagram, and we will not mention this later.
	We first define elements $X_{i,\lambda}$ which are based on the circle diagram $\underline{\lambda}\lambda\overline{\lambda}$. In case that the vertex at position $i$ is not part of a circle we set $X_{i,\lambda}\coloneqq 0$, and otherwise we reverse the orientation of this circle in contrast to $\underline{\lambda}\lambda\overline{\lambda}$.
	
	The proof of this theorem is split into two parts. First we are going to show that the elements $X_{i,\lambda}$ for every $i$ and $\lambda$ are generated by degree $1$ and $0$ elements.
	
	The second part is going to be an induction over the degree, where the first part allows us to consider only anticlockwise oriented circles.
	
	Suppose that we are given $X_{i,\lambda}$. We may assume $X_{i,\lambda}\neq 0$, as otherwise the claim is trivial. Furthermore denote the circle containing the vertex $i$ by $C$. We are going to consider three different cases, depending on what happens directly to the right of $C$. Either there is a line, the starting point of a cup, or the endpoint of a cup.
	
	If there is a line to the right of $C$, we can look schematically (meaning that there might be dots involved, which we omit here) at the picture 
	\begin{center}
		\begin{tikzpicture}[scale=0.6]
			\begin{scope}[yscale=-1]
				\caps{0 1}\rays{2 i}\wdiagnoline{- - -}\cups{1 2}\rays{i 0}\node at (2.5,0) {$\cdot$};
			\end{scope}
			\begin{scope}[xshift=3cm]
				\caps{0 1}\rays{2 i}\wdiagnoline{- - -}\cups{1 2}\rays{i 0}\node at (2.75,0) {$=\;\pm$};
			\end{scope}
			\begin{scope}[xshift=6.5cm]
				\caps{0 1}\rays{2 i}\wdiagnoline{- v -}\cups{0 1}\rays{i 2}
			\end{scope}
		\end{tikzpicture}.
	\end{center}
	Let $\mu$ be the weight such that $\underline{\mu}$ agrees with $\underline{\lambda}$ except that the cup belonging to $C$ and the line to the right of it are swapped. If the cup and the line in $\underline{\lambda}$ contain a dot, we choose  $\underline{\mu}$ such that it has no dot on either of them. If exactly one of the cup and the line are dotted in $\underline{\lambda}$ we require the corresponding line in $\underline{\mu}$ to be dotted. In other words, we want to have an even number of dots on this curved line in the above picture and the picture should be admissible.
	
	Then the circle diagram $\underline{\lambda}\overline{\mu}$ admits a unique degree $1$ orientation $\nu$ (i.e.~every circle is oriented anticlockwise). Then by the definition of the surgery procedure we have 
	\begin{equation*}
		\underline{\lambda}\nu\overline{\mu}\cdot\underline{\mu}\nu\overline{\lambda}=\pm X_{i,\lambda}.
	\end{equation*}
	
	If there is a circle directly to the right of $C$ we look schematically at \cref{circdirectright}.
	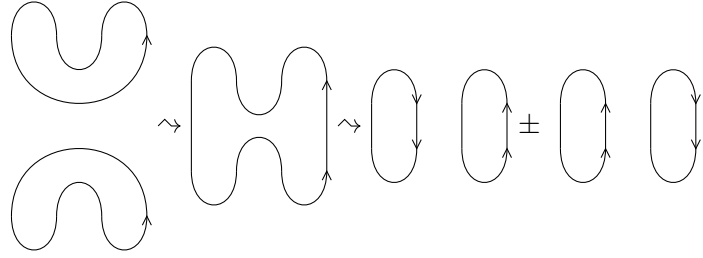
\begin{figure}[ht]\centering
		\begin{tikzpicture}[scale=0.6, baseline={(current bounding box.center)}]
			\caps{0 1, 2 3}\wdiagnoline{- - - w}\cups{0 3, 1 2}\caps{0 3, 1 2}\draw[snake, ->] (3.25, \currm) -- +(0.5, 0);\wdiagnoline{- - - w}\cups{0 1, 2 3}
		\end{tikzpicture}
		\begin{tikzpicture}[scale=0.6, baseline={(current bounding box.center)}]
			\caps{0 1, 2 3}\wdiagnoline{- - - w}\cups{1 2}\caps{1 2}\rays{0 0, 3 3}\draw[snake, ->] (3.25, \currm) -- +(0.5, 0);\wdiagnoline{- - - w}\cups{0 1, 2 3}
		\end{tikzpicture}
		\begin{tikzpicture}	[scale=0.6, baseline={(current bounding box.center)}]
			\caps{0 1, 2 3}\wdiagnoline{- v - w}\rays{0 0, 1 1, 2 2, 3 3}\node at (3.5, \currm) {$\pm$};\wdiagnoline{- v - w}\cups{0 1, 2 3}	
		\end{tikzpicture}
		\begin{tikzpicture}[scale=0.6, baseline={(current bounding box.center)}]
			\caps{0 1, 2 3}\wdiagnoline{- w - v}\rays{0 0, 1 1, 2 2, 3 3}\wdiagnoline{- w - v}\cups{0 1, 2 3}	
		\end{tikzpicture}
		\caption{The case that a circle is directly to the right of $C$}\label{circdirectright}
	\end{figure}
	We choose $\mu$ such that there are two nested cups in $\underline{\mu}$ instead of the two next to each other. We may equip the outer cup in $\underline{\mu}$ with a dot if exactly one of the cups in $\underline{\lambda}$ is dotted. If we denote the unique orientation of degree $1$ by $\nu$, we have by definition of the surgery procedures (see picture above)
	\begin{equation*}
		\underline{\lambda}\nu\overline{\mu}\cdot\underline{\mu}\nu\overline{\lambda}=\pm X_{i,\lambda}\pm X_{i+2, \lambda}.
	\end{equation*}
	Now we can repeat the argument for $X_{i+2, \lambda}$ and see that $X_{i+2, \lambda}$ is generated by degree $1$ elements. Hence this is true for $X_{i,\lambda}$. Note that this recursion has to stop at some point as $\underline{\lambda}$ has only finitely many cups.
	
	Lastly the cup corresponding to $C$ may be nested in some other cup. Then we can proceed as indicated in \cref{cnestedinsomeothercup}.
	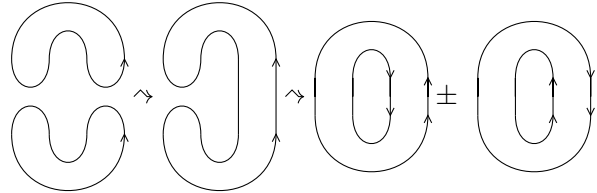
\begin{figure}[ht]\centering
		\begin{tikzpicture}[scale=0.5, baseline={(current bounding box.center)}]
			\caps{0 3, 1 2}\wdiagnoline{- - - w}\cups{0 1, 2 3}\caps{0 1, 2 3}\draw[snake, ->] (3.25, \currm) -- +(0.5, 0);\wdiagnoline{- - - w}\cups{0 3, 1 2}
		\end{tikzpicture}
		\begin{tikzpicture}[scale=0.5, baseline={(current bounding box.center)}]
			\caps{0 3, 1 2}\wdiagnoline{- - - w}\cups{0 1}\caps{0 1}\rays{2 2, 3 3}\draw[snake, ->] (3.25, \currm) -- +(0.5, 0);\wdiagnoline{- - - w}\cups{0 3, 1 2}	
		\end{tikzpicture}
		\begin{tikzpicture}[scale=0.5, baseline={(current bounding box.center)}]
			\caps{0 3, 1 2}\wdiagnoline{- - v w}\rays{0 0, 1 1, 2 2, 3 3}\node at (3.5, \currm) {$\pm$};\wdiagnoline{- - v w}\cups{0 3, 1 2}		
		\end{tikzpicture}
		\begin{tikzpicture}[scale=0.5, baseline={(current bounding box.center)}]
			\caps{0 3, 1 2}\wdiagnoline{- - w v}\rays{0 0, 1 1, 2 2, 3 3}\wdiagnoline{- - w v}\cups{0 3, 1 2}		
		\end{tikzpicture}
		\caption{The case that the cup of $C$ is nested in some other cup}\label{cnestedinsomeothercup}
	\end{figure}
	We choose $\mu$ such that there are two cups in $\underline{\mu}$ next to each other instead of two nested ones. We may equip the left cup in $\underline{\mu}$ with a dot if exactly one of the cups in $\underline{\lambda}$ is dotted. If we denote the unique orientation of degree $1$ by $\nu$, we have by definition of the surgery procedures (see picture above)
	\begin{equation*}
		\underline{\lambda}\nu\overline{\mu}\cdot\underline{\mu}\nu\overline{\lambda}=\pm X_{i,\lambda}\pm X_{j, \lambda},
	\end{equation*}
	where $j$ denotes a vertex belonging to the outer cup.
	Now similar as before we can repeat the argument for $X_{j, \lambda}$ and see that $X_{j, \lambda}$ is generated by degree $1$ elements and hence $X_{i,\lambda}$ as well. Note that this recursion has to stop at some point as $\underline{\lambda}$ has only finitely many cups. 
	This finishes the first part of the proof.
	
	The second step is to show the general statement. We prove this via induction over the degree of the circle diagram. If the degree is $0$ or $1$ the statement is trivial, so let $\underline{\lambda}\nu\overline{\mu}$ be any circle diagram of degree $>1$. By the first step, we may assume that $\nu$ is the orientation $\nu_{\min}$ of $\underline{\lambda}\overline{\mu}$ of minimal degree, as any other orientation arises from $\underline{\lambda}\nu_{\min}\overline{\mu}$ by multiplying with some $X_{i,\lambda}$.
	Take any component $C$ of $\underline{\lambda}\nu\overline{\mu}$ of degree $\geq 1$. This is either a circle or a line.
	
	If it is a line, it (or its horizontal mirror image) looks schematically like
	
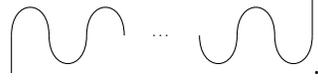
\begin{center}
		\begin{tikzpicture}[scale=0.5]
			\caps{0 1, 2 3, 6 7}\rays{8 i}\wdiagnoline{- - - - {$\dots$} - - - -}\cups{1 2, 5 6, 7 8}\rays{i 0}
		\end{tikzpicture}.
	\end{center}
	We let $\mu'$ be the weight such that $\underline{\mu'}$ differs from $\underline{\lambda}$ in the way that the ray (corresponding to the line) is swapped with the cup to the right. The cup and the ray are decorated with dots, in the unique way such that $\underline{\mu'}\overline{\mu}$ is orientable. We furthermore let $\nu'$ be the unique orientation of $\underline{\lambda}\overline{\mu'}$ of degree $1$ and $\nu''$ be the unique orientation of minimal degree of $\underline{\mu'}\overline{\mu}$.
	This then looks locally as in \cref{redprocessline}.
	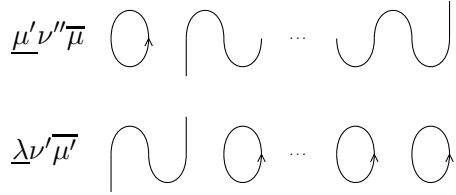
\begin{figure}[ht]\centering
		\[
		\begin{array}{ll}
			\underline{\mu'}\nu''\overline{\mu}&
			\begin{tikzpicture}[baseline = {([yshift=-0.4ex] current bounding box.center)}, scale=0.5]
				\caps{0 1, 2 3, 7 8}\rays{9 i}\wdiagnoline{- w - - - {$\dots$} - - - -}\cups{0 1, 3 4, 6 7, 8 9}\rays{i 2}
			\end{tikzpicture}\vspace{15pt}\\
			\underline{\lambda}\nu'\overline{\mu'}&
			\begin{tikzpicture}[baseline = {([yshift=-0.4ex] current bounding box.center)}, scale=0.5]\caps{0 1, 3 4, 6 7, 8 9}\rays{2 i}\wdiagnoline{- - - - w {$\dots$} - w - w}\cups{1 2, 3 4, 6 7, 8 9}\rays{i 0}
			\end{tikzpicture}
		\end{array}
		\]
		\caption{The reduction process for a line}\label{redprocessline}
	\end{figure}
	Looking at the above pictures and the definition of the surgery procedures one easily checks that 
	\begin{equation*}
		\underline{\lambda}\nu'\overline{\mu'}\cdot\underline{\mu'}\nu''\overline{\mu}=\pm\underline{\lambda}\nu\overline{\mu}.
	\end{equation*}
	By \cref{onemoreoneless} we have that the degree of $\underline{\mu'}\nu''\overline{\mu}$ is one less than the degree of $\underline{\lambda}\nu\overline{\mu}$, hence it is generated by degree $0$ and $1$ elements by induction. Therefore we see that $\underline{\lambda}\nu\overline{\mu}$ is generated by degree $0$ and $1$ elements.
	
	However, if the component $C$ is a circle, it necessarily consists of at least two cups and caps by \cref{onemoreoneless}, and we need to have a pair of cups or a pair of caps $\gamma_1$ and $\gamma_2$ nested in each other. Without loss of generality, we may assume that this is a pair of cups. We choose $\gamma_1$ such that it is not contained in any other cup and $\gamma_2$ such that it is only contained in $\gamma_1$. \Cref{excircledegreeonegen} gives an overview about our choices of $\gamma_1$ and $\gamma_2$.
	\begin{figure}[ht]\centering
		\begin{tikzpicture}[scale=0.4]
			\draw[dotted, black, out=270, in=270] (0,0) .. controls+(0,1)  and +(0,1) .. (1, 0);
			\draw[dotted, black, out=270, in=270] (2,0) .. controls+(0,1)  and +(0,1) .. (3, 0);
			\draw[dotted, black, out=270, in=270] (4,0) .. controls+(0,1)  and +(0,1) .. (5, 0);
			\draw[dotted, black, out=270, in=270] (6,0) .. controls+(0,1)  and +(0,1) .. (7, 0);
			\wdiagnoline{- - - - - - - w}\cups{0 7, 3 4}
			\node at (3.5, -3.5) {$\gamma_1$};
			\node at (3.5, -1.2) {$\gamma_2$};
			\draw[dotted, black, out=270, in=270] (1,0) .. controls+(0,-1)  and +(0,-1) .. (2, 0);
			\draw[dotted, black, out=270, in=270] (5,0) .. controls+(0,-1)  and +(0,-1) .. (6, 0);
		\end{tikzpicture}
	\vspace{-10pt}
		\caption{The choice of $\gamma_1$ and $\gamma_2$ in $C$}\label{excircledegreeonegen}
	\end{figure}
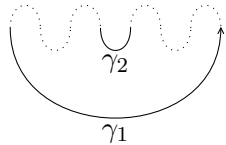
	Then we choose $\mu'$ such that $\underline{\mu'}$ is the same as $\underline{\lambda}$ except that these nested cups are replaced by two neighbored ones. \Cref{lastcasegendegone} describes the definition of $\mu'$. Then $\underline{\lambda}\overline{\mu'}$ admits a unique orientation of degree $1$, which we call $\nu'$. Additionally, we define $\nu''$ to be the orientation of minimal degree of $\underline{\mu'}\overline{\mu}$.
	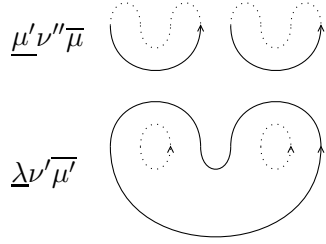
\begin{figure}[ht]\centering
		\[
		\begin{array}{ll}
			\underline{\mu'}\nu''\overline{\mu}&
			\begin{tikzpicture}[baseline = {([yshift=-0.4ex] current bounding box.center)}, scale=0.4]
				\draw[dotted, black, out=270, in=270] (0,0) .. controls+(0,1)  and +(0,1) .. (1, 0);
				\draw[dotted, black, out=270, in=270] (2,0) .. controls+(0,1)  and +(0,1) .. (3, 0);
				\draw[dotted, black, out=270, in=270] (4,0) .. controls+(0,1)  and +(0,1) .. (5, 0);
				\draw[dotted, black, out=270, in=270] (6,0) .. controls+(0,1)  and +(0,1) .. (7, 0);
				\wdiagnoline{- - - w - - - w}\cups{0 3, 4 7}
				\draw[dotted, black, out=270, in=270] (1,0) .. controls+(0,-1)  and +(0,-1) .. (2, 0);
				\draw[dotted, black, out=270, in=270] (5,0) .. controls+(0,-1)  and +(0,-1) .. (6, 0);
			\end{tikzpicture}\vspace{0pt}\\
			\underline{\lambda}\nu'\overline{\mu'}&
			\begin{tikzpicture}[baseline = {([yshift=-0.4ex] current bounding box.center)}, scale=0.4]\caps{0 3, 4 7}
				\draw[dotted, black, out=270, in=270] (1,0) .. controls+(0,1)  and +(0,1) .. (2, 0);
				\draw[dotted, black, out=270, in=270] (5,0) .. controls+(0,1)  and +(0,1) .. (6, 0);\wdiagnoline{- - w - - - w w}\cups{0 7, 3 4}\draw[dotted, black, out=270, in=270] (1,0) .. controls+(0,-1)  and +(0,-1) .. (2, 0);
				\draw[dotted, black, out=270, in=270] (5,0) .. controls+(0,-1)  and +(0,-1) .. (6, 0);
			\end{tikzpicture}
		\end{array}
		\]
		\vspace{-30pt}
		\caption{Definition of $\mu'$.}\label{lastcasegendegone}
	\end{figure}
	Then by construction and \cref{onemoreoneless} the degree of $\underline{\mu'}\nu''\overline{\mu}$ is one less than the degree of $\underline{\lambda}\nu\overline{\mu}$. Hence it is generated by degree $0$ and $1$ elements by induction.
	Furthermore by the definition of the surgery procedure we have 
	\begin{equation}
		\underline{\lambda}\nu'\overline{\mu'}\cdot\underline{\mu'}\nu''\overline{\mu}=\pm\underline{\lambda}\nu\overline{\mu}.
	\end{equation}
	Thus we see that $\underline{\lambda}\nu\overline{\mu}$ is generated by degree $1$ and $0$ elements, finishing the proof. 
\end{proof}

\begin{defi}\label{defiprojfunctornuclear}
	Let $t$ be a proper $\Lambda\Gamma$-matching and define $\tilde{G}^t_{\Lambda\Gamma}$ to be the functor $e\tilde{K}^t_{\Lambda\Gamma}e\langle-\ca(t)\rangle\otimes_{e\tilde{K}_\Gamma e}\_$. We call (possibly shifted) direct sums of these functors \emph{projective functors} as well.
\end{defi}

With these definitions the analog of \cite[Theorem 4.2]{BS2} holds for $\tilde{G}^t_{\Lambda\Gamma}\overline{P}(\gamma)$ by replacing ${K}^t_{\Lambda\Gamma})e_\gamma\langle-\ca(t)\rangle$  with $(e\tilde{K}^t_{\Lambda\Gamma}e)e_\gamma\langle-\ca(t)\rangle$.

\begin{lem}\label{nodegbilinformadjunctionnuclear}
	The map $\phi$ from \cref{defiphi} induces a homogeneous $(e\tilde{K}_\Gamma e, e\tilde{K}_\Gamma e)$-bimodule  map $\tilde{\phi}\colon e\tilde{K}^{t^*}_{\Gamma\Lambda}e\otimes e\tilde{K}^t_{\Lambda\Gamma}e\to e\tilde{K}_\Gamma e$ of degree $-2\ca(t)$, which is also $e\tilde{K}_\Lambda e$-balanced.
\end{lem}

\begin{proof}
	Let $x\coloneqq a\gamma t^*\mu d\otimes d^*\mu't\gamma'b\in e\mathbb{I}^{t^*}_{\Gamma\Lambda}e\otimes e {K}^t_{\Lambda\Gamma}e$. We want to show that $\phi(x)\in e\mathbb{I}_\Gamma e$.
	First of all observe that each basis vector of $e\mathbb{I}^t_{\Gamma\Lambda}e$ contains at least one upper and one lower line by definition and assumptions on the matching $t$. Now note that the process of upper reduction preserves lower lines, and thus if $a\gamma t^*\mu d\in e\mathbb{I}^t_{\Gamma\Lambda}e$ and if $c$ denotes the upper reduction of $t^*d$, then $a\gamma c$ contains a nonpropagating line ending at the bottom. But this line is preserved under surgeries for $(a\gamma c)(c^*\gamma'b)$ and hence $\phi(x)=\pm(a\gamma c)(c^*\gamma'b)\in e\mathbb{I}_\Gamma e$. Using this and the dual argument for  
	$ e{K}^{t^*}_{\Gamma\Lambda}e\otimes e\mathbb{I}^t_{\Lambda\Gamma}e$, we see that $\phi$ indeed factors as claimed in the statement of the lemma. The remaining properties follow from \cref{nodegbilinformadjunction}.
\end{proof}

From this we can deduce analogous results to \cite[Theorem 4.7]{BS2},  \cite[Corollary 4.8]{BS2} and  \cite[Corollary 4.9]{BS2} with the same proofs by replacing  \cite[Theorem 4.2]{BS2} with the analogous statement for $e\tilde{K}_\Lambda e$, resulting in the following corollary.

\begin{cor}\label{geombimodadjunctionnuclear}
	We have an adjoint pair of functors 
	\begin{equation*}
		(\tilde{G}^{t^*}_{\Gamma\Lambda}\langle\cu(t)-\ca(t)\rangle, \tilde{G}^t_{\Lambda\Gamma})
	\end{equation*} giving rise to a degree $0$ adjunction between $\Mod_{lf}(e\tilde{K}_\Gamma e)$ and $\Mod_{lf}(e\tilde{K}_\Lambda e)$.
\end{cor}

From \cref{nodegbilinformadjunctionnuclear} we get with the same proof as in \cite[Theorem 4.10]{BS2} (using that $\mathbb{I}^t_{\Lambda\Gamma}$ and $\mathbb{I}_\Lambda$ are preserved under $\dual$) the following theorem.
\begin{thm}\label{geombimodcommwithdualitynuclear}
	Given any proper $\Lambda\Gamma$-matching $t$ and any graded $\tilde{K}_\Gamma$-module $M$, there exists a natural isomorphism $\tilde{G}^t_{\Lambda\Gamma}(M^{\dual})\cong(\tilde{G}^t_{\Lambda\Gamma}M)^\dual$ of graded $\tilde{K}_\Lambda$-modules.
\end{thm}
Now we have all the ingredients to state the equivalent of \cite[Theorem 4.11]{BS2} in the setting of nuclear diagrams.

\begin{thm}
	\label{projfunctorsonsimplenuclear}
	Suppose we are given a proper $\Lambda\Gamma$-matching $t$ and $\gamma\in\Gamma$. Then
	\begin{enumerate}
		\item\label{projfunctorsonsimpleinuclear} in the graded Grothendieck group of $\Mod_{lf}(e\tilde{K}_\Lambda e)$
		\begin{equation*}
			[\tilde{G}^t_{\Lambda\Gamma}\overline{L}(\gamma)] = \sum_{\mu}(q+q^{-1})^{n_\mu}[\overline{L}(\mu)],
		\end{equation*}
		where $n_\mu$ denotes the number of lower circles in $\underline{\mu}t$ and we sum over all $\mu\in\Lambda$ such that 
		\begin{enumerate}[label=\normalfont{(\alph*)}]
			\item\label{projfunctorsonsimpleianuclear} $\underline{\gamma}$ is the lower reduction of $\underline{\mu}t$,
			\item\label{projfunctorsonsimpleibnuclear} there exists no lower line in $\underline{\mu}t$,
		\end{enumerate}
		\item\label{projfunctorsonsimplenonzeronuclear} the module $\tilde{G}^t_{\Lambda\Gamma}\overline{L}(\gamma)$ is nonzero if and only if all cups of $t\gamma$ are anticlockwise oriented and
		\item\label{projfunctorsonsimpleexplicitnuclear} under the assumptions of \cref{projfunctorsonsimplenonzeronuclear} define $\lambda\in\Lambda$ such that $\overline{\lambda}$ is the upper reduction of $t\overline{\gamma}$ or alternatively $\lambda t\gamma$ is oriented and every cup and cap is oriented anticlockwise. In this case $\tilde{G}^t_{\Lambda\Gamma}\overline{L}(\gamma)$ is a self-dual indecomposable module with irreducible head $\overline{L}(\lambda)\langle -\ca(t)\rangle$.
	\end{enumerate}
\end{thm}

\begin{proof} Part (i) can be proven as \cite[Corollary 4.11]{BS2}.
	
	For \cref{projfunctorsonsimplenonzeronuclear} and \cref{projfunctorsonsimpleexplicitnuclear} observe that a Jordan--Hölder series for $\tilde{G}^t_{\Lambda\Gamma}\overline{L}(\gamma)$ as $e\tilde{K}_\Lambda e$-module is the same as one Jordan--Hölder series as $eK_\Lambda e$-module. Thus we can look at a Jordan--Hölder series as $eK_\Lambda e$-modules. Now note that $(e\mathbb{I}_\Gamma e) L(\gamma) = 0$ and thus $\tilde{G}^t_{\Lambda\Gamma}\overline{L}(\gamma) = (G^t_{\Lambda\Gamma}L(\gamma))/((e\mathbb{I}_\Lambda e)G^t_{\Lambda\Gamma}L(\gamma))$. Hence $\tilde{G}^t_{\Lambda\Gamma}\overline{L}(\gamma)$ is a quotient of $G^t_{\Lambda\Gamma}L(\gamma)$ as $eK_\Lambda e$-modules and thus it can be only nonzero if each cup of $t\gamma$ is oriented anticlockwise by the type B analog of \cite[Theorem 4.11]{BS2}. But in this case $G^t_{\Lambda\Gamma}L(\gamma)$ has irreducible head $L(\lambda)\langle-\ca(t)\rangle$, where $\lambda$ is such that $\lambda t\gamma$ oriented and all every cup and cap is oriented anticlockwise. So $\tilde{G}^t_{\Lambda\Gamma}\overline{L}(\gamma)$ is zero or it has the same irreducible head. But this composition factor can only occur by \cref{projfunctorsonsimpleinuclear} if $\lambda$ satisfies \cref{projfunctorsonsimpleianuclear} and \cref{projfunctorsonsimpleibnuclear}. These two conditions are automatically satisfied as any nonpropagating line needs to have a clockwise oriented cup or cap in $t$. That $\tilde{G}^t_{\Lambda\Gamma}\overline{L}(\gamma)$ is self-dual follows from \cref{geombimodcommwithdualitynuclear} and the fact that $\overline{L}(\gamma)^\dual\cong \overline{L}(\gamma)$.
\end{proof}

\begin{rem}\label{remprojfunctorsobtainproj}
	Note that the formulas in \cref{decompofbarproj} and \cref{projfunctorsonsimplenuclear} share many similarities.
	Let $\gamma\in\Gamma$ and $\mu\in\Lambda$ denote any super weight diagrams. Define $t$ to be $\overline{\mu}\underline{\gamma}$, i.e.~we draw the cap diagram of $\mu$ under the cup diagram of $\gamma$ and connect the rays from left to right. A basis for $e\tilde{K}^t_{\Lambda\Gamma}e$ is given by $a\nu t\eta b$. As we apply $\tilde{G}^t_{\Lambda\Gamma}$ to the irreducible module $\overline{L}(\gamma)$, it follows from the definition of $\overline{L}(\gamma)$ that a basis for  $\tilde{G}^t_{\Lambda\Gamma}\overline{L}(\gamma)$ is given by $a\nu t\gamma\overline{\gamma}$. By definition of $t$ this can then be easily identified with the basis $a\nu\overline{\mu}$ for $P(\mu)$. Thus we have $\tilde{G}^t_{\Lambda\Gamma}\overline{L}(\gamma)\cong \overline{P}(\mu)\langle -\Def(\mu)\rangle$. Note the degree shift in \cref{defiprojfunctornuclear} and that $\ca(t)=\Def(\mu)$. In this case, we also see that the assumption \cref{projfunctorsonsimpleianuclear} in \cref{projfunctorsonsimplenuclear} is automatically satisfied.
\end{rem}
\begin{rem}\label{preservingfindimnuclear}
	 \cref{geombimodadjunctionnuclear} tells us, in particular, that $\tilde{G}^t_{\Lambda\Gamma}$ is exact and then we can use the same argument as in \cite[Corollary 4.12]{BS2}, replacing \cite[Theorem 4.11]{BS2} by \cref{projfunctorsonsimplenuclear}\cref{projfunctorsonsimpleinuclear} to show that $\tilde{G}^t_{\Lambda\Gamma}$ preserves finite dimensional modules.
\end{rem}


\section{Graded Brauer algebras} \label{sec:Brauer}

In this section we are going to prove the main theorem stated in the introduction. We will relate the category $\cF$ and $i$-translation with a Khovanov algebra of type B and the corresponding projective functors. Throughout this chapter $eKe$ will denote the idempotent truncation of $K$ by the super weight diagrams and $e\tilde{K}e$ the quotient of $eKe$ by the nuclear ideal, similar to \cref{sec:nuclear}.
We fix $r$, $n\in\mathbb{Z}_{\geq 0}$, set $m\coloneqq \lfloor\frac{r}{2}\rfloor$ and furthermore $\delta=r-2n$. 

\begin{figure}[H]
\begin{tabular}{rrr}
	\itemtab \begin{tabular}{m{0.75cm}m{0.75cm}m{0.75cm}m{0.75cm}}
		\begin{tikzpicture}[scale=0.5, inner xsep=0, outer xsep=0]
			\FPset\stddiff{2}
			\wdiagnoline{- o}\rays{1 0}\wdiagnoline{o -}
			\draw (-0.25, {\currh-0.3}) -- ++(1.5, 0) -- (1.25, 0.3) -- ++(-1.5, 0) -- cycle;
			\node[anchor=north, white] at (0.5, {\currh-0.3}) {$\Theta_{-\frac{1}{2}}$};
			\node[anchor=north] at (0.5, {\currh-0.3}) {$\Theta_{-i}$};
		\end{tikzpicture}&
		\begin{tikzpicture}[xscale=-1, scale=0.5, inner xsep=0, outer xsep=0]
			\FPset\stddiff{2}
			\wdiagnoline{- x}\rays{1 0}\wdiagnoline{x -}
			\draw (-0.25, {\currh-0.3}) -- ++(1.5, 0) -- (1.25, 0.3) -- ++(-1.5, 0) -- cycle;
			\node[anchor=north, white] at (0.5, {\currh-0.3}) {$\Theta_{-\frac{1}{2}}$};
			\node[anchor=north] at (0.5, {\currh-0.3}) {$\Theta_{-i}$};
		\end{tikzpicture}&
		\begin{tikzpicture}[scale=0.5, inner xsep=0, outer xsep=0]
			\FPset\stddiff{2}
			\wdiagnoline{- x}\rays{1 0}\wdiagnoline{x -}
			\draw (-0.25, {\currh-0.3}) -- ++(1.5, 0) -- (1.25, 0.3) -- ++(-1.5, 0) -- cycle;
			\node[anchor=north, white] at (0.5, {\currh-0.3}) {$\Theta_{-\frac{1}{2}}$};
			\node[anchor=north] at (0.5, {\currh-0.3}) {$\Theta_{i}$};
		\end{tikzpicture}&
		\begin{tikzpicture}[xscale=-1, scale=0.5, inner xsep=0, outer xsep=0]
			\FPset\stddiff{2}
			\wdiagnoline{- o}\rays{1 0}\wdiagnoline{o -}
			\draw (-0.25, {\currh-0.3}) -- ++(1.5, 0) -- (1.25, 0.3) -- ++(-1.5, 0) -- cycle;
			\node[anchor=north, white] at (0.5, {\currh-0.3}) {$\Theta_{-\frac{1}{2}}$};
			\node[anchor=north] at (0.5, {\currh-0.3}) {$\Theta_{i}$};
		\end{tikzpicture}
	\end{tabular}&
	\itemtab \begin{tabular}{m{0.75cm}m{0.75cm}m{0.75cm}m{0.75cm}}
		\begin{tikzpicture}[scale=0.5, inner xsep=0, outer xsep=0]
			\FPset\stddiff{2}
			\wdiagnoline{x o}\caps{0 1}\wdiagnoline{- -}
			\draw (-0.25, {\currh-0.3}) -- ++(1.5, 0) -- (1.25, 0.3) -- ++(-1.5, 0) -- cycle;
			\node[anchor=north, white] at (0.5, {\currh-0.3}) {$\Theta_{-\frac{1}{2}}$};
			\node[anchor=north] at (0.5, {\currh-0.3}) {$\Theta_{-i}$};
		\end{tikzpicture}&
		\begin{tikzpicture}[xscale=-1, scale=0.5, inner xsep=0, outer xsep=0]
			\FPset\stddiff{2}
			\wdiagnoline{- -}\cups{0 1}\wdiagnoline{x o}
			\draw (-0.25, {\currh-0.3}) -- ++(1.5, 0) -- (1.25, 0.3) -- ++(-1.5, 0) -- cycle;
			\node[anchor=north, white] at (0.5, {\currh-0.3}) {$\Theta_{-\frac{1}{2}}$};
			\node[anchor=north] at (0.5, {\currh-0.3}) {$\Theta_{-i}$};
		\end{tikzpicture}&
		\begin{tikzpicture}[scale=0.5, xscale=-1, inner xsep=0, outer xsep=0]
			\FPset\stddiff{2}
			\wdiagnoline{x o}\caps{0 1}\wdiagnoline{- -}
			\draw (-0.25, {\currh-0.3}) -- ++(1.5, 0) -- (1.25, 0.3) -- ++(-1.5, 0) -- cycle;
			\node[anchor=north, white] at (0.5, {\currh-0.3}) {$\Theta_{-\frac{1}{2}}$};
			\node[anchor=north] at (0.5, {\currh-0.3}) {$\Theta_{i}$};
		\end{tikzpicture}&
		\begin{tikzpicture}[scale=0.5, inner xsep=0, outer xsep=0]
			\FPset\stddiff{2}
			\wdiagnoline{- -}\cups{0 1}\wdiagnoline{x o}
			\draw (-0.25, {\currh-0.3}) -- ++(1.5, 0) -- (1.25, 0.3) -- ++(-1.5, 0) -- cycle;
			\node[anchor=north, white] at (0.5, {\currh-0.3}) {$\Theta_{-\frac{1}{2}}$};
			\node[anchor=north] at (0.5, {\currh-0.3}) {$\Theta_{i}$};
		\end{tikzpicture}
	\end{tabular}&\\
\itemtab \begin{tabular}{m{0.75cm}m{0.75cm}m{0.75cm}m{0.75cm}}
	\begin{tikzpicture}[scale=0.5, inner xsep=0, outer xsep=0]\FPset\stddiff{2}
		\wdiagnoline{d o}\rays{1 0 d}\wdiagnoline{o -}
		\draw (-0.25, {\currh-0.3}) -- ++(1.5, 0) -- (1.25, 0.3) -- ++(-1.5, 0) -- cycle;
		\node[anchor=north, white] at (0.5, {\currh-0.3}) {$\Theta_{-\frac{1}{2}}$};
		\node[anchor=north] at (0.5, {\currh-0.3}) {$\Theta_{-\frac{1}{2}}$};
	\end{tikzpicture}&
	\begin{tikzpicture}[scale=0.5, inner xsep=0, outer xsep=0]\FPset\stddiff{2}
		\wdiagnoline{d o}\rays{1 0}\wdiagnoline{o -}
		\draw (-0.25, {\currh-0.3}) -- ++(1.5, 0) -- (1.25, 0.3) -- ++(-1.5, 0) -- cycle;
		\node[anchor=north, white] at (0.5, {\currh-0.3}) {$\Theta_{-\frac{1}{2}}$};
		\node[anchor=north] at (0.5, {\currh-0.3}) {$\Theta_{-\frac{1}{2}}$};
	\end{tikzpicture}&
	\begin{tikzpicture}[scale=0.5, inner xsep=0, outer xsep=0]\FPset\stddiff{2}
		\wdiagnoline{o -}\rays{0 1 d}\wdiagnoline{d o}
		\draw (-0.25, {\currh-0.3}) -- ++(1.5, 0) -- (1.25, 0.3) -- ++(-1.5, 0) -- cycle;
		\node[anchor=north, white] at (0.5, {\currh-0.3}) {$\Theta_{-\frac{1}{2}}$};
		\node[anchor=north] at (0.5, {\currh-0.3}) {$\Theta_{\frac{1}{2}}$};
	\end{tikzpicture}&
	\begin{tikzpicture}[scale=0.5, inner xsep=0, outer xsep=0]\FPset\stddiff{2}
		\wdiagnoline{o -}\rays{0 1}\wdiagnoline{d o}
		\draw (-0.25, {\currh-0.3}) -- ++(1.5, 0) -- (1.25, 0.3) -- ++(-1.5, 0) -- cycle;
		\node[anchor=north, white] at (0.5, {\currh-0.3}) {$\Theta_{-\frac{1}{2}}$};
		\node[anchor=north] at (0.5, {\currh-0.3}) {$\Theta_{\frac{1}{2}}$};
	\end{tikzpicture}
\end{tabular}&
	\itemtab \begin{tabular}{m{0.75cm}m{0.75cm}m{0.75cm}m{0.75cm}}
		\begin{tikzpicture}[scale=0.5, inner xsep=0, outer xsep=0]\FPset\stddiff{2}
			\wdiagnoline{d -}\cups{0 1 d}\wdiagnoline{o x}
			\draw (-0.25, {\currh-0.3}) -- ++(1.5, 0) -- (1.25, 0.3) -- ++(-1.5, 0) -- cycle;
			\node[anchor=north, white] at (0.5, {\currh-0.3}) {$\Theta_{-\frac{1}{2}}$};
			\node[anchor=north] at (0.5, {\currh-0.3}) {$\Theta_{-\frac{1}{2}}$};
		\end{tikzpicture}&
		\begin{tikzpicture}[scale=0.5, inner xsep=0, outer xsep=0]\FPset\stddiff{2}
			\wdiagnoline{d -}\cups{0 1}\wdiagnoline{o x}
			\draw (-0.25, {\currh-0.3}) -- ++(1.5, 0) -- (1.25, 0.3) -- ++(-1.5, 0) -- cycle;
			\node[anchor=north, white] at (0.5, {\currh-0.3}) {$\Theta_{-\frac{1}{2}}$};
			\node[anchor=north] at (0.5, {\currh-0.3}) {$\Theta_{-\frac{1}{2}}$};
		\end{tikzpicture}&
		\begin{tikzpicture}[scale=0.5, inner xsep=0, outer xsep=0]\FPset\stddiff{2}
			\wdiagnoline{o x}\caps{0 1 d}\wdiagnoline{d -}
			\draw (-0.25, {\currh-0.3}) -- ++(1.5, 0) -- (1.25, 0.3) -- ++(-1.5, 0) -- cycle;
			\node[anchor=north, white] at (0.5, {\currh-0.3}) {$\Theta_{-\frac{1}{2}}$};
			\node[anchor=north] at (0.5, {\currh-0.3}) {$\Theta_{\frac{1}{2}}$};
		\end{tikzpicture}&
		\begin{tikzpicture}[scale=0.5, inner xsep=0, outer xsep=0]\FPset\stddiff{2}
			\wdiagnoline{o x}\caps{0 1}\wdiagnoline{d -}
			\draw (-0.25, {\currh-0.3}) -- ++(1.5, 0) -- (1.25, 0.3) -- ++(-1.5, 0) -- cycle;
			\node[anchor=north, white] at (0.5, {\currh-0.3}) {$\Theta_{-\frac{1}{2}}$};
			\node[anchor=north] at (0.5, {\currh-0.3}) {$\Theta_{\frac{1}{2}}$};
		\end{tikzpicture}
	\end{tabular}&
	\itemtab \begin{tabular}{m{0.75cm}}
		\begin{tikzpicture}[scale=0.5, inner xsep=0, outer xsep=0]\FPset\stddiff{2}
			\wdiagnoline{-}\rays{0 0 d}\wdiagnoline{-}
			\draw (-0.25, {\currh-0.3}) -- ++(0.5, 0) -- (.25, 0.3) -- ++(-0.5, 0) -- cycle;
			\node[anchor=north, white] at (0, {\currh-0.3}) {$\Theta_{-\frac{1}{2}}$};
			\node[anchor=north] at (0, {\currh-0.3}) {$\Theta_0$};
		\end{tikzpicture}
	\end{tabular}
\end{tabular}
\caption{Local moves}
\label{localmoves}
\end{figure}
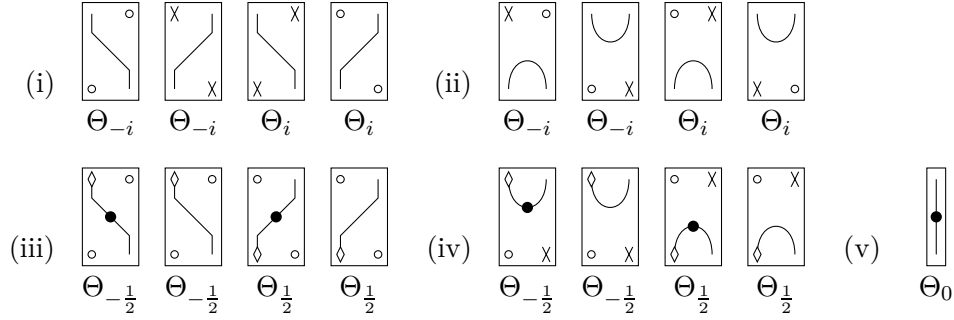

\begin{defi}\label{defitildethetai}
	Let $i\in\bbZ+\frac{\delta+1}{2}$. Given a block $\Gamma$ of Deligne weight diagrams, suppose that the number line of $\Gamma$ agrees at the vertices $\abs{i}\pm\frac{1}{2}$ with the bottom line of one of the pictures $C$ in \cref{localmoves}. By adding vertical strands $C$ can be extended to a unique $\Gamma_{t_i^C}\Gamma$-matching $t_i^C$, where $\Gamma_{t_i^C}$ is the block which is obtained from $\Gamma$, when replacing the symbols at positions $\abs{i}\pm\frac{1}{2}$ with the top of the picture $C$. 
	
	We then define the functor $\Theta_i^{\Gamma}\coloneqq \bigoplus_{C}G_{\Gamma_{t_i^C}\Gamma}^{t_i^C}\colon\Mod_{lf}(K_\Gamma)\to\Mod_{lf}(K)$, where the direct sum runs through all possible pictures, which can be put at positions $\abs{i}\pm\frac{1}{2}$ onto $\Gamma$.
	We remark here, that whenever $i\neq-\frac{1}{2}$ there is always at most one choice and if $i=-\frac{1}{2}$ and the block sequence of $\Gamma$ starts with $\diamond\circ$ we have two choices.
	
	Given this we define $\Theta_i\colon \Mod_{lf}(K)\to \Mod_{lf}(K)$ as $\bigoplus_\Gamma\Theta_i^{\Gamma}$.
	
	In the same way we define $\tilde{\Theta}_i^{\Gamma}\coloneqq \bigoplus_{C}\tilde{G}_{\Gamma_{t_i^C}\Gamma}^{t_i^C}\colon e\tilde{K}_\Gamma e\text{-mod}\to e\tilde{K}e\text{-mod}$ and $\tilde{\Theta}_i\coloneqq \bigoplus_{\Gamma}\tilde{\Theta}_i^{\Gamma}\colon e\tilde{K}e\text{-mod}\to e\tilde{K}e\text{-mod}$.
\end{defi}

\begin{defi}
	Define $T_d\coloneqq \bigoplus_{\mathbf{i}\in(\mathbb{Z}+\frac{\delta+1}{2})^d}\Theta_{\mathbf{i}}L(\emptyset_\delta)$, where
	$\Theta_{\mathbf{i}}\coloneqq \Theta_{i_d}\dots\Theta_{i_1}$ if $\mathbf{i}=(i_1,\dots i_d)$.
\end{defi}


\begin{thm}\label{blackboxformainthm}
	There exist isomorphisms of algebras $\xi_d\colon\Br_d(\delta)\overset{\cong}{\to}\End_K(T_d)$ such that the following diagram commutes
	\begin{tikzcd}
		\Br_d(\delta)\arrow[r, "\xi_d"]\arrow[d, "i\mind"]&\End_K(T_d)\arrow[d, "\Theta_i"]\\
		\Br_{d+1}(\delta)\arrow[r, "\xi_{d+1}"]&\End_K(T_{d+1})\\
	\end{tikzcd}.
\end{thm}

\begin{proof} In order to prove this theorem we replace the Brauer algebra with two graded lifts. 
	The problem is that the idempotents picking out the eigenspaces for the $i$-induction are not part of the definition of the Brauer algebra and very hard to handle. 
	So one would like to find a variant of the Brauer algebra that has these idempotents build in the definition. This is the algebra $G_d(\delta)$ provided by \cite{GELI}. This algebra $G_d(\delta)$ is the Brauer analogue of the cyclotomic Khovanov--Lauda--Rouquier algebra $R_d$ in \cite{BK} and plays the same role as $R_d$ for the degenerate affine Hecke algebra.
	
	Via the same definition of $i$-induction for $G_d(\delta)$ one can easily verify that the isomorphism between $\Br_d(\delta)$ and $G_d(\delta)$ in \cite{GELI} is compatible with $i$-induction.
	
	On the other hand we have the so called cup-cap algebra $C_d(\delta)$. It consists out of so called oriented stretched circle diagrams of height $d$ and a multiplication, which is also given by some surgery procedure, for details see \cite{OSPII}*{Section 11} and \cite{diss}*{Section 4}. This can then easily be identified with $\End_K(T_d)$ using the diagrammatic description of $\Theta_i$. For the cup-cap algebra one can also define a version of $\Theta_i$ which is given by inserting the local moves in the middle of an oriented stretched circle diagram.
	
	The most difficult and important part is the identification of $G_d(\delta)$ with $C_d(\delta)$. An explicit isomorphism can be found in \cite{diss}. By construction this isomorphism swaps $i$-induction and $\Theta_i$.
\end{proof}

Now we have all the ingredients to prove the main theorem from the introduction.

\begin{thm}\label{mainthm}
	We have an equivalence of categories $\Psi\colon(e\tilde{K}e)\text{-mod}\to \cF$ such that $\theta_i\circ\Psi\cong\Psi\circ\tilde{\Theta}_i$, which maps $\overline{L}(\lambda^\owedge_\eps)$ to $L(\lambda, \eps)$ for every $(\lambda, \eps)\in s\Gamma_\delta(m,n)\cong X^+(\Osp[r][2n])$.
\end{thm}

\begin{proof}
	By \cref{blackboxformainthm}, we have algebra isomorphisms $\psi_d\colon\End_{K}(T_d)\to\Br_d(\delta)$ for every $d\geq 0$ such that $\psi_{d+1}\circ\Theta_i=i\mind \circ\psi_d$.
	
	We can take the direct limit of $\End_k(T_d)$ with respect to the inclusion \begin{equation*}
		\bigoplus_{i\in\bbZ+\frac{\delta+1}{2}}\Theta_i\colon\End_k(T_d)\to\End_k(T_{d+1})
	\end{equation*} and we can take the direct limit of $\Br_d(\delta)$ with respect to the natural inclusion $\Br_d(\delta)\to\Br_{d+1}(\delta), f\mapsto f\otimes 1$. Note that this natural inclusion is the same as $\bigoplus_{i\in\bbZ+\frac{\delta+1}{2}}i\mind $ and thus we obtain an isomorphism \begin{equation}
		\psi\colon\varinjlim\End_k(T_d)\to\varinjlim\Br_d(\delta)
	\end{equation} with $\psi\circ\Theta_i=i\mind \circ\psi$.
	
	By \cref{FromBrauerToOspIsFull} we have a surjective algebra homomorphism $\Br_d(\delta)\to\End_{\cF}(V^{\otimes d})$. Taking the direct limit of $\End_{\cF}(V^{\otimes d})$ with respect to the embedding $f\mapsto f\otimes 1$, we obtain a surjective algebra homomorphism \begin{equation}\Phi\colon\varinjlim\Br_d(\delta)\to\varinjlim\End_{\cF}(V^{\otimes d})
	\end{equation} such that $\Phi\circ i\mind =\theta_i\circ\Phi$ (the compatibility follows from \cref{iinductionswapsitranslation}).
	
	Putting this together we obtain a surjective algebra homomorphism \begin{equation}\Psi'\coloneqq \Phi\circ\psi\colon\varinjlim\End_k(T_d)\to\varinjlim\End_{\cF}(V^{\otimes d})
	\end{equation} such that $\theta_i\circ\Psi'=\Psi'\circ\Theta_i$.

	Now we take a look at the algebra $A_{(r|2n)}$ which we define as
	\begin{equation*}
		A_{(r|2n)}\coloneqq \bigoplus_{(\lambda,\eps),(\mu,\eps')\in X^+(\Osp[r][2n])}\Hom_{\cF}(P(\lambda, \eps), P(\mu, \eps')).
	\end{equation*}
	Let $f\in A_{(r|2n)}$. By definition this can be identified with some $f\in\End_{\cF}(P,P)$ for some projective module $P$ which is the direct sum of finitely many nonisomorphic indecomposable projective objects in $\cF$. We may assume that they lie in the same block. Then by \cite{OSPII}*{Proposition 5.10} there exists $V^{\otimes d}$ containing $P$ as a summand, and thus we can consider $f$ as an endomorphism of $V^{\otimes d}$. In this way we can realize $A_{(r|2n)}$ as a subalgebra of $\varinjlim\End_{\cF}(V^{\otimes d})$.
	
	By \cref{classificationofindecomposablesummandsvialayernumberofdelignewdiag} and \cref{daggermapgiveshwofrlambda} $\Psi'$ restricts to a surjective algebra homomorphism $\bar{\Psi}\colon eKe \to A_{(r|2n)}$ which identifies the idempotent corresponding to $P(\lambda, \eps)\in\cF$ with the idempotent corresponding to the super weight diagram associated to $(\lambda, \eps)$\footnote{Actually it identifies the idempotent corresponding to the super weight diagram with the \underline{reversed} sign rule (see \cref{swdiagdifferentsignrules}) associated to $(\lambda, \eps)$. But changing the parity of the dot on the leftmost ray in $e\tilde{K}e$ is an automorphism, so we just twist in the end by this automorphism and obtain the desired result. On the super side, this would correspond to tensoring with $L(\emptyset, -)$.}. We have a commutative diagram
	\begin{center}
	\begin{tikzcd}
		\varinjlim\End_{K}(T_d)\arrow[r, "\Psi'"]&\varinjlim\End_{\cF}(V^{\otimes d})\\
		eKe\arrow[u, hook, "\iota_1"]\arrow[r, "\bar{\Psi}"]&A_{(r|2n)}\arrow[u, hook, "\iota_2"],
	\end{tikzcd}
\end{center}
when we identify $eKe=\bigoplus_{\lambda, \mu}\Hom_K(P(\lambda), P(\mu))$, where we sum over pairs of super weight diagrams with a subalgebra of $\varinjlim\End_{K}(T_d)$.
	We clearly have $\theta_i\circ\iota_2=\iota_2\circ\theta_i$, but unfortunately we do not have $\Theta_i\circ\iota_1=\iota_1\circ e\Theta_ie$. The problem is that in $K$ (and thus in $\varinjlim\End_{K}(T_d)$) we are allowed to have circle diagrams $\underline{\lambda}\nu\overline{\mu}$ such that $\kappa(\lambda)\neq\kappa(\mu)$, but in $eKe$ this is not possible by \cref{defieke} as it is the idempotent truncation by super weight diagrams and by \cref{defsuperweightdiag} every super weight diagram $\mu$ satisfies $\kappa(\mu)=\min(m,n)$. By \cite{OSPII}*{Proposition 8.8}, we know that $\Theta_i$ never decreases the layer number, but it might increase it. However, in this case $\Psi'$ produces $0$ by \cref{classificationofindecomposablesummandsvialayernumberofdelignewdiag}.
	Thus we have $\theta_i\circ\bar{\Psi}=\bar{\Psi}\circ e\Theta_ie$.
	
	By \cite{OSPII}*{Lemma 10.4} and \cref{classificationofindecomposablesummandsvialayernumberofdelignewdiag} $\bar{\Psi}$ factors through the nuclear ideal $e\mathbb{I}e$ giving rise to $\Psi\colon e\tilde{K}e\to A_{(r|2n)}$. By additionally looking at \cite{OSPII}*{Proposition 8.8} and the definition of $\tilde{\Theta}_i$ we also see that $\theta_i\circ\Psi=\Psi\circ\tilde{\Theta}_i$. By \cite{ES3}*{Theorem 5.1} $\Psi$ is an isomorphism, so we get an equivalence of categories $\Psi\colon e\tilde{K}e\text{-mod}\to\cF$ such that $\theta_i\circ\Psi\cong\Psi\circ\tilde{\Theta}_i$.
	This equivalence maps $\overline{L}(\lambda^\owedge_\eps)$ to $L(\lambda, \eps)$ as the isomorphism $\phi\colon e\tilde{K}e\to A_{(r|2n)}$ identifies the idempotent corresponding to $P(\lambda, \eps)$ in $A_{(r|2n)}$ with $e_{\lambda^\owedge_\eps}$ in $e\tilde{K}e$.
\end{proof}

\begin{rem}\label{indecsummandinkhovanov}
If we summarize our results so far in terms of understanding direct summands of $V^{\otimes d}$, we know by \cref{defiitranslation} that it suffices to understand $\theta_{i_1}\circ\dots\circ\theta_{i_d}L(\emptyset,+)$. By \cref{mainthm} this is the same as $\Psi\circ\tilde{\Theta}_{i_1}\circ\dots\circ\tilde{\Theta}_{i_d}L((\emptyset_\delta)^\owedge_+)$.
Forgetting the grading on $e\tilde{K}e$ we know that (by \cref{defitildethetai}) $\tilde{\Theta}_i$ is given by tensoring with some $\bigoplus_j e\tilde{K}^{t_j}_{\Lambda_j\Gamma_j}e$ for certain blocks $\Lambda_j$ and $\Gamma_j$ and $\Lambda_j\Gamma_j$-matchings $t_j$. Note that by definition each of these $t_j$ is a translation diagram.
\Cref{mutliplicationisonuclear} then tells us that $\tilde{\Theta}_{i_1}\circ\dots\circ\tilde{\Theta}_{i_d}$ is actually given by $\bigoplus_j e\tilde{K}^{\t_j}_{\LAMBDA_j}e$ for some sequences of blocks $\LAMBDA_j$ and $\LAMBDA_j$-matchings $\t_j$.
Finally using \cref{reductiontogeombimodnuclear} we see that the sum $\bigoplus_j e\tilde{K}^{\t_j}_{\LAMBDA_j}e$ can be reduced to $\bigoplus_{j'} e\tilde{K}^{t_{j'}}_{\Lambda_{j'}\Gamma_{j'}}e$. Furthermore by \cref{projfunctorsonsimplenuclear}\cref{projfunctorsonsimpleexplicitnuclear} we know that $e\tilde{K}^{t_{j'}}_{\Lambda_{j'}\Gamma_{j'}}e\otimes_{e\tilde{K}e}\overline{L}((\emptyset_\delta)^\owedge_+)$ is indecomposable.

As the equivalence $\Psi$ from \cref{mainthm} is necessarily additive, every indecomposable summand  of $V^{\otimes d}$ is then of the form $\Psi(e\tilde{K}^t_{\Lambda\Gamma}e\otimes_{e\tilde{K}e}\overline{L}((\emptyset_\delta)^\owedge_+))$ for some blocks $\Lambda$, $\Gamma$ and a $\Lambda\Gamma$-matching $t$. This is the same (forgetting the grading) as writing that every indecomposable summand is of the form $\Psi(\tilde{G}^t_{\Lambda\Gamma}\overline{L}((\emptyset_\delta)^\owedge_+))$.

Conversely every such choice of $\Lambda$, $\Gamma$ and $t$ gives in this way an indecomposable summand in some $V^{\otimes d}$.
\end{rem}



\section{Indecomposable tensors} \label{applications}
Recall that the indecomposable modules in $V^{\otimes d}$ are parametrized by partitions which give rise to Deligne weight diagrams. The indecomposable summands are then given by $\{\mathbb{F}\R_\delta(\lambda)\mid\kappa(\lambda_\delta)\leq\min(m,n)\}$ by \cref{classificationofindecomposablesummandsvialayernumberofdelignewdiag}.

By \cref{indecsummandinkhovanov} we know that each $\mathbb{F}\R_\delta(\lambda)$ arises as $\Psi(\tilde{G}^t_{\Lambda\Gamma}\overline{L}((\emptyset_\delta)^\owedge_+))$ for some blocks $\Lambda$ and $\Gamma$ in $e\tilde{K}e$ and some $\Lambda\Gamma$-matching $t$. It also follows from \cite{OSPII}*{Theorem 12.1} that $\mathbb{F}\R_\delta(\lambda)$ is self-dual.

\begin{prop}\label{indecradicalandsoclefiltrationagree}
	The radical and socle filtration of $\tilde{G}^t_{\Lambda\Gamma}\overline{L}((\emptyset_\delta)^\owedge_+)$ agrees with the grading filtration. In particular the radical and socle filtration of $\mathbb{F}\R_\delta(\lambda)$ is induced by the grading filtration on $\tilde{G}^t_{\Lambda\Gamma}\overline{L}((\emptyset_\delta)^\owedge_+)$, it is rigid and has Loewy length $ll(\mathbb{F}\R_\delta(\lambda))=2d(\lambda_\delta)+1$, where $d(\lambda_\delta)$ denotes the number of caps in the cap diagram of $\lambda_\delta$.
\end{prop}
\begin{proof}
	Let us abbreviate $X\coloneqq \tilde{G}^t_{\Lambda\Gamma}\overline{L}((\emptyset_\delta)^\owedge_+)$. This admits a filtration by submodules $X(j)$ which is spanned by all graded pieces of degree $\geq j$ as $e\tilde{K}e$ is a nonnegatively graded algebra. So for some $m<n$ we have
	\[
	X = X(m)\supset X(m+1)\supset\dots\supset X(n)
	\]
	and $X(n+1)=0$.
	Using \cref{nucleargenerateddegreeone} and the fact that the degree $0$ part of the Khovanov algebra is semisimple (it is a direct sum of irreducible modules) we can apply \cite{BGS96}*{Prop. 2.4.1} in conjunction with \cref{projfunctorsonsimplenuclear}\cref{projfunctorsonsimpleexplicitnuclear} to get that the socle and radical filtration agree with the grading filtration up to a shift and that $m$ respectively $n$ has to be $-\ca(t)$ respectively $\ca(t)$. The Loewy length $ll(X)$ is thus given by $2\ca(t)+1$.
	This then translates to $\mathbb{F}\R_\delta(\lambda)$ by using $\Psi$ from \cref{mainthm}.
	It remains to see that $\ca(t)=d(\lambda_\delta)$. But \cite{OSPII}*{Theorem 8.7} tells us that $\overline{\lambda_\delta}$ is the upper reduction of $t\overline{\emptyset_\delta}$. Thus as $\overline{\emptyset_\delta}$ is cap-free (as follows easily from the definitions) we have that $\ca(t)=d(\lambda_\delta)$ as no cap in $t$ gets removed during the process of upper reduction (see \cref{defiuplowreduction}).
\end{proof}

\begin{prop}\label{eachBlockHasUniqueIrredSummand}
	Each block has a unique irreducible $\mathbb{F}\R_\delta(\lambda)$.
\end{prop}
\begin{proof}
	By \cref{indecradicalandsoclefiltrationagree}, we have that $\mathbb{F}\R_\delta(\lambda)$ is irreducible if and only if $d(\lambda_\delta)=0$, i.e.~the cup diagram $\underline{\lambda_\delta}$ contains no cup.
	Now recall that by \cref{defiblock} any block is uniquely determined by the positions of $\circ$ and $\times$ as well as the parity of the number of $\wedge$'s. Note that for each of these choices we can create exactly one Deligne weight diagram $\lambda_\delta$ such that its cup diagram does not contain a cup.
	
	This follows because the positions of $\circ$ and $\times$ are fixed and by observing that whenever we have at least two $\wedge$'s, we necessarily create a cup. Therefore, we have at most one $\wedge$. In the case of one $\wedge$, note further that no $\vee$ can appear to the left of it without creating a cup, thus $\wedge$ has to appear on the leftmost free position.
	Hence we only have one choice.
	It is additionally easy to see that this choice of a weight diagram gives rise to an irreducible $\mathbb{F}\R_\delta(\lambda)$. 
\end{proof}

\subsection{Kazhdan--Lusztig polynomials of type B and Deligne--Kostant weights}\label{kostandkazhdan}

In \cite{BS2}*{Section 5} Brundan and Stroppel defined certain polynomials $p_{\lambda, \mu}(q)$ associated to weight diagrams $\lambda$, $\mu$ in type $A$ via labellings of diagrams. These also agree (up to a scaling factor) with the polynomials $Q^v_w(q)$ defined in \cite{LS}*{Section 6}, which are shown in \cite{LS}*{Théorème
7.8} to be equal to the (geometric) Kazhdan--Lusztig polynomials associated to
Grassmannians in the sense of \cite{KL}. In \cite{TS} the diagrammatic definition was extended for weight diagrams of type $D$, which is very similar to our situation. The only difference is that the weight diagrams of type $D$ are finite, whereas ours are infinite.

In \cite{BS2}*{Section 5} Brundan and Stroppel related the polynomials $p_{\lambda, \mu}(q)$ to the dimension of some extension group and gave a diagrammatical condition when these polynomials are monomials (for a fixed $\mu$). 
The same results were obtained in \cite{TS}*{Section 5} for type $D$ and we will prove these statements also for type $B$. The main advantage of the diagrammatical description of the Kazhdan--Lusztig polynomials $p_{\lambda, \mu}(q)$ in \cite{BS2} is that the definition also makes sense for unbounded weights and this gives us the possibility to directly adapt the proofs of \cite{TS}*{Section 5}.

Throughout this section we fix $\delta\in\bbZ$ and consider two Deligne weight diagrams $\lambda\leq\mu$. 
%

\begin{defi}
	A cap $\gamma$ in $\overline{\mu}$ is called \emph{$D$-nested} inside a cap $\gamma'$ if either $\gamma$ lies under $\gamma'$, or $\gamma'$ is dotted and $\gamma$ lies to the left of $\gamma'$.
	
	Suppose we are given two Deligne weight diagrams $\lambda_\delta\leq\mu_\delta$ such that $l_0(\lambda, \mu)=2k$. A \emph{$\lambda$-labelling $C$} of the oriented cap diagram $\mu\overline{\mu}$ assigns to every cap a natural number, such that the following properties are satisfied:
	\begin{enumerate}
		\item If the left end of an undotted cap is at position $i$, its label is at most $l_i(\lambda, \mu)$.
		\item The label of any dotted cap is even and at most $l_0(\lambda, \mu)$.
		\item If a cap $\gamma$ is $D$-nested inside another cap $\gamma'$, the label of $\gamma$ is greater or equal to the label of $\gamma'$.
		\item A cap may only have an odd label if there is some other cap above it or to the left of it, which has a strictly smaller label, or if there is a ray to the left of it.
	\end{enumerate}

We denote the set of $\lambda$-labellings of $\mu\overline{\mu}$ by $D(\lambda, \mu)$. The \emph{value} $\abs{C}$ of a labelling $C\in D(\lambda, \mu)$ is defined to be the sum of the labels in $C$.
\end{defi}

\begin{defi}
For two Deligne weight diagram $\lambda\leq\mu$	we define the \emph{dual Kazhdan--Lusztig polynomial $p_{\lambda, \mu}(q)$} to be
\begin{equation*}
	p_{\lambda, \mu}(q)=q^{l(\lambda, \mu)}\sum_{C\in D(\lambda, \mu)} q^{-2\abs{C}}.
\end{equation*} 
\end{defi}

The following theorem is proven in \cite{HNS}.

\begin{thm}\label{standardkoszul}
	For every standard module $V(\lambda)$ there exists a linear projective resolution
	\begin{equation*}
		\dots\to P^k\to P^{k-1}\to\dots\to P^1\to P^0\to V(\lambda)
	\end{equation*}
such that $P^k\cong \bigoplus_{\mu}p_{\lambda, \mu}^{(k)} P(\mu)\langle k\rangle$, where $p_{\lambda, \mu}^{(k)}$ denotes the coefficient of $q^k$ in $p_{\lambda, \mu}(q)$.
\end{thm}

\begin{defi}
	A Deligne weight diagram $\mu$ is called a \emph{Deligne--Kostant weight} if 
	\begin{equation*}
		\sum_{k\geq 0}\dim\Ext^k_K(V(\lambda), L(\mu))\leq 1
	\end{equation*} for all $\lambda$ in the same block as $\mu$.
\end{defi}
This definition agrees with the definition of Kostant weights for the Khovanov algebra of type $A$ given by Brundan and Stroppel in \cite{BS2}*{Section 7} (which is motivated from \cite{BH}). We speak of Deligne--Kostant weights to distinguish them from Kostant weights in the sense of \cite{GH}, which we will use in \cref{charcdirectsummands}.

\begin{prop}\label{kostantiffmonomial}
	A weight diagram $\mu$ is a Deligne--Kostant weight if and only if $p_{\lambda, \mu}(q)=q^{l(\lambda, \mu)}$ for all $\lambda\leq\mu$.
\end{prop}

\begin{proof}
	This is a direct application of \cref{standardkoszul}. Note that first of all \begin{equation*}
		\sum_{k\geq 0}\dim\Ext^k_K(V(\lambda), L(\mu))=p_{\lambda, \mu}(1)\in\bbZ\cup\{\infty\}
	\end{equation*} if $\lambda\leq\mu$ and $0$ otherwise. So we only have to consider the case $\lambda\leq\mu$. There we observe that the term $q^{l(\lambda, \mu)}$ always occurs (just take the $0$-labelling). Thus $\mu$ is a Deligne--Kostant weight if and only if $p_{\lambda, \mu}(q)=q^{l(\lambda, \mu)}$ for all $\lambda\leq\mu$.
\end{proof}
We can also directly characterize Deligne--Kostant weights in terms of the weight diagram.

\begin{defi}Let $\chi$ be a finite sequence of $\wedge$'s and $\vee$'s. A weight diagram $\lambda$ is called $\chi$-avoiding if $\chi$ does not occur as a subsequence of $\lambda$.
\end{defi}

\begin{prop}\label{delignekostant}
	For a Deligne weight diagram $\mu$ the following are equivalent:
	\begin{enumerate}
		\item\label{delignekostantone} $\mu$ is a Deligne--Kostant weight.
		\item\label{delignekostanttwo} $\mu$ is $\vee\wedge$-avoiding, $\wedge\wedge$-avoiding and $\diamond\wedge$-avoiding.
		\item\label{delignekostantthree} $\overline{\mu}$ contains no caps.
		\item\label{delignekostantfour} $\mu_\delta$ is maximal in the Bruhat order from \cref{defbruhatmove}.
	\end{enumerate}
\end{prop}

\begin{proof}
	That \cref{delignekostanttwo} is equivalent to \cref{delignekostantthree} follows directly from \cref{assoccupdiag}.
	The equivalence of \cref{delignekostanttwo} and \cref{delignekostantfour} is a direct consequence of \cref{defbruhatmove}.
	If there are no caps in $\overline{\mu}$, we only have the $0$-labelling, so (iii) directly implies (i).
	
	
	Lastly it suffices to construct a nontrivial labelling for some $\lambda\leq\mu$ in case of a subsequence $\vee\wedge$ or $\wedge\wedge$. For this we may assume that $\mu$ does not contain any $\circ$ or $\times$. 
	
	If $\mu$ contains $\vee\wedge$, we choose $i<i+1<j<k$ labelled $\vee\wedge\vee\vee$ (this is always possible as $\mu$ is admissible) and define $\lambda$ to be obtained from $\mu$ by replacing the $\vee$'s at positions $j$ and $k$ by $\wedge$'s. Then we have $l_i(\lambda, \mu)=2$ and $\overline{\mu}$ has a small cap at positions $i$, $i+1$. Labelling this cap $2$ and all other ones by $0$ gives a nontrivial $\lambda$-labelling of $\mu\overline\mu$.
	
	If $\mu$ however contains $\wedge\wedge$ and no $\vee\wedge$, the first two symbols are necessarily $\wedge\wedge$. Let $\lambda$ be the weight obtained from $\mu$ by replacing two $\vee$'s from $\mu$ by $\wedge$'s. Then $\lambda\leq\mu$ and $l_0(\lambda, \mu)=2$. Labelling the small dotted cup of $\mu$ coming from $\wedge\wedge$ by $2$ and all other ones by $0$ we obtain a nontrivial $\lambda$-labelling of $\mu\overline{\mu}$.
\end{proof}
%

\section{Characterizations of direct summands \texorpdfstring{$L(\lambda)$}{L(λ)} in \texorpdfstring{$V^{\otimes d}$}{tensor powers of V}}\label{charcdirectsummands}

The aim of this section is to give a characterization of those weights $\lambda$ such that $L(\lambda)$ appears as a direct summand of some $V^{\otimes d}$ for $\Osp[r][2n]$. So we are interested in the cases such that $\mathbb{F}\R_\delta(\mu)\cong L(\lambda, \eps)$ for some partition $\mu$ and some $(\lambda, \eps)\in s\Gamma_\delta(m,n)$. Similar results were obtained by \cite{H} for $\mathrm{GL}(m|n)$.

This gives us two different viewpoints to characterize these summands. For example we could classify those Deligne weight diagrams $\mu$ such that $\mathbb{F}\R_\delta(\mu)$ is irreducible, or we could classify those pairs $(\lambda, \eps)$ such that $L(\lambda, \eps)$ appears as a direct summand in some $V^{\otimes d}$. For the first point of view we get the following characterization

\begin{cor}\label{mudeltairred}
	For a Deligne weight diagram $\mu_\delta$ the following are equivalent.
	\begin{itemize}
		\item $\mathbb{F}\R_\delta(\mu)$ is irreducible.
		\item $\overline{\mu_\delta}$ contains no caps.
		\item $\mu_\delta$ is $\vee\wedge$, $\wedge\wedge$ and $\diamond\wedge$-avoiding.
		\item $\mu_\delta$ is maximal in the Bruhat order from \cref{defbruhatmove}.
		\item $\mu_\delta$ is a Deligne--Kostant weight.
		\item $p_{\lambda_\delta, \mu_\delta}(q)$ is a monomial for all Deligne weight diagrams $\lambda_\delta\leq\mu_\delta$.
	\end{itemize}
\end{cor}

\begin{proof}
	The equivalence of the first two properties is \cref{indecradicalandsoclefiltrationagree}. The middle four are equivalent by \cref{delignekostant} and the last two by \cref{kostantiffmonomial}.
\end{proof}

The main idea for the classification of the $(\lambda,\eps)$ is to compute $\mu_\delta^\dagger$ of a Deligne weight diagram $\mu_\delta$ with irreducible $\mathbb{F}\R_\delta(\mu_\delta)$. This will give the highest weights of the irreducible indecomposable summands. For this we first introduce the sign of a weight diagram.

\begin{defi}
Suppose that $\delta$ is odd. For a weight diagram $\lambda$ of hook partition type we define $\sgn(\lambda)$ as follows: For each $\circ$ and $\times$ appearing in $\lambda$ we count the number of symbols $\vee$ and $\wedge$ to the left of it and denote their sum $X$. We set $\sgn(\lambda)\coloneqq (-1)^{X+\#{\vee}(\lambda)}$.
If $\delta$ is even, we set $\sgn(\lambda)\coloneqq +$ for a weight diagram $\lambda$ of hook partition type.
\end{defi}
This definition allows us to characterize explicitly the irreducible direct summands in terms of weight diagrams.

\begin{thm}\label{characOfIrredTensors}
	For $(\lambda, \eps)\in X^+(\Osp[r][2n])$, the corresponding irreducible module $L(\lambda, \eps)$ is a direct summand in some $V^{\otimes d}$ if and only if $\lambda$ is typical, or it is $?{\vee}$ avoiding for $?\in\{\diamond,\vee,\wedge\}$ and $\eps=\sgn(\lambda^\infty)$.
\end{thm}

\begin{proof}
	Note that every such $L(\lambda, \eps)$ is necessarily isomorphic to some $\mathbb{F}\R(\mu_\delta)$ and thus $(\lambda, \eps)=\mu_\delta^\dagger$ by \cref{daggermapgiveshwofrlambda}. Now observe that $\mathbb{F}\R(\mu_\delta)$ is irreducible if and only if there are no caps in $\overline{\mu_\delta}$ by \cref{indecradicalandsoclefiltrationagree}. It is easy to check using the definition of the associated cap diagram that $\overline{\mu_\delta}$ contains no caps if and only if $\mu_\delta$ is $\wedge\wedge$-, $\vee\wedge$- and $\diamond\wedge$-avoiding (see also \cref{mudeltairred}).
	When removing all $\circ$'s and $\times$'s $\mu_\delta$ has to look like $\vee\vee\vee\dots$, $\wedge\vee\vee\dots$ or $\diamond\vee\vee\dots$. Then one only needs to determine $\mu^\dagger_\delta$. We remark here that the case distinction comes from the distinction between projective and nonprojective weight diagrams (i.e.~the typical and atypical case). For the sign in the odd case, note that $\sgn(\lambda^{\infty})$ is the same as $(-1)^{\abs{\mu}}$.
\end{proof}

Translating the definition of \cite{GH}*{Definition 3.5.3} to the combinatorics of Ehrig and Stroppel which we are using here, we obtain the following definition of \emph{Kostant weight}.
\begin{defi}
	We call $\lambda\in X^+(\osp[r][2n])$ a \emph{Kostant weight} if the associated weight diagram $\lambda^\infty$ is $\vee$-avoiding.
\end{defi}
\begin{rem}
	Note that every pair $(\lambda, \eps)$ for a typical highest weight $\lambda$ (which means $\min(\#{\circ}(\lambda^{\infty}), \#{\times}(\lambda^{\infty}))=\min(n,m)$) is automatically $\vee$ avoiding by \cref{hookweightdiagatypicalityandorder}.
	On the other hand if we have $\min(\#{\circ}(\lambda^{\infty}), \#{\times}(\lambda^{\infty}))<\min(n,m)$ (the atypical case), $L(\lambda, \eps)$ (for $\eps=\sgn(\lambda^{\infty})$) appears as a direct summand if and only if it is $\vee$-avoiding except for maybe the first position.
\end{rem}

In the following paragraph we are going to define a twist, which turns the first symbol different from $\circ$ or $\times$ upside down. We will use this to relate Kostant weights in the sense of \cite{GH} with the weights $\lambda$ such that $L(\lambda, \sgn(\lambda^{\infty}))$ appears as a direct summand in $V^{\otimes d}$.

Given a super weight diagram $\lambda$ we can look at the leftmost position where a $\diamond$, $\vee$ or $\wedge$ occurs. We denote the weight diagram which is obtained by turning this symbol upside down by $\lambda^\Box$. 
Comparing $\underline{\lambda}$ and $\underline{\lambda^\Box}$, this means that they agree if $\diamond$ is present and that we change the parity of dots on the leftmost component otherwise. Thus if $\underline{\lambda}\gamma\overline{\mu}$ is an oriented circle diagram, so is $\underline{\lambda^\Box}\gamma^\Box\overline{\mu^\Box}$, and this amounts to an isomorphism $\Box\colon e\tilde{K}e\to e\tilde{K}e$. Hence by \cref{mainthm} we get a self-equivalence $\Box\colon\cF\to\cF$.
We define for $(\lambda, \eps)\in X^+(\Osp[r][2n])$ the hook weight diagram $(\lambda, \eps)^\Box=(\lambda^\Box, \eps^\Box)$ by first taking the associated super weight diagram $\lambda^\owedge_\eps$, applying $\Box$ and going back to hook weight diagrams.

The map $\Box$ can also be defined on the supergroup side. Every block of $\Osp[r][2n]$ is equivalent to the principal block of $\Osp[2k+1][2k]$ or $\Osp[2k][2k]$. This can be achieved via transporting this block through $\Psi$ from \cref{mainthm} to $e\tilde{K}e$-mod. There we can remove all $\circ$'s and $\times$'s as they play no role in the module structure and transport back (this is very similar to \cite{GS10}*{Theorem 2} which relates blocks of $\osp[r][2n]$ to principal blocks in $\osp[2k+1][2k]$, $\osp[2k][2k]$ or $\osp[2k+2][2k]$).

For the principal block of $\Osp[2k+1][2k]$ the map $\Box$ is then just given by $\theta_0=\pr_{\chi_0}(\_\otimes V)$ and for the one of $\Osp[2k][2k]$ this is just the identity (all blocks containing a $\diamond$ are equivalent to this one and turning $\diamond$ upside down changes nothing).

In general $\Box$ is defined by identifying a block with the corresponding principal block under the identification above, applying the explicit description of $\Box$ there and transferring back to the original block.

\begin{prop}[\cite{GH}*{Remark 3.5.4}]
	For $\lambda\in X^+(\osp[2m+1][2n])$ we have that $L(\lambda)$ satisfies the Kac--Wakimoto character formula if and only if $\lambda$ is a Kostant weight.
	For $\lambda\in X^+(\osp[2m][2n])$ of atypicality greater than one we have that $L(\lambda)$ satisfies the Kac--Wakimoto character formula if and only if $\lambda$ is a Kostant weight.
\end{prop}

Putting everything together we obtain the following corollary:

\begin{cor}\label{irredasindecsummand}
For $\lambda\in X^+(\osp[r][2n])$ the following are equivalent:
	\begin{itemize}
		\item $L(\lambda, \eps)$ is a direct summand of some $V^{\otimes d}$, where $\eps=\sgn(\lambda^{\infty})$ if $\lambda$ is atypical and $\eps\in\{\pm\}$ otherwise.
		\item $\lambda$ or $\lambda^\Box$ is a Kostant weight.
	\end{itemize}
And if $r$ is odd or $\mathrm{at}(\lambda)>1$ this is equivalent to
\begin{itemize}
		\item $L(\lambda, \eps)$ or $L(\lambda^\Box, \eps^\Box)$ satisfies the Kac--Wakimoto condition.
	\end{itemize}
\end{cor}

\begin{rem}
	In \cite{GH}*{Remark 3.5.4}, Gorelik and Heidersdorf stated that a weight $\lambda$ for odd $r$ or $\lambda$ of atpyicality $>1$ satisfies the Kac--Wakimoto conditions if and only if it is a Kostant weight in their sense. Here a weight $\lambda$ satisfies the Kac--Wakimoto condition if
it is the highest weight of some irreducible module $L$ with respect to a base $\Sigma\supseteq S$ of simple roots, such that $S$ consists out of exactly $\mathrm{at}(\lambda)$ mutually orthogonal isotropic roots and $(S,S)=(S,\lambda+\rho)=0$.
	
	In \cite{CK}*{Theorem 5.2}, Cheng and Kwon proved that the Kac--Wakimoto conditions imply the Kac--Wakimoto character formula, i.e.~
	\begin{equation}
		Re^{\rho}\mathrm{ch}L(\lambda) = j^{-1}\sum_{w\in W}\sgn(w)w\Bigl(\frac{e^{\lambda+\rho}}{\prod_{\beta\in S}(1+e^{-\beta})}\Bigr)
	\end{equation}
where $R$ denotes the Weyl denominator $\frac{\prod_{\alpha\in\Phi^+_0}(e^{\alpha/2}-e^{-\alpha/2})}{\prod_{\beta\in\Phi^+_1}(e^{\beta/2}-e^{-\beta/2})}$, $W$ denotes the Weyl group of $\osp[r][2n]$ (which is the Weyl group of $\lie{so}(r)\oplus\lie{sp}(2n)$), and $j$ is some scalar.
For details see \cite{CK} or \cite{GH}.
\end{rem}

\section{Acknowledgements} The research of T.H., J.N. and C.S. was partially funded by the Deutsche Forschungsgemeinschaft (DFG, German Research Foundation) under Germany's Excellence Strategy – EXC-2047/1 – 390685813.


\appendix

\section{Small rank examples}

We discuss $V^{\otimes d}$ and translation functors for $\Osp[1][2]$, $\Osp[3][2]$ and $\Osp[2][2]$. In these three cases every indecomposable summand is either projective or irreducible, and so the situation simplifies a lot. We hope that these small examples serve to give the reader an idea how everything fits together.

\subsection{The semisimple case: \texorpdfstring{$\Osp[1][2]$}{OSp(1|2)}}
The category of finite dimensional representations of $\Osp[1][2]$ is semisimple, and one can use the correspondence between $\osp[1][2]$ and $\lie{so}(3)$ from \cite{RS82} to decompose tensor powers. We treat this example as a toy application for our general theory.

By \cref{integraldominanceosp} and \cref{irreduciblemodulesforOspodd} we know that the finite dimensional representations of $\Osp[1][2]$ are labelled by $(n, \eps)$ with $n\in\mathbb{N}_0$ and $\eps\in\{\pm\}$.

\subsubsection{Translating to Khovanov's arc algebra of type D}

Translating $(n, \varepsilon)$ into a super weight diagram (using \cref{defsuperweightdiag}), we obtain a weight diagram consisting of a $\times$ at position $n+\frac{1}{2}$ and a $\vee$ at every other position except for maybe the first free one. There we put in case $n$ is even a $\wedge$ if $\varepsilon = +$ and a $\vee$ if $\varepsilon=-$. If $n$ is odd, we just reverse the previous assignment, for example:
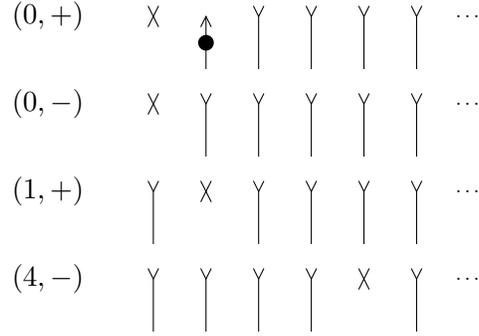
\begin{center}
	\begin{tabular}{c}
		\begin{tikzpicture}[scale=0.7]
			\node at (-2, \currh) {$(0,+)$}; \wdiagnoline{x w v v v v {$\dots$}}\rays{i 1 d,i 2,i 3,i 4,i 5}
		\end{tikzpicture}\\
		\begin{tikzpicture}[scale=0.7]
			\node at (-2, \currh) {$(0,-)$}; \wdiagnoline{x v v v v v {$\dots$}}\rays{i 1,i 2,i 3,i 4,i 5}
		\end{tikzpicture}\\
		\begin{tikzpicture}[scale=0.7]
			\node at (-2, \currh) {$(1,+)$}; \wdiagnoline{v x v v v v {$\dots$}}\rays{i 0,i 2,i 3,i 4,i 5}
		\end{tikzpicture}\\
		\begin{tikzpicture}[scale=0.7]
			\node at (-2, \currh) {$(4,-)$}; \wdiagnoline{v v v v x v {$\dots$}}\rays{i 0,i 2,i 3,i 1,i 5}
		\end{tikzpicture}
	\end{tabular}
\end{center}
As $\Osp[1][2]$ is semisimple (or equivalently as $\min(m,n)=0$) there are no cups and caps involved.
If we consider the associated cup respectively cap diagrams to these weights, we obtain diagrams consisting of one free vertex and apart from that only lines where the leftmost one might be dotted. From this it is fairly easy to see that the only circle diagrams we can build are the $e_{\lambda}$, where $\lambda$ is one of the super weight diagrams from the previous paragraph, and furthermore we cannot have any nuclear diagrams as we have no cups or caps. 
Then by \cref{mainthm} we know that the category of finite dimensional $\Osp[1][2]$-modules is equivalent to $e\tilde{K}e$-mod.

In order to later analyze the effect of $\_\otimes V$ we first take a look at the geometric bimodules $K^t_{\Lambda\Gamma}$. Now observe that $G^t_{\Lambda\Gamma}P(\gamma)$ will be $0$ whenever $t$ contains a cup or a cap. Thus the only relevant $t$'s look locally like 
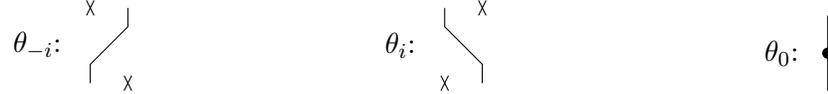
\begin{center}
	\begin{tabularx}{\textwidth}{*3{>{\Centering}X}}
		\begin{tikzpicture}[scale=0.5]
			\FPset\stddiff{2}
			\wdiagnoline{x -}\rays{0 1}\node[anchor=east] at (-0.5, \currm) {$\theta_{-i}$:};\wdiagnoline{- x}
		\end{tikzpicture}&
		\begin{tikzpicture}[scale=0.5]
			\FPset\stddiff{2}
			\wdiagnoline{- x}\rays{1 0}\node[anchor=east] at (-0.5, \currm) {$\theta_{i}$:};\wdiagnoline{x -}
		\end{tikzpicture}&
		\begin{tikzpicture}[scale=0.5]
			\FPset\stddiff{2}
			\wdiagnoline{-}\rays{0 0 d}\node[anchor=east] at (-0.5, \currm) {$\theta_0$:};\wdiagnoline{-}
		\end{tikzpicture}
	\end{tabularx}
\end{center}
where the $i$ means that it involves the positions $i+\frac{1}{2}$ and $i-\frac{1}{2}$. The last picture can only be present on the vertex $\frac{1}{2}$. Apart from these involved vertices $t$ consists only of straight lines. The geometric bimodules $K^t_{\Lambda\Gamma}$ are thus also one-dimensional and by \cref{projfunctorsonprojnuclear} using $L(\mu)=P(\mu)$ we have that $G^t_{\Lambda\Gamma}L(\gamma)=L(\lambda)$ where $\overline{\lambda}$ is the upper reduction of $t\overline{\gamma}$. In this case the upper reduction for the first two picture is obtained by swapping the $\times$ one position to the left respectively right, and in the first picture we change a $\vee$ at position $\frac{1}{2}$ into a $\wedge$ and vice versa.

Observe that for each $L(\gamma)$ we have three projective functors producing something nonzero if the $\times$ is not at position $\frac{1}{2}$ and only $\theta_1$ if it is at position $\frac{1}{2}$.

\subsubsection{Decomposition of $V^{\otimes d}$ into irreducible summands}

Translating the results of the previous paragraph, we obtain for our translation functors $\theta_i$ 
\begin{equation}\label{explicitdescriptiontranslationfunctorsosp12}
	\theta_i L(n, \varepsilon) = \begin{cases}L(n+1, -\varepsilon)&\text{if $i=n+1$,}\\L(n, -\varepsilon)&\text{if $i=0$ and $n>0$,}\\ L(n-1, -\varepsilon)&\text{if $i=-n$ and $n>0$,}\\ 0&\text{otherwise.} \end{cases}
\end{equation}
With this at hand and using $\_\otimes V=\bigoplus_{i\in\mathbb{Z}}\theta_i$, we can directly write down the first decompositions of $V^{\otimes d}$ into irreducibles.
\begin{align*}
	V^{\otimes 0} &= L(0, +) \\ V^{\otimes 1} &= L(1, -) \\V^{\otimes 2} &= L(2, +)\oplus L(1,+)\oplus L(0,+)\\ V^{\otimes 3} &= L(3, -)\oplus L(2,-)^{\oplus 2}\oplus L(1, -)^{\oplus 3}\oplus L(0, -)^{\oplus 1}
\end{align*}
Now note that in $V^{\otimes d}$ the signs of all irreducible summands are the same depending on the parity of $d$. Moreover, the multiplicity of $L(n, \varepsilon)$ in $V^{\otimes d}$ is either $0$ (if $d$ even and $\varepsilon=-$ or $d$ odd and $\varepsilon=+$) or it agrees with the multiplicity of $L(n)$ in $V^{\otimes d}$ as $\mathrm{SOSp}(1|2)$-modules.
Define $m(n, d)$ to be multiplicity of $L(n)$ in $V^{\otimes d}$, or equivalently the multiplicity of $L(n, \varepsilon)$ for the correct choice (see above) of $\varepsilon$.
Using \eqref{explicitdescriptiontranslationfunctorsosp12} we quickly get the following recurrence relations for $m(n,d)$:
\begin{align*}
	m(0,0) &= 1,\\
	m(0, d) &= m(1, d-1),\\
	m(n, d) &= m(n+1, d-1) + m(n, d-1) + m(n-1, d-1)\quad\text{if $n>0$.}
\end{align*}

Using some tricks and combinatorics, one gets then for $m(n, d)$ the explicit formulas
\begin{align*}
	m(n, d) &= \sum_{i=n}^{d}(-1)^{i+n}\begin{psmallmatrix}
		d\\ i
	\end{psmallmatrix}\begin{psmallmatrix}
		i\\ \lfloor\frac{i-n}{2}\rfloor
	\end{psmallmatrix}\\
	&= \sum_{j=0}^{\frac{d-n}{2}}\frac{d!}{j!\,(n+j)!\,(d-n-2j)!} -\sum_{j=0}^{\frac{d-n-1}{2}}\frac{d!}{j!\,(n+1+j)!\,(d-n-2j-1)!}\\
	&= T(n,d) - T(n+1, d)
\end{align*}
where $T(n, d)$ denotes the coefficient of $x^{n+d}$ in the expansion of $(1+x+x^2)^d$. The number $T(n,d)$ denotes also the number of possible outcomes of elections with $d$ votes of three parties $A$, $B$ and $C$ su
ch that $B$ obtains $n$ votes more than $A$.

\subsection{The smallest nonsemisimple odd case: \texorpdfstring{$\Osp[3][2]$}{OSp(3|2)}}


\subsubsection{The irreducible representations of $\osp[3][2]$ and $\Osp[3][2]$}


We choose the simple roots $\eps_1-\delta_1, \delta_1$ and $\rho=(-\frac{1}{2}\mid\frac{1}{2})$. Let $\lambda\in\lie{h}^*$ be a weight and write $\lambda+\rho=a\eps_1+b\delta_1$. Then $\lambda$ is integral dominant by \cref{integraldominanceosp} if and only if $a,b\in\frac{1}{2}+\bbZ$ and either $a,b\geq\frac{1}{2}$ or $-a=b=\frac{1}{2}$.
Rephrasing this means that if $\lambda=(a|b)$ is integral dominant we either have $a=b=0$ or $a\geq1$ and $b\geq 0$ where $a$ and $b$ are integers. These can be identified with $(1,1)$-hook partitions via $(a|b)\mapsto (a, 1^b)^t$. By \cref{irreduciblemodulesforOspodd} the irreducible modules for $\Osp[3][2]$ are labeled by $(\lambda,\pm)$ where $\lambda$ is a $(1,1)$-hook partition.

\subsubsection{Translating to Khovanov's arc algebra of Type D}

Let $\lambda=(k, 1^l)$ be a $(1,1)$-hook partition. 
We will distinguish three cases, the first being that $k\neq 0$ and $k-1\neq l$.
The associated flipped weight diagram then looks like
\begin{center}
\begin{tikzpicture}[scale=0.7]
	\node at (3, 0.5) {$k-\frac{1}{2}$};	\node at (7, 0.5) {$l+\frac{1}{2}$};
	\wdiagnoline{w  {$\dots$} w o w {$\dots$} w x w {$\dots$}}
\end{tikzpicture}.
\end{center}
	The positions of $\circ$ and $\times$ are swapped if $l<k-1$.
	The corresponding super weight diagram to $(\lambda, \eps)$ is created by replacing all $\wedge$'s with $\vee$'s
	except for possibly the leftmost one. There we leave the $\vee$ if $k+l$ is even and $\eps=-$, or if $k+l$ is odd and $\eps=+$. In all other cases we change the leftmost vertex to $\wedge$.

	In the case $\lambda = \emptyset$ we get the flipped weight diagram
	\begin{center}
		\begin{tikzpicture}
			\wdiagnoline{w w w w {$\dots$}}
		\end{tikzpicture}.
	\end{center}

The associated super weight diagram is 
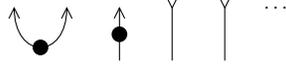
\begin{center}
	\begin{tikzpicture}[scale=0.7]
		\wdiagnoline{w w w v v {$\dots$}}\cups{0 1 d}\rays{i 2 d, i 3, i 4}
	\end{tikzpicture}
\end{center} for $(\emptyset, +)$, and for $(\emptyset, -)$ we get
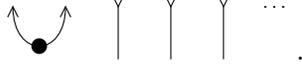
\begin{center}
	\begin{tikzpicture}[scale=0.7]
		\wdiagnoline{w w v v v {$\dots$}}\cups{0 1 d}\rays{i 2, i 3, i 4}
	\end{tikzpicture}.
\end{center}
	
	The last case is $k=l+1$. In this case the flipped weight diagram is
	
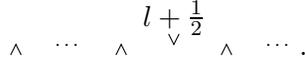
\begin{center}
		\begin{tikzpicture}[scale=0.7]
			\node at (3, 0.5) {$l+\frac{1}{2}$};
			\wdiagnoline{w {$\dots$} w v w {$\dots$}}
		\end{tikzpicture}.
	\end{center}
The associated super weight diagrams are thus given by
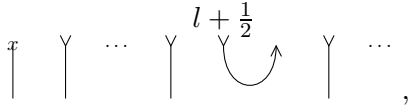
\begin{center}
	\begin{tikzpicture}[scale=0.7]
		\node at (4, 0.5) {$l+\frac{1}{2}$};
		\wdiagnoline{{$x$} v {$\dots$} v v w v {\dots}}\cups{4 5}
		\rays{i 0,i 1, i 3, i 6}
	\end{tikzpicture},
\end{center}
where $x=\vee$ if $\eps=+$, and $x=\wedge$ if $\eps=-$ (in this case there is also a dot on the leftmost ray).
Note that if $l=0$ the $x$ appears at position $\frac{5}{2}$.

If we want to directly go from highest weights to super weight diagrams, we get the following connection:
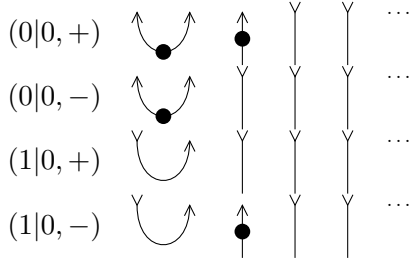
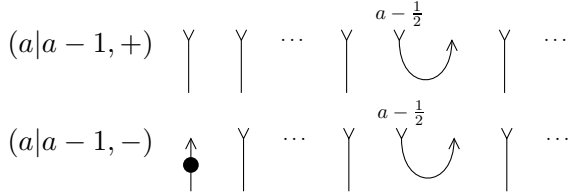
\begin{center}
\begin{tabular}{l}
$\quad\quad(0|0,+)\quad$\tikz[baseline={([yshift=-.6ex]current bounding box.center)}, scale=0.7]{\wdiagnoline{w w w v v {$\dots$}}\cups{0 1 d}\rays{i 2 d, i 3, i 4}}\\
$\quad\quad(0|0,-)\quad$\tikz[baseline={([yshift=-.6ex]current bounding box.center)}, scale=0.7]{\wdiagnoline{w w v v v {$\dots$}}\cups{0 1 d}\rays{i 2, i 3, i 4}}\\
$\quad\quad(1|0,+)\quad$\tikz[baseline={([yshift=-.6ex]current bounding box.center)}, scale=0.7]{\wdiagnoline{v w v v v {$\dots$}}\cups{0 1}\rays{i 2, i 3, i 4}}\\
$\quad\quad(1|0,-)\quad$\tikz[baseline={([yshift=-.6ex]current bounding box.center)}, scale=0.7]{\wdiagnoline{v w w v v {$\dots$}}\cups{0 1}\rays{i 2 d, i 3, i 4}}\\
\\
 For $a>1$:\\
$\quad\quad(a|a-1,+)\quad$\tikz[baseline={([yshift=-.6ex]current bounding box.center)}, scale=0.7]{\node[transform shape] at (4, 0.5) {$a-\frac{1}{2}$};\wdiagnoline{v v {$\dots$} v v w v {$\dots$}}\cups{4 5}\rays{i 0, i 1,i 3, i 6}}\\
$\quad\quad(a|a-1,-)\quad$\tikz[baseline={([yshift=-.6ex]current bounding box.center)}, scale=0.7]{\node[transform shape] at (4, 0.5) {$a-\frac{1}{2}$};\wdiagnoline{w v {$\dots$} v v w v {$\dots$}}\cups{4 5}\rays{i 0 d, i 1,i 3, i 6}}\\\\
\noindent For $b\neq a-1$ and $(a|b)\neq(0|0)$: \\
\noindent\hspace{2em}There is a $\circ$ at position $a-\frac{1}{2}$ and a $\times$ at position $b+\frac{1}{2}$. We have a dot\\
\noindent\hspace{2em}on the leftmost ray if $a+b$ is even and $\eps=+$ or if $a+b$ is odd and $\eps=-$. \\\noindent\hspace{2em}In all other cases we have no dot.
\end{tabular}
\end{center}
So our super weight diagrams either consist only of $\vee$'s and $\wedge$'s (then the cup diagram has one cup) or it has exactly one $\circ$ and one $\times$ (then the cup diagram has no cup). If $\lambda$ belongs to the second group, we have $L(\lambda)=P(\lambda)$ and this forms a semisimple block.
The super weight diagrams of the first form give rise to two blocks, the one containing $L(0|0, +)$ and $L(a|a-1, +)$ for $a>0$ (where we have an even number of dots) and the one containing $L(0|0, -)$ and $L(a|a-1, -)$ for $a>0$ (where we have an odd number of dots).

Looking at the diagrammatics, we can easily write down the socle (resp.~radical) filtration of the nonirreducible projectives.
To make things more clearly, we write the highest weight with the sign instead of the super weight diagram. The translation between these two can be found in the previous paragraph.
\begin{center}
	\setlength{\tabcolsep}{3.9pt}
\begin{tabularx}{\textwidth}{c|c|c|c}
	$P(0|0, \pm)$&$P(1|0, \pm)$&$P(2|1, \pm)$&$P(k|k-1, \pm)\text{ for $k>2$}$\\\hline
	$L(0|0, \pm)$&$L(1|0, \pm)$&$L(2|1, \pm)$&$L(k|k-1, \pm)$\\
	$L(2|1, \pm)$&$L(2|1, \pm)$&$L(0|0, \pm)\,L(1|0, \pm)\,L(3|2, \pm)$&$L(k-1|k-2, \pm)\,L(k+1|k, \pm)$\\
	$L(0|0, \pm)$&$L(1|0, \pm)$&$L(2|1, \pm)$&$L(k|k-1, \pm)$
\end{tabularx}
\end{center}

\subsubsection{Translation functors}
In this section we are going to give formulas for the decomposition of $V^{\otimes d}$ into indecomposable summands. As $V\otimes \_=\bigoplus_{i\in\mathbb{Z}}\theta_i$ decomposes into translation functors, we are only going to describe the decomposition $\theta_i M$ into indecomposable summands for an indecomposable summand $M$ of $V^{\otimes d}$. First of all note that every indecomposable summand in $V^{\otimes d}$ is actually projective or irreducible (as $m=n=1$, see \cref{indecradicalandsoclefiltrationagree} and the comment just after  \cref{classificationofindecomposablesummandsvialayernumberofdelignewdiag}). By \cref{eachBlockHasUniqueIrredSummand} we know that every block contains a unique $L(\lambda, \eps)$ that appears as a direct summand. In every typical block, this is also the corresponding indecomposable projective, and for the two atypical blocks we know that $V^{\otimes 0}=L(0|0, +)$ and $V^{\otimes 1}=L(1|0, -)$, it suffices to consider translation functors for $L(0|0, +)$, $L(1|0, -)$ and indecomposable projectives.

From the weight diagram \begin{tikzpicture}[baseline={([yshift=-.6ex]current bounding box.center)}, scale=0.5]
	\wdiagnoline{w w w v v {$\dots$}}\cups{0 1 d}\rays{i 2 d, i 3, i 4}
\end{tikzpicture} we see that the only applicable local move (see \cref{localmoves}) is given by \tikz[baseline={([yshift=-.6ex]current bounding box.center)}, scale=0.5]{\wdiagnoline{-}\rays{0 0 d}\wdiagnoline{-}}, and thus
\begin{equation*}
	\theta_iL(0|0, +)=\begin{cases}
		L(1|0, -)&\text{if $i=0$,}\\
		0&\text{otherwise,}
	\end{cases}
\end{equation*}
by \cref{projfunctorsonsimplenuclear}.
For \tikz[baseline={([yshift=-.6ex]current bounding box.center)}, scale=0.5]{\wdiagnoline{v w w v v {$\dots$}}\cups{0 1}\rays{i 2 d, i 3, i 4}} we find three applicable local moves namely \begin{tikzpicture}
	[baseline={([yshift=-.65ex]current bounding box.center)}, scale=0.5]\wdiagnoline{-}\rays{0 0 d}\wdiagnoline{-}
	\end{tikzpicture}, \tikz[baseline={([yshift=-.65ex]current bounding box.center)}, scale=0.5]{\wdiagnoline{o x}\caps{0 1}\wdiagnoline{- -}} and  \tikz[outer sep=5pt, inner sep = 5pt, baseline={([yshift=-.65ex]current bounding box.center)}, scale=0.5]{\wdiagnoline{x o}\caps{0 1}\wdiagnoline{- -}}, where the last two are applied at positions $\frac{1}{2}$ and $\frac{3}{2}$.
Again by \cref{projfunctorsonsimplenuclear} we get 
\begin{equation*}
\theta_i L(1|0, -) = \begin{cases}
	L(0|0, +)&\text{if $i=0$},\\
	L(1|1, +)=P(1|1, +)&\text{if $i=-1$},\\
	L(2|0, +)=P(2|0, +)&\text{if $i=1$},\\
	0&\text{otherwise}.
\end{cases}
\end{equation*}
So only the effects of translation functors on indecomposable projectives are left to establish, which can be deduced easily from the diagrammatic description above and the type B analog of \cite[Theorem 4.2]{BS2}.
\begin{align*}
	\intertext{For $(a|a-1)$ and $a>1$ we have}
	\theta_i P(a|a-1, \eps)&=\begin{dcases}
		P(a|a-1, -\eps)&\text{if $i=0$,}\\
		P(a|a, -\eps)^{\oplus 2}&\text{if $i=-a$,}\\
		P(a+1|a-1, -\eps)^{\oplus 2}&\text{if $i=a$,}\\
		P(a+2|a, -\eps)&\text{if $i=a+1$,}\\
		P(a+1|a+1, -\eps)&\text{if $i=-a-1$,}\\
		P(a|a-2, -\eps)&\text{if $i=a-1$,}\\
		P(a-1|a-1, -\eps)&\text{if $i=-a+1$,}\\
		0&\text{otherwise}.
	\end{dcases}\\
	\intertext{For $(0|0)$ we have}
	\theta_i P(0|0, \eps)&=\begin{dcases}
		P(1|0, -\eps)&\text{if $i=0$,}\\
		P(2|2, -\eps)&\text{if $i=-2$,}\\
		P(3|1, -\eps)&\text{if $i=2$,}\\
		0&\text{otherwise}.
	\end{dcases}\\
\intertext{For $(1|0)$ we have}
\theta_i P(1|0, \eps)&=\begin{dcases}
	P(0|0, -\eps)&\text{if $i=0$,}\\
	P(1|1, -\eps)^{\oplus 2}&\text{if $i=-1$,}\\
	P(2|0, -\eps)^{\oplus 2}&\text{if $i=1$,}\\
	P(2|2, -\eps)&\text{if $i=-2$,}\\
	P(3|1, -\eps)&\text{if $i=2$,}\\
	0&\text{otherwise}.
\end{dcases}\\
\intertext{For $(a|b)\neq(0|0)$ and $b\neq a-1$ we have}
\theta_i P(a|b, \eps)&=\begin{dcases}
	P(a|b, -\eps)&\text{if $i=0$ and $a>1$ and $b>0$,}\\
	P(a+1|b, -\eps)&\text{if $i=a$ and $b\neq a$,}\\
	P(a-1|b, -\eps)&\text{if $i=-a+1$ and $a>1$,}\\
	P(a|b-1, -\eps)&\text{if $i=b$,}\\
	P(a|b+1, -\eps)&\text{if $i=-b-1$ and $b+2\neq a$,}\\
	0&\text{otherwise}.
\end{dcases}
\end{align*}

\subsection{The smallest even case: \texorpdfstring{$\Osp[2][2]$}{OSp(2|2)}}

\subsubsection{The irreducible representations of $\osp[2][2]$ and $\Osp[2][2]$}
According to \eqref{simplerootsospevenmleqn} we choose the simple roots $\delta_1-\eps_1$, $\delta_1+\eps_1$ and have $\rho=(0|0)$. Let $\lambda=a\eps_1+b\delta_1\in\lie{h}^*$. By \cref{integraldominanceosp} $\lambda$ is integral dominant if and only if $a,b\in\bbZ$ and either $b>0$ or $a=b=0$. When inducing the irreducible representation $L^{\lie{g}}(a|b)$ of $\osp[2][2]$ to a representation $M$ of $\Osp[2][2]$ we must distinguish two cases. If $a=0$, the representation $M$ decomposes into $L(0|b,+)$ and $L(0|b,-)$. If $a\neq 0$,  the representation $M$ is irreducible and isomorphic to the induced one from $L^{\lie{g}}(-a|b)$ and we denote it by $L((a|b)^G)$. By \cref{irreduciblemodulesforOspeven} these are all irreducible modules, which appear.

\subsubsection{Translating to Khovanov's algebra of type D}

For the study of super weight diagrams we distinguish some cases. First assume that our highest weight is denoted by $(a|b)^G$ with $a>0$ and $a\neq b$. Then the associated flipped weight diagram looks like 
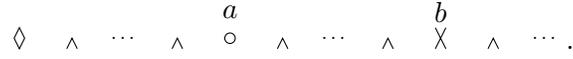
\begin{center}
	\begin{tikzpicture}[scale=0.7]
		\node at (4, 0.5) {$a$};	\node at (8, 0.5) {$b$};
		\wdiagnoline{d w {$\dots$} w o w {$\dots$} w x w {$\dots$}}
	\end{tikzpicture}.
\end{center}
The corresponding super weight diagram is obtained from this by replacing all $\wedge$'s with $\vee$'s. These all give rise to a semisimple block in Khovanov's arc algebra.
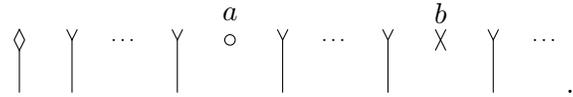
\begin{center}
	\begin{tikzpicture}[scale=0.7]
		\node at (4, 0.5) {$a$};	\node at (8, 0.5) {$b$};
		\wdiagnoline{d v {$\dots$} v o v {$\dots$} v x v {$\dots$}}\rays{i 0, i 1, i 3, i 5, i 7, i 9}
	\end{tikzpicture}.
\end{center}

In the case of $(0|b, \eps)$ with $b\neq 0$ the associated flipped diagram looks like 
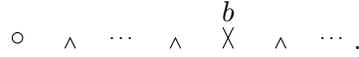
\begin{center}
	\begin{tikzpicture}[scale=0.7]
		\node at (4, 0.5) {$b$};
		\wdiagnoline{o w {$\dots$} w x w {$\dots$}}
	\end{tikzpicture}.
\end{center}
When passing to super weight diagrams, we again change all $\wedge$'s to $\vee$'s except for maybe the leftmost one. This stays a $\wedge$ if $\eps=+$ and gets changed to $\vee$ if $\eps = -$. Similar to the previous case, these all give rise to a semisimple block.
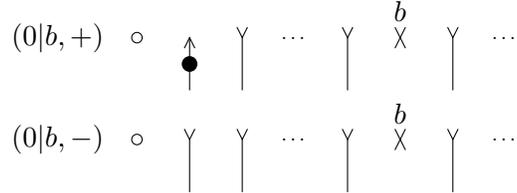
\begin{center}
	\begin{tikzpicture}[scale=0.7]
\node at (5, 0.5) {$b$};\node at (-1.5, 0){$(0|b,+)$};
\wdiagnoline{o w v {$\dots$} v x v {$\dots$}}
\rays{i 1 d, i 2, i 4, i 6}
	\end{tikzpicture}

\begin{tikzpicture}[scale=0.7]
	\node at (5, 0.5) {$b$};\node at (-1.5, 0){$(0|b,-)$};
	\wdiagnoline{o v v {$\dots$} v x v {$\dots$}}
	\rays{i 1, i 2, i 4, i 6}
\end{tikzpicture}
\end{center}

All remaining ones lie in the same block, but we distinguish whether we have $(0|0,\eps)$ or $(a|a)^G$ for $a>0$. The flipped weight diagram for $(0|0)$ is given by 
\begin{center}
	\begin{tikzpicture}[scale=0.7]
		\wdiagnoline{d w w w {$\dots$}}
	\end{tikzpicture}
\end{center}
In case of $(0|0,+)$ the super weight diagram is given by 
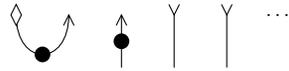
\begin{center}
	\begin{tikzpicture}[scale=0.7]
		\wdiagnoline{d w w v v {$\dots$}}\cups{0 1 d}\rays{i 2 d, i 3, i 4}
	\end{tikzpicture}
\end{center}
and for $(0|0,-)$ we get
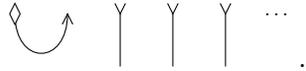
\begin{center}
	\begin{tikzpicture}[scale=0.7]
		\wdiagnoline{d w v v v {$\dots$}}\cups{0 1}\rays{i 2, i 3, i 4}
	\end{tikzpicture}.
\end{center}
For $(a|a)^G$ with $a>0$ the flipped weight diagram is given by 
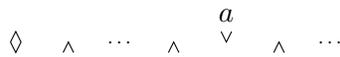
\begin{center}
	\begin{tikzpicture}[scale=0.7]
		\node at (4, 0.5) {$a$};
	\wdiagnoline{d w  {$\dots$} w v w  {$\dots$}}
	\end{tikzpicture}
\end{center}
and the associated super weight diagram is given by 
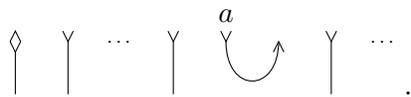
\begin{center}
	\begin{tikzpicture}[scale=0.7]
				\node at (4, 0.5) {$a$};
		\wdiagnoline{d v {$\dots$} v v w v {$\dots$}}\cups{4 5}\rays{i 0, i 1, i 3, i 6}
	\end{tikzpicture}.
\end{center}

Only this last block is nonsemisimple, but looking at the diagrammatics, we can easily establish the socle (reps. radical) filtration of the indecomposable projectives. The following table presents these (we replaced the super weight diagrams by the highest weights):

\begin{equation*}
	\begin{array}{c|c|c}
		P(0|0, \pm)&P((1|1)^G)&P((k|k)^G)\text{ for $k>1$}\\[5pt]\hline&&\\[-10pt]
		L(0|0, \pm)&L((1|1)^G)&L((k|k)^G)\\
		L((1|1)^G)&L(0|0, +)\,L(0|0, -)\,L((2|2)^G)&L((k-1|k-1)^G)\,L((k+1|k+1)^G)\\
		L(0|0, \pm)&L((1|1)^G)&L((k|k)^G)
	\end{array}
\end{equation*}
\begin{rem}
By identifying $P(0|0,+)$ with $P(0|0,+)$, $P(0|0,-)$ with $P(1|0,+)$ and $P((a|a)^G) $with $P(a+1|a, +)$ we see that the principal block of $\Osp[2][2]$ and the principal block of $\Osp[3][2]$ are equivalent.
\end{rem}

\subsubsection{Translation functors}
Similar to the argumentation for $\Osp[3][2]$ (as $m=n=1$) all indecomposable summands of $V^{\otimes d}$ are either irreducible or projective. As $V^{\otimes 0}=L(0|0,+)$ is not projective, we actually know by \cref{eachBlockHasUniqueIrredSummand} and our knowledge of the blocks that this is the only summand which is not projective, so except for $\theta_iL(0|0,+)$ we only need to deal with indecomposable projectives.

For the irreducible $L(0|0,+)$, one easily sees using \cref{projfunctorsonsimplenuclear} that 
\begin{equation*}
	\theta_iL(0|0, +)=\begin{cases}
		L(0|1, +)=P(0|1,+)&\text{if $i=-\frac{1}{2}$,}\\
		0&\text{otherwise.}
	\end{cases}
\end{equation*}
The study of translation functors on projective objects is explicitly written down in \cref{projfunctorsonprojnuclear} and we will just state the results here.
\begin{align*}
	\intertext{For $(1|1)^G$ we have}
	\theta_i P((1|1)^G)&=\begin{dcases}
		P(0|1, +)\oplus P(0|1, -)&\text{if $i=-\frac{1}{2}$,}\\
		P((1|2)^G)^{\oplus 2}&\text{if $i=-\frac{3}{2}$,}\\
		P((2|1)^G)^{\oplus 2}&\text{if $i=\frac{3}{2}$,}\\
		P((2|3)^G)^{\oplus 2}&\text{if $i=-\frac{5}{2}$,}\\
		P((3|2)^G)^{\oplus 2}&\text{if $i=\frac{5}{2}$,}\\
		0&\text{otherwise},
	\end{dcases}\\
	\intertext{For $(0|0)$ we have}
	\theta_i P(0|0, \pm)&=\begin{dcases}
		P(0|1, \pm)^{\oplus 2}&\text{if $i=-\frac{1}{2}$,}\\
		P((1|2)^G)&\text{if $i=-\frac{3}{2}$,}\\
		P((2|1)^G)&\text{if $i=\frac{3}{2}$,}\\
		0&\text{otherwise},
	\end{dcases}\\
\intertext{For $(0|1)$ we have}
\theta_i P(0|1, \pm)&=\begin{dcases}
P(0|0, \pm)^{\oplus 2}&\text{if $i=\frac{1}{2}$,}\\
P(0|2, \pm)&\text{if $i=-\frac{3}{2}$,}\\
0&\text{otherwise}.
\end{dcases}\\
	\intertext{For $a>1$ we have}
	\theta_i P((a|a)^G)&=\begin{dcases}
		P((a|a+1)^G)^{\oplus2}&\text{if $i=-a-\frac{1}{2}$,}\\
		P((a+1|a)^G)^{\oplus 2}&\text{if $i=a+\frac{1}{2}$,}\\
		P((a-1|a)^G)&\text{if $i=-a+\frac{1}{2}$,}\\
		P((a|a-1)^G)&\text{if $i=a-\frac{1}{2}$,}\\
		P((a+1|a+2)^G)&\text{if $i=-a-\frac{3}{2}$,}\\
		P((a+2|a+1)^G)&\text{if $i=a+\frac{3}{2}$,}\\
		0&\text{otherwise},
	\end{dcases}\\
	\theta_i P(0|a, \pm)&=\begin{dcases}
P((1|0)^G)&\text{if $i=\frac{1}{2}$,}\\
P(0|a+1, \pm)&\text{if $i=-a-\frac{1}{2}$,}\\
P(0|a-1, \pm)&\text{if $i=a-\frac{1}{2}$,}\\
0&\text{otherwise},
\end{dcases}\\
\theta_i P((1|a)^G)&=\begin{dcases}
P(0|a,+)\oplus P(0|a, -)&\text{if $i=\frac{1}{2}$,}\\
P((1|a+1)^G)&\text{if $i=-a-\frac{1}{2}$,}\\
P((1|a-1)^G)&\text{if $i=a-\frac{1}{2}$,}\\
0&\text{otherwise}.
\end{dcases}\\
	\intertext{For $a>1$ and $b\neq a$ we have}
	\theta_i P((a|b)^G)&=\begin{dcases}
		P((a-1|b)^G)&\text{if $i=-a+\frac{1}{2}$,}\\
		P((a+1|b)^G)&\text{if $i=a+\frac{1}{2}$ and $a+1\neq b$,}\\
		P((a|b-1)^G)&\text{if $i=b-\frac{1}{2}$,}\\
		P((a|b+1)^G)&\text{if $i=-b-\frac{1}{2}$ and $a-1\neq b$,}\\
		0&\text{otherwise}.
	\end{dcases}
\end{align*}

\bigskip

\bibliography{biblio}

\end{document}